\documentclass[reqno]{amsart}
\usepackage{tkz-euclide}
\usepackage{hyperref}
\usepackage{tikz-3dplot}
\usepackage{amsmath}  
\usepackage{amsfonts}
\usepackage{amssymb} 
\usepackage{mathrsfs} 
\usepackage{graphicx}  
\usepackage{bm}
\usepackage{tikz} 
\usetikzlibrary{calc} 
\usepackage{xcolor}  
\usepackage{tikz}
\usepackage{soul}

\usepackage{tkz-euclide}
\usepackage{hyperref}
\usepackage{amsmath}  
\usepackage{amsfonts}
\usepackage{amssymb} 
\usepackage{tikz-3dplot}
\usepackage{mathrsfs} 
\usepackage{graphicx}  
\usepackage{bm}
\usepackage{tikz} 
\usetikzlibrary{svg.path} 
\usepackage{amssymb}
\usepackage{pgfplots}

\usepackage{pgfplots}
\pgfplotsset{compat = newest}

\usetikzlibrary{calc} 
\usepackage{xcolor}  
\usetikzlibrary{arrows.meta} 
\usepackage{amsthm} 
\usepackage{float}
\usepackage{enumitem}
\usepackage[english]{babel}
\usepackage[font=footnotesize, labelfont=bf]{ caption}
\usepackage{ esint }
\usepackage{stackengine,scalerel,graphicx}

\usepackage{amsthm} 
\usepackage{float}
\usepackage{enumitem}
\usepackage[english]{babel}
\usepackage[font=footnotesize, labelfont=bf]{ caption}
\usepackage{ esint }
\usepackage{stackengine,scalerel,graphicx}
\stackMath

\definecolor{trueblue}{rgb}{0.0, 0.45, 0.81}

\definecolor{forestgreen}{rgb}{0.13, 0.55, 0.13}

\newcommand{\NNN}{\color{black}} 
\newcommand{\MMM}{\color{black}} 
 
\newcommand{\GGG}{\color{black}} 
 
\newcommand{\EEE}{\color{black}}

\newcommand{\AAA}{\color{black}}

\newcommand{\eps}{\varepsilon}

\theoremstyle{plain}
\begingroup
\newtheorem{theorem}{Theorem}[section]
\newtheorem{lemma}[theorem]{Lemma}
\newtheorem{example}[theorem]{Example}
\newtheorem{remark}[theorem]{Remark}
\newtheorem{proposition}[theorem]{Proposition}
\newtheorem{corollary}[theorem]{Corollary}
\endgroup

\theoremstyle{definition}
\begingroup
\newtheorem{definition}[theorem]{Definition}
\endgroup

\renewcommand{\tilde}{\widetilde}

\DeclareMathOperator{\dist}{dist}

\numberwithin{equation}{section}
\newcommand{\N}{\mathbb{N}}

\newcommand{\F}{\mathcal{F}}
\newcommand{\E}{\mathcal{E}}
\newcommand{\R}{\mathbb{R}}

\renewcommand{\S}{\mathbb{S}}
\renewcommand{\L}{\mathcal{L}}
\renewcommand{\H}{\mathcal{H}}

\newcommand{\x}{{\times}}

\newcommand{\Q}{\mathcal{Q}}

\usepackage[normalem]{ulem}

\begin{document}

\title[Derivation of Effective theories for rods with voids]{Derivation of effective theories for thin\\
3D nonlinearly elastic rods with voids}

\author[Manuel Friedrich]{Manuel Friedrich} 
\address[Manuel Friedrich]{Department of Mathematics, FAU Erlangen-N\"urnberg. Cauerstr.~11,
	D-91058 Erlangen, Germany, \& Mathematics M\"{u}nster,  
	University of M\"{u}nster, Einsteinstr.~62, D-48149 M\"{u}nster, Germany}
\email{manuel.friedrich@fau.de}

\author{Leonard Kreutz}
\address[Leonard Kreutz]{Department of Mathematics, School of Computation, Information and Technology, 
Technical University of Munich,
Boltzmannstr. 3, 85748 Garching}
\email{leonard.kreutz@tum.de}

\author{Konstantinos Zemas}
\address[Konstantinos Zemas]{Institute for Analysis and Numerics, University of M\"unster\\
	Einsteinstrasse 62, D-48149 M\"unster, Germany}
\email{konstantinos.zemas@uni-muenster.de}

\begin{abstract}
We derive a dimension-reduction limit for a three-dimensional rod with material voids by means of $\Gamma$-convergence. Hereby, we generalize the results of the purely elastic setting \cite{Mora} to a framework of free discontinuity problems. The effective one-dimensional model features a classical elastic bending-torsion energy, but also accounts for the possibility that the limiting rod can be broken apart into several pieces or folded. The latter phenomenon can occur because of the persistence of voids in the limit, or due to their collapsing into a {discontinuity} of the limiting deformation or its derivative. The main ingredient in the proof is a novel rigidity estimate in varying domains under vanishing curvature regularization, obtained in \cite{KFZ:2021}.

\end{abstract}

\maketitle

\section{Introduction}\label{introduction}
A fundamental question in continuum mechanics is the rigorous derivation of lower dimensional theories for plates, shells, and rods in various energy scaling regimes, starting from three-dimensional models of nonlinear elasticity. Although this question has received considerable attention \cite{Antman1, Antman2}, early derivations were typically based on some a priori \textit{ansatzes}, often leading to theories which were not consistent with each other. The last decades, however, have witnessed a remarkable progress in the rigorous derivation of effective energies for thin elastic objects via variational methods, based on a fundamental cornerstone: the celebrated rigidity estimate by {\sc G.~Friesecke, R.D.~James}, and {\sc S.~M\"uller}   \cite{friesecke2002theorem}. Ever since its appearance, this rigidity result has had numerous applications in dimension-reduction problems providing  a thorough understanding of thin elastic  materials. We refer the reader to the by far nonexhaustive list \cite{Abels-Mora-Muller, Braun-Schmidt, ContiDolzmann, DelgadoSchmidt, MFMK2, lennart1, lennart2,  Mora3, friesecke2002theorem, hierarchy, HornungNeukammVelcic, LewickaLucic, LewickaMahadevanPakzad, Mora4, Mora, Mora2, MoraMullerSchultz, MullerPakzad, NeukammVelcic, schmidt_atomistic_dimension_reduction, schmidt_multilayers} \GGG for references\EEE.  

 On the contrary,  beyond the purely elastic regime, when one is interested in the behavior of materials which might have defects and impurities such as \emph{plastic slips}, \textit{cracks}, or \textit{stress-induced voids}, the   situation is far less-well understood. The goal of this article is to advance the mathematical understanding of thin materials with voids. This  corresponds to the investigation of  energies that are driven by the competition between elastic and surface energies of  perimeter type. Models of this form are gathered under the term  \textit{stress driven rearrangement instabilities} (SDRI), see \cite{BonCha02, BraChaSol07, CrismaleFriedrich, FonFusLeoMil11, GaoNix99, Grin86, Grin93, KhoPio19,  KrePio19, SantiliSchmidt-old,  SieMikVoo04, Spe99} for some mathematical and physical literature on the subject.

We start with a short overview of the literature on dimension reduction in settings beyond elasticity.   Concerning plasticity, we refer the reader, e.g., to \cite{elisa2, elisa1, elisa3, liero-mielke, maggiani}. For models in brittle  fracture   \cite{francfort}, there are several results on brittle plates and shells in a linear setting  \cite{Almi-etal, almitasso,  Baba-dim2, ginsglad}.  In the nonlinear framework, instead, the theory is  mainly restricted to static and evolutionary models in the membrane regime \cite{solo-mem,  Baba-dim, Braides-fonseca}. The only result in a smaller energy regime appears to be   \cite{schmidt2017griffith} for the case of a two-dimensional thin brittle beam. In the limit of vanishing thickness,  the author obtains an effective \textit{Griffith-Euler-Bernoulli} energy defined on the midline of the possibly fractured beam, accounting also for jump discontinuities of the limiting deformation and its derivative. At the core of the arguments in \cite{schmidt2017griffith} lies a suitable generalization of \cite{friesecke2002theorem}, namely a \textit{quantitative piecewise geometric rigidity theorem for SBD functions}   \cite{friedrich_rigidity}. As to date this result is available only in two dimensions, the generalization of dimension-reduction results  
to  three-dimensional fracture is still impeded. Let us however mention that analogous rigidity results in higher dimensions  are available in models for nonsimple materials \cite{friedrich_nonsimple}, where the elastic energy depends additionally on the second gradient of the deformation.
 
In the setting of material voids, a recent result \cite{SantiliSchmidt2022} deals with the derivation of a plate theory in the  bending energy regime. There  the analysis  is limited to voids with restrictive assumptions on their geometry, still allowing to resort to the classical rigidity theorem of \cite{friesecke2002theorem}. Our goal is to derive a related result for thin rods without  restriction on the void geometry. The cornerstone of our approach is a novel rigidity result in the realm of SDRI-models \cite{KFZ:2021}, based on a curvature regularization of the surface term. We now describe our setting in more detail.

We consider a three-dimensional thin rod with reference configuration $\Omega_h  = (0,L)\times h S\subset \R^3$ of thickness $0<h\ll 1$, for a cross section $S \subset \R^2$. For simplicity of the exposition we focus on the case $S= (-\frac{1}{2},\frac{1}{2})^2$, but mention that adaptations to more general geometries are possible. From a variational viewpoint, models describing the formation of material voids in  thin rods fall into the framework of \textit{free discontinuity problems} \cite{Ambrosio-Fusco-Pallara:2000}, and typical energies take the form
\begin{equation}\label{typical_energies_1}
\mathcal F^h_{\rm{el, per}}(v,E):=\int_{\Omega_h\setminus \overline{E}} W(\nabla v)\,\mathrm{d}x+\beta_h\int_{\partial E\cap \Omega_h}\varphi(\nu_E)\,\mathrm{d}\H^2\,.
\end{equation}
Here, $E \subset \Omega_h$ represents the (sufficiently regular) void set within an elastic rod with reference configuration $ \Omega_h \subset \R^3$, and $v$ is the corresponding elastic deformation. The first part of \eqref{typical_energies_1} represents the nonlinear elastic energy with density $W$ (see Section \ref{model_main_result} for details), whereas the second one depends on a parameter $\beta_h>0$ and on
a possibly anisotropic density $\varphi$ evaluated at the outer unit normal $\nu_E$ to $ \partial E\cap\Omega_h$. For purely expository reasons, we will restrict ourselves to the isotropic case, i.e., $\varphi(\cdot)=|\cdot|_{2}$. 

Regarding the energy scaling, at a heuristic level, it is well known that elastic energies of the order $h^4$ correspond to bending and torsion, keeping the midline unstretched, cf. \cite{Mora}. At the same time, the surface area of voids completely separating the rod is 
of order $h^2$. Now, depending on the choice of $\beta_h$, different limiting models can be expected: the case $\beta_h \gg h^2$ will result in a purely elastic rod model, whereas the case $\beta_h \ll h^2$ will result in a model of purely brittle fracture.   The critical regime $\beta_h \sim h^2$ is the most interesting and mathematically most challenging case, for the elastic and surface contributions are of the same order.

Therefore, we set $\beta_h:=h^2$ from now on. Rescaling the energy in \eqref{typical_energies_1} by $h^{-4}$, the natural attempt would be to rigorously derive a corresponding effective one-dimensional theory by means of $\Gamma$-convergence   \cite{Braides:02, DalMaso:93}. However, the presence of a priori unprescribed voids in the model hinders the use of the classical rigidity result of \cite{friesecke2002theorem}. Indeed, the voids might possibly exhibit extremely complicated geometries  such as densely packed thin spikes or microscopically small components  with small surface measure on different length scales,  see Figure~\ref{fig:spikes}.

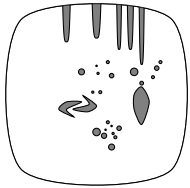
\begin{figure}[htp]
\begin{tikzpicture}
		
\draw[clip] plot [smooth cycle] coordinates {(0,0)(2,0)(2,2)(0,2)};
		
\begin{scope}[shift={(1,1.75)},scale=.25]
\draw[fill=gray, domain=-.5:.5, smooth, variable=\x, xscale=.5,yscale=1] plot ({\x}, {40*\x*\x*\x*\x});
			
\end{scope}
		
\begin{scope}[shift={(.6,1.7)},scale=.2]
\draw[fill=gray, domain=-.8:.8, smooth, variable=\x, xscale=.5,yscale=1] plot ({\x}, {40*\x*\x*\x*\x});
			
\end{scope}
		
\begin{scope}[shift={(1.6,1.4)},scale=.1]
\draw[fill=gray, domain=-1.8:1.8, smooth, variable=\x, xscale=.5,yscale=1] plot ({\x}, {40*\x*\x*\x*\x});
			
\end{scope}
		
\begin{scope}[shift={(1.3,1.6)},scale=.1]
\draw[fill=gray, domain=-1.8:1.8, smooth, variable=\x, xscale=.5,yscale=1] plot ({\x}, {40*\x*\x*\x*\x});
			
\end{scope}

\begin{scope}[shift={(1,.85)},scale=.1,rotate=170]
			
\draw [fill=gray] plot [smooth cycle] coordinates {(0,0) (1,1) (3,1) (1,0) (2,-1)};
			
\end{scope}
		
\begin{scope}[shift={(.5,.8)},scale=.1]
			
\draw [fill=gray] plot [smooth cycle] coordinates {(0,0) (1,1) (3,1) (1,0) (2,-1)};

\draw [fill=gray] plot [smooth cycle] coordinates {(10,2) (11,3) (12,1) (11,-2) (10,0)};

\end{scope}
		
\begin{scope}[shift={(-.2,.8)}]
			
\draw[fill=gray](1.2,.58) circle(.01);
\draw[fill=gray](1.22,.48) circle(.015);
\draw[fill=gray](1,.5) circle(.04);
\draw[fill=gray](1.4,.45) circle(.03);
\draw[fill=gray](1.25,.23) circle(.02);
			
\draw[fill=gray](1.8,.35) circle(.025);
\draw[fill=gray](1.15,.23) circle(.015);

\draw[fill=gray](1.35,.63) circle(.02);
			
\end{scope}

\begin{scope}[shift={(.5,.8)}]

\draw[fill=gray](1,.5) circle(.05);
\draw[fill=gray](1.3,.55) circle(.03);
\draw[fill=gray](1.25,.43) circle(.02);
\draw[fill=gray](1.35,.63) circle(.02);
			
\end{scope}

\begin{scope}[shift={(1.45,1.6)},scale=.14]
\draw[fill=gray, domain=-1.8:1.8, smooth, variable=\x, xscale=.5,yscale=1] plot ({\x}, {40*\x*\x*\x*\x});
			
\end{scope}
		
\draw plot [smooth cycle] coordinates {(0,0)(2,0)(2,2)(0,2)};

\draw[fill=gray](1,.5) circle(.05);
\draw[fill=gray](1.2,.3) circle(.04);
\draw[fill=gray](1.1,.45) circle(.035);
\draw[fill=gray](1.3,.55) circle(.03);
\draw[fill=gray](1.25,.43) circle(.02);
\draw[fill=gray](1.15,.63) circle(.02);
\draw[fill=gray](1.12,.53) circle(.01);
\draw[fill=gray](1.2,.58) circle(.01);
\draw[fill=gray](1.22,.48) circle(.015);

\end{tikzpicture}
\caption{Densely packed thin spikes and microscopically small components leading to loss of rigidity. For simplicity, the figure illustrates and example in dimension two.}
\label{fig:spikes}
\end{figure}

As a remedy, motivated by our work in \cite{KFZ:2021}, we introduce a \textit{curvature regularization}  of the form
\begin{equation}\label{vanishing_curvature_term}
\F^h_{\rm{curv}}(E):= h^2\kappa_h \int_{\partial E\cap \Omega_h}|\bm{A}|^{2}\,\mathrm{d}\H^2\,,
\end{equation}
where  $\bm{A}$ denotes the second fundamental form of $\partial E\cap \Omega_h$ and   $\kappa_h$ satisfies \eqref{rate_1_gamma_h}, which allows in particular for $ \kappa_h\to 0^{+}$ as $h\to 0^{+}$ at a sufficiently slow rate. The presence of such an extra Willmore-type energy  penalization allows  to employ the \textit{piecewise rigidity estimate} \cite[Theorem 2.1]{KFZ:2021} in the analysis.  It   is a \textit{singular perturbation} for the void set $E$ and not for the deformation $v$, i.e., no higher-order gradient of $v$ is involved in the model.  We   refer the reader to our recent work \cite{KFZ:2022}, where a related discrete model is studied and an additional explanation for the presence of a microscopic analogue of the term in \eqref{vanishing_curvature_term} is given, see \cite[Subsection 2.5]{KFZ:2022}. We also mention that  curvature regularizations are widely used in the mathematical and physical literature of  SDRI models, including the description of  elastically stressed thin films or  material voids, see \cite{AnGurt,  FonFusLeoMor14, FonFusLeoMor15, GurtJabb, Herr,   voigt,  SieMikVoo04}. \NNN In spite of possible modeling relevance, we emphasize that we include the curvature contribution in our model only for mathematical reasons as a regularization term. In particular, it does not affect the effective limiting problem. \EEE

The total energy of a pair $(v,E)$ is then given by the sum of the two terms in \eqref{typical_energies_1} and \eqref{vanishing_curvature_term}, i.e.,  
\begin{equation}\label{typical_full_energy}
\mathcal F^h(v,E)=\mathcal F^h_{\rm{el, per}}(v,E)+\mathcal F^h_{\rm{curv}}(E)\,.
\end{equation}
(We set  $\beta_h = h^2$ and  $\varphi(\nu)\equiv 1$ for all $\nu\in \S^2$.) The main result of this contribution is then Theorem \ref{main_gamma_convergence_thm}, where we show that the rescaled energies $(h^{-4}\F^h(\cdot,\cdot))_{h>0}$ $\Gamma$-converge (in an appropriate topology) to an effective one-dimensional functional that takes the form 
\begin{equation}\label{limiting_energy_intro}
\frac{1}{2}\int_{(0,L)\setminus I}\Q_2(R^TR')\,\mathrm{d}x_1+\H^0(\partial I \cap (0,L))+2\H^0\big((J_y\cup J_R)\setminus \partial I\big)\,. 
\end{equation}
Here, $I\subset (0,L)$ denotes a union of finitely many intervals in $(0,L)$   and represents the void part in the limiting one-dimensional rod.  The deformation $y\colon(0,L) \to  \R^3$ is an isometric piecewise $W^{2,2}$-regular curve that represents the deformed rod. The rotation field $R\colon(0,L) \to SO(3)$ whose first column is the velocity $y'$ represents the \textit{Frenet frame} with respect to $y$. The  elastic part of the limiting energy corresponds to the one identified in the purely elastic setting \cite{Mora}: it is quadratic in terms of the skew-symmetric tensor $R^TR'$ which encodes the information for the \textit{curvature and torsion} of the curve $y$. The associated quadratic form $\Q_2$ is defined through the quadratic form $D^2W(I)$ of linearized elasticity via a suitable minimization problem, see \eqref{def_Q_2} for details. The second term in \eqref{limiting_energy_intro} accounts for the presence of voids by counting their endpoints ($\H^0$ stands for the counting measure in $\R$). The last term therein takes into account the fact that, in the limit, voids might collapse exactly into discontinuity points of the limiting $y$ or its Frenet frame $R$, corresponding to cracks or kinks of the limiting rod, respectively. Accordingly, these discontinuity points should be counted twice in the energy.

Let us highlight the relation to the result in \cite{SantiliSchmidt2022}, where a similar model of \textit{Blake-Zisserman type} (cf. \cite{Blake-Zisserman, Carriero-Leaci-BZ}) for elastic plates with voids in  the Kirchhoff bending energy regime is obtained. First, in \cite{SantiliSchmidt2022} plates are considered, whereas we treat the case of rods.  We decided to present our approach based on the model \eqref{typical_full_energy}  first  for a dimension reduction from 3D-to-1D to avoid some technicalities arising in the 3D-to-2D analysis. The latter, however, can be performed as well, and is the subject of  a forthcoming work,  both in the  Kirchhoff \cite{friesecke2002theorem} and the {von K{\'a}rm{\'a}n} \cite{hierarchy} regime. The fundamental  difference between our work and \cite{SantiliSchmidt2022} concerns the assumptions on the void set. Whereas we allow for voids with general geometry employing a mild curvature regularization, \cite{SantiliSchmidt2022} is based on  specific restrictive assumptions on the void geometry, namely the  so-called \textit{$\psi$-minimal droplet assumption}, cf.~\cite[Equation (6)]{SantiliSchmidt2022}. This can be interpreted as an $L^{\infty}$-diverging bound on the curvature of the boundary of the voids. In our setting, the curvature regularization term in \eqref{typical_full_energy} can be thought of as imposing an $L^2$-diverging bound on the curvature: firstly, this allows the void set to concentrate at arbitrarily small scales (independently  of $h$) and, secondly,  allows the boundary of the void set to consist of a diverging (with $h$) number of connected components, see Example~\ref{example:scaleandcardinality}. Our more general model comes at the expense of the necessity of more sophisticated geometric rigidity results \cite{KFZ:2021} compared to \cite{friesecke2002theorem}.  

\subsection{Organization of the paper and proof strategy}\label{plan_of_paper} 
The paper is organized as follows. In Section \ref{model_main_result} we introduce  our model  and state the main compactness and $\Gamma$-convergence results, i.e., Theorems \ref{compactness_thm} and \ref{main_gamma_convergence_thm}, respectively.  \NNN We also include some comments on the limiting model and discuss possible boundary value problems. \EEE Section \ref{preparatory_modifications} contains the core of our paper by deriving a \NNN blockwise \EEE Sobolev approximation of sequences $(v_h,E_h)_{h>0}$ with 
$$\sup_{h>0}h^{-4}\mathcal{F}^h(v_h,E_h)<+\infty\,.$$
We perform a careful  enlargement of the voids $E_h$ according to \cite[Proposition 2.8]{KFZ:2021}, as well as an appropriate modification of the deformations. This is the content of Propositions~\ref{Lipschitz_replacement} and \ref{prop: 2nd main} stated  \GGG at the beginning of Section \ref{preparatory_modifications}\EEE, where we modify the deformations $v_h$ and their gradients $\nabla v_h$ on a small part of the rod, such that  the new deformations  are actually Sobolev in \NNN big blocks \EEE of the rod $\Omega_h$ with a good control on their elastic energy. Moreover, the modification is done in such a way that the jump height of the new sequence along the entire rod is suitably controlled \GGG as well as \EEE producing the correct jump points of the limiting deformation and its curvature-torsion tensor. 

The main technical tools to obtain these modifications are the \textit{piecewise rigidity estimate}  \cite[Theorem 2.1]{KFZ:2021} and a \textit{Korn inequality for functions with small jump set} \cite{Cagnetti-Chambolle-Scardia}, applied on long cuboids that partition $\Omega_h$.   More precisely, splitting the rod $\Omega_h$ into $\sim h^{-1}$ many long cuboids of length~$\sim h$, we focus on those cuboids where the perimeter of the enlarged void is locally not large enough to produce macroscopic fracture, see \eqref{good_cuboids}--\eqref{eq: ugly}. In these cuboids, by means of \textit{isoperimetric arguments}  and our piecewise rigidity estimate, we obtain large in volume sets in which slight modifications of $v_h$ are approximately $W^{1,2}$-rigid in terms of the local elastic energy. As we believe that the isoperimetric inequalities may be of independent interest, we state and prove them in arbitrary space dimension, see  Subsection~\ref{sec: isoperimetric}.

Although \cite[Theorem 2.1]{KFZ:2021} provides an optimal estimate only in terms of the symmetrized gradient, a use of the Korn-Poincar{\' e} inequality for functions with small jump set \cite{Cagnetti-Chambolle-Scardia} allows us to upgrade our estimate to the full gradient in all but finitely many cuboids. This leads to an optimal estimate for the difference between the rigid motions in terms of the local elastic energy, again in all but finitely many adjacent cuboids, see Proposition \ref{summary_estimates_proposition} and Corollary \ref{difference_rigid-motions} in Subsection \ref{sec: locest}. In Subsection \ref{sec: global_constructions_proofs}, we eventually construct the global \NNN blockwise \EEE Sobolev modifications and give the proofs of Propositions \ref{Lipschitz_replacement} and \ref{prop: 2nd main}.
  
Based on these preparations, the rest of the paper is more standard and the results  of the elastic case \cite{Mora}  can be employed directly.  Section \ref{compactness} is devoted to the proof of  compactness (Theorem~\ref{compactness_thm}) and Section \ref{gamma_liminf} to the proof of the $\Gamma$-liminf inequality of Theorem \ref{main_gamma_convergence_thm}. The proof of the $\Gamma$-limsup inequality is given in Section \ref{gamma_limsup}  by  exhibiting a recovery sequence $(v_h,E_h)_{h>0}$. Here, we use  the corresponding recovery sequence from \cite{Mora} for the deformations, and we construct the voids $E_h$ with planar interfaces in order to approximate the one-dimensional limiting void sets and the jump points.

We also remark that, from a technical viewpoint, our proof strategy provides --  to our view -- a simplified alternative to obtain the $\Gamma$-liminf inequality compared to the methods used in \cite{schmidt2017griffith}, which were based on delicate interpolation and difference quotients estimates. The latter were dictated by the fact that, in the same fashion  as the result \cite{KFZ:2021}, the two-dimensional piecewise rigidity estimate in $SBD$ \cite[Theorem 2.1]{friedrich_rigidity}   used in \cite{schmidt2017griffith}  provides  an optimal estimate in terms of the elastic energy \emph{only} for symmetrized gradients. Therefore, the standard difference quotients method used in \cite{Mora} was not directly applicable. Our method instead leads to a \NNN blockwise \EEE Sobolev replacement with the aid of the \textit{Korn inequality  for functions with small jump set} \cite{Cagnetti-Chambolle-Scardia}. This  actually enables us to use directly the results of \cite{Mora}. We emphasize that in this regard our approach is  general: given any kind of geometric rigidity result delivering a sharp control for symmetrized gradients, e.g.\ also the result in \cite{friedrich_nonsimple}, our techniques carry directly over and allow to work  with \NNN blockwise \EEE Sobolev replacements. \EEE  

\subsection{Notation}\label{Notation}
We close the introduction with some basic notation. Given $U\subset \R^3$ open, we denote by $\mathcal{P}(U)$ the collection of subsets of finite perimeter in $U$. Given $E\in \mathcal{P}(U)$, for any $s\in[0,1]$ we denote by $E^s$ the set of points with 3-dimensional density $s$ with respect to $E$, and by $\partial^*E$ its essential boundary, see \cite[Definition~3.60]{Ambrosio-Fusco-Pallara:2000}. The family of sets of finite perimeter on a one-dimensional interval $(0,L)$ will be simply denoted by $\mathcal{P}(0,L)$. We also denote by $\mathcal{A}_{\rm reg}(U)$ the collection of all open sets $E \subset U$ such that $\partial E \cap U $ is a two-dimensional $C^2$-surface in $\mathbb{R}^3$. Surfaces and functions of $C^2$-regularity will be called \NNN $C^2$-regular \EEE in the following. For $E \in \mathcal{A}_{\rm reg}(U)$ we denote by $\bm{A}$ the second fundamental form of $\partial E\cap U$, i.e., $|\bm{A}| = \sqrt{\kappa_1^2 + \kappa_2^2}$, where $\kappa_1$ and $\kappa_2$ are the corresponding principal curvatures.  By $\nu_{E}$ we indicate the outer unit normal to  $\partial E\cap U$.  For every $a,b\in \R$ we denote   $a\wedge b:=\min\{a,b\}$ and $a\vee b:=\max\{a,b\}\,.$

For $p\in [1,\infty]$ and $d, k \in \N$ we denote by $L^p(U;\R^d)$ and $W^{k,p}(U;\R^d)$ the standard Lebesgue and Sobolev spaces, respectively.  Partial derivatives of a function $f\colon U \to \R^3$ will be denoted by $(f_{,i})_{i=1,2,3}$. Given measurable sets $A, B$, we write $\chi_A$ for the characteristic function of $A$, $A \subset \subset B$ if $\overline{A} \subset B$, and $\mathrm{dist}_{\H}(A,B)$ for the Hausdorff distance between $A$ and $B$. For $d,k\in \N$, we denote by $\mathcal{L}^d$ and $\mathcal{H}^{k}$ the $d$-dimensional Lebesgue measure and the $k$-dimensional Hausdorff measure, respectively. 

We set $\R_+ := [0,+\infty)$. By ${\rm id}$ we   denote the identity mapping on $\R^3$ and by ${\rm Id} \in \R^{3\times 3}$  the identity matrix.  For each $F \in \R^{3 \times 3}$ we let $${\rm sym}(F) := \frac{1}{2}\left(F+F^T\right)\,,$$
and we also define 
 $$SO(3) := \lbrace F\in\R^{3 \times 3}\colon F^TF = {\rm Id}, \, \det F = 1\rbrace\,.$$ 
Moreover, we denote by $\R^{3\times 3}_{\rm sym}$ and $\R^{3 \times 3}_{\rm skew}$ the space of symmetric and skew-symmetric matrices, respectively. We further write $\mathbb{S}^2:= \lbrace \nu \in \R^3\colon \, |\nu|=1\rbrace$. For $\sigma>0$, we denote by $T_\sigma$ the linear transformation in $\R^3$ with matrix representation being given by 
\begin{equation}\label{anisotropic_dilation}
T_\sigma:=\mathrm{diag}(1,\sigma,\sigma)
\end{equation} 
with respect to the canonical basis $\{e_1,e_2,e_3\}$. 
 
We use standard notation for $SBV$-functions, cf. \cite[Chapter~4]{Ambrosio-Fusco-Pallara:2000} for the definition and a detailed presentation of the properties of this space. In particular, for a function $u\in SBV(U;\R^d)$,  we write $\nabla u$ for the approximate gradient, $J_u$ for  the jump set, and $u^{\pm}$ for the one-sided traces on $J_u$. We also use the notation $[u]:=u^+-u^-$ for the corresponding jump height. We consider the space
\begin{equation*}
SBV^2(U;\R^d):=\Big\{u\in SBV(U;\R^d)\colon \int_{U}|\nabla u|^2\,\mathrm{d}x+\H^{d-1}(J_u\cap U)<+\infty\Big\}\,.
\end{equation*}
In dimension one, given $a<b\in\R$ and $d\in \N$, the space $SBV^2((a,b);\R^d)$ coincides with the space \NNN $P$-$W^{1,2}((a,b);\R^d)$ of piecewise $W^{1,2}$\EEE-Sobolev functions, which  consists of those $Y\in L^1((a,b); \R^d)$ for which there exists a partition  
$${a =: t_0<t_1<\dots<t_m<t_{m+1} := b \  \ \ \text{  such that } \ \ \ \NNN Y\in W^{1,2}\EEE((t_{i-1},t_i);\R^d) \ \forall i=1,\dots,m+1\,.}$$ 
The jump set of $Y$ is precisely the minimal set $J_Y = \{t_1, \dots,t_m\}$ with the above property. By taking an appropriate representative, we may then assume that $Y$ is uniformly continuous on $\{(t_{i-1}, t_i)\}_{i=1,\dots,m+1}$ and $Y(t_i^{\pm})$ are the limits of $Y(t)$ as $t\to t_i^{\pm}$. 

Analogously, for $k\in \N$, we define $P$-$\NNN W^{k,2}\EEE((a,b);\R^d)$ as the space of $Y\in L^1((a,b);\R^d)$ for which there exists $\{a =:t_0 < t_1<\dots< t_{m+1}:= b\}$ such that $Y\in \NNN W^{k,2}\EEE((t_{i-1}, t_i); \R^d)\ \forall i=1,\dots,m+1$. For $Y$ in this space, the
minimal set $\{t_1,\dots,t_m\}$ with the above property is $\bigcup_{l=0}^{k-1} J_{Y^{(l)}}$\NNN, where $Y^{(l)}$ denotes the $l$-th derivate of $Y$.\EEE

\section{The model and the main results}\label{model_main_result}

\textbf{Model in the reference domain:} 
We denote the reference configuration of the thin rod by  
\begin{align}\label{eq: Omegah}
\Omega_{h}:=(0,L)\times (-\tfrac{h}{2},\tfrac{h}{2})^2\subset \R^3\,,
\end{align}
where $L>0$ is a macroscopic parameter describing the length of its midline, and $0<h\ll L$ denotes its infinitesimal thickness.   For a fixed large constant $M\gg 1$, the set of \emph{admissible pairs} of function and set is given by 
\begin{equation}\label{initial_admissible_configurations}
\mathcal{A}_h:=\left\{(v,E)\colon\ \ E\in \mathcal{A}_{\rm reg}(\Omega_h),\ v\in \NNN W^{1,2}\EEE(\Omega_h\setminus \overline{E};\R^3)\,,\ v|_E\equiv\mathrm{id}\,,\ \|v\|_{L^\infty(\Omega_h)} \le   M\right\}\,.
\end{equation}
The third condition in \eqref{initial_admissible_configurations} is for definiteness only. The last one is merely of technical nature to ensure compactness. At the same time, it is also justified from a physical point of view, for it corresponds to the assumption that the material under investigation is confined in a bounded region. For each pair $(v,E) \in \mathcal{A}_h$, \EEE   we consider the energy
\begin{align}\label{initial_energy}
\F^{h}(v,E):= 
\int_{\Omega_h\setminus \overline{E}} W(\nabla v)\,\mathrm{d}x+h^2  \H^2(\partial E\cap \Omega_h)+  h^2\kappa_h  \int_{\partial E\cap \Omega_h}|\bm{A}|^2\,\mathrm{d}\mathcal{H}^2\,.
\end{align}
Here, the first and second term correspond to the \emph{elastic} and the \emph{surface energy} of perimeter type, while the third term is a \emph{curvature regularization} of Willmore-type, where $\bm A$ denotes the second fundamental form of $\partial  E\cap \Omega_h$ and $\kappa_h$ is a suitable parameter. The factor $h^2$ in front of the surface terms ensures that the elastic and the surface energy are of same order for our choice of  the bending regime, where the elastic energy per unit volume is of order $h^2$. We refer to the introduction for more details.

The function $W \colon \mathbb{R}^{3\times 3}\to\R_+$ in \eqref{initial_energy} represents the \emph{stored elastic energy density}, satisfying the usual assumptions of nonlinear elasticity. Altogether, we suppose that $W\in C^0(\R^{3\times 3};  \R_+ )$ satisfies 
\begin{align}\label{eq: nonlinear energy}
\begin{split}
\rm{(i)} & \ \  \text{Frame indifference: $W(RF) = W(F)$ for all $R \in SO(3)$ and $F\in \mathbb{R}^{3\times 3}$}\,,\\ 
\rm{(ii)} & \ \ \text{Single energy-well structure:}\ \{W=0\} \equiv SO(3)\,,\\
\rm{(iii)} & \ \ \text{Regularity:}  \ \ \text{$W$ is $ C^2 $  in a neighborhood of $SO(3)$}\,,\\
\rm{(iv)} & \ \ \text{Coercivity:}  \ \  \text{There exists $c>0$ such that for all $F \in \mathbb{R}^{3\times 3}$ it holds that} \\ 
& \quad   \quad \quad  \quad \quad \quad \, W(F) \geq c\, \mathrm{dist}^2(F,SO(3))\,. 
\end{split}
\end{align} 
Our choice of an isotropic surface energy is for simplicity only and can be generalized, as we briefly explain in Remark \ref{choice_of_model} below. As for the parameter $\kappa_h>0$ in the curvature regularization, we require  
\begin{equation}\label{rate_1_gamma_h}
\kappa_h  h^{\AAA-52/25} \to +\infty \quad \text{ as $h\to 0$}\,. 
\end{equation}
\NNN We point out that \eqref{rate_1_gamma_h} is a technical assumption and chosen for simplicity rather than optimality. \EEE Its role is connected to the application of suitable rigidity results  \cite{KFZ:2021, Cagnetti-Chambolle-Scardia} and will become apparent along the proof, see in particular \eqref{parameters_for_uniform_bounds}. \EEE

\textbf{Rescaling of the model:} As it is customary in dimension-reduction problems, we perform a change of variables to a fixed reference domain: recalling \eqref{anisotropic_dilation}, we rescale our variables and set 
\begin{equation}\label{order_1_domain}
\Omega:=\Omega_1\,, \quad V:=\{x\in \Omega: (x_1,hx_2,hx_3)\in E\}=T_{1/h}(E)\,.
\end{equation}
We also rescale the deformations accordingly, by defining  $y\colon\Omega\to \R^3$ via
\begin{equation}\label{from_v_to_y}
y(x):=y(x_1,x_2,x_3):=v(x_1,hx_2,hx_3)\,.
\end{equation}
We rescale  the energy by the factor $h^4$ and set 
\begin{equation}\label{rescaled_energy}
\mathcal{E}^{h}(y,V):=h^{-4}\F^{h}(v,E)\,,
\end{equation}
where the pair $(y,V)$ is related to $(v, E)$ via \eqref{order_1_domain}--\eqref{from_v_to_y}. Here,  one  factor $h^2$  corresponds to the change of volume  and the other factor $h^2$  corresponds to the average elastic energy per unit volume, reflecting our choice of the bending energy regime.

For the corresponding rescaled gradients, we will use the notation
\begin{equation}\label{rescaled_deformation_gradient}
\nabla_hy(x):=\Big(\partial_1 y,\frac{1}{h}\partial_2 y, \frac{1}{h}\partial_3y\Big)(x)=\nabla v(x_1,hx_2,hx_3)\,. 
\end{equation}
Therefore,  by a change of variables  we find  
\begin{align}\label{eq: newenergy}
\mathcal{E}^{h}(y,V) = h^{-2}\int_{\Omega\setminus \overline{V}} W(\nabla_hy(x))\,\mathrm{d}x+\int_{\partial V\cap \Omega}\big|\big(\nu^1_{V}(z), h^{-1}\nu^2_{V}(z), h^{-1}\nu^3_{V}(z)\big)\big|\, \mathrm{d}\H^2(z) + \mathcal{E}^h_{\rm  curv}(V) \,,
\end{align}
where $\nu_{V}(z):=\big(\nu^1_{V}(z),\nu^2_{V}(z), \nu^3_{V}(z)\big)$ denotes the outer unit normal to $\partial V\cap \Omega$ at the point $z$. (For the rescaling of the perimeter 
part, one can test with smooth functions and use the divergence theorem.) \EEE Here, the term $\mathcal{E}^h_{\rm  curv}(V)$ denotes the curvature contribution for the rescaled set $V$, for which we refrain from performing the change of variables explicitly.

In view of \eqref{anisotropic_dilation} and  \eqref{initial_admissible_configurations}, the space of rescaled admissible pairs (deformations-voids) is given by
\begin{equation}\label{admissible_configurations_h_level}
\hat{\mathcal{A}}_h:=\{(y,V)\colon V\in \mathcal{A}_{\mathrm{reg}}(\Omega)\,,\ y\in \NNN W^{1,2}\EEE(\Omega\setminus \overline{V}; \R^3)\,, \ y|_{V}\equiv T_h(\mathrm{id})\,, \ \|y\|_{L^\infty(\Omega)}\leq M\}\,.
\end{equation}

\textbf{Limiting model:}  The  limiting energy will be defined on the space
\begin{equation}\label{limiting_admissible_pairs}
\begin{split}
\hspace{-0.5em}\mathcal{A}&:=\big\{\big((y\vert \, d_2\vert \, d_3),I\big) \colon  (y\vert \, d_2 \vert \,  d_3)\in SBV^2_{\mathrm{isom}}(0,L)\,,\,     y|_I(x_1) \equiv x_1 
\,,\  (y_{,1}\vert \, d_2\vert \, d_3)|_{I} \equiv \mathrm{Id}\,, \\
&\quad  \quad \quad  \quad \quad  \quad \quad  \quad \quad  \quad \quad  \quad \quad  \quad \quad  \quad \quad \quad  \quad \quad  \ \ \ \ \ \  \|y\|_{L^\infty(\Omega)}\leq M\,, \,  I\in \mathcal{P}(0,L)
\big\}\,,
\end{split}
\end{equation}
where, recalling the definition of \NNN $P$-$W^{k,2}$ \EEE in Subsection \ref{Notation}, we define 
\begin{equation}\label{SBV_2_isom}
SBV^2_{\mathrm{isom}}(0,L):=\begin{array}{lr} \Big\{
 ({y}\vert \,  {d_2}\vert \, {d_3})\in \NNN\big(P\text{-}W^{2,2}\times P\text{-}W^{1,2}\times P\text{-}W^{1,2}\big)\EEE\big((0,L);\R^{3\times 3}\big)\, \text{with }\\[5pt]
\ \   R:=({y}_{,1}\vert \, {d_2}\vert \, {d_3})\in SO(3) \text{ a.e. in } (0,L)\Big\}\,.
\end{array}
\end{equation}
By a slight abuse of notation, for triplets $(\bar{y}\vert\, \bar{d}_2\vert\,\bar{d}_3)\colon \Omega  \to \R^{3\times 3} $  we will also use the notation $(\bar{y}\vert\,  \bar{d}_2\vert\, \bar{d}_3) \in SBV^2_{\mathrm{isom}}(0,L)$ if and only if
\begin{align}\label{eq: convention1}
(\bar{y}\vert\, \bar{d}_2\vert\,\bar{d}_3)(x)=(y\vert \, d_2\vert \, d_3)(x_1) \ \text{for all $x \in \Omega$, for some $(y\vert \, d_2\vert \, d_3) \in SBV^2_{\mathrm{isom}}(0,L)$}\,. 
\end{align} 
In a similar fashion, we will write
\begin{align}\label{eq: convention2}
\bar R(x) := ( y_{,1}\vert \, d_2\vert \, d_3)( x_1) \  \text{for all $x \in \Omega$}\,. 
\end{align}
With these definitions, for each $((y\vert \, d_2\vert \, d_3),I)\in \mathcal{A}$, the limiting one-dimensional energy of Blake-Zisserman type (cf. \cite{Blake-Zisserman, Carriero-Leaci-BZ, SantiliSchmidt2022} for analogous models in different settings) is defined as
\begin{equation}\label{limiting_one_dimensional_energy}
\mathcal{E}^{0}\big((y|\, d_2|\, d_3),I\big):=
\frac{1}{2}\int_{(0,L)\setminus I}\Q_2({R}^T{R}_{,1})\,\mathrm{d}x_1+\H^0\big(\partial^*I\cap (0,L)\big)+2\H^0\big((J_{{y}}\cup J_{{R}})\setminus \partial^*I\big)\,. 
\end{equation}
Here, $R$ is defined as in \eqref{SBV_2_isom}, and the quadratic form $\Q_2\colon\R^{3\times 3}_{\mathrm{skew}}\to \R_+$  is defined through a minimization problem as  
\begin{equation}\label{def_Q_2}
\Q_2(A):= \min_{a\in W^{1,2}\left(\left(-\frac{1}{2},\frac{1}{2}\right)^2;\R^3\right)} \int_{\left(-\frac{1}{2},\frac{1}{2}\right)^2}\Q_3\left( A \begin{pmatrix}
0 \\
x_{2} \\
x_{3}
\end{pmatrix}\Bigg\vert \, \alpha_{,2}\Bigg \vert \, \alpha_{,3}\right)\,\mathrm{d}x_2\, \mathrm{d}x_3
\end{equation}
 for all $A \in \R^{3 \times 3}_{\rm skew}$, where, for every $G\in \R^{3\times 3}$,
\begin{equation}\label{linear_elasticity_quadratic_form}
\Q_3(G):= D^2 W( {\rm Id}) [G,G] 
\end{equation}
is the corresponding quadratic form of linearized elasticity. Note that, as $R$ belongs to $SO(3)$, $R^TR_{,1}$ is skew symmetric, and thus the elastic energy in \eqref{limiting_one_dimensional_energy} is well defined.  Moreover, due to \eqref{eq: nonlinear energy}, $\mathcal{Q}_3$ vanishes on $\R^{3\times 3}_{\mathrm{skew}}$ and is strictly positive definite on $\R^{3\times 3}_{\mathrm{sym}}$.

As mentioned also in the introduction, the limiting one-dimensional model features the classical bending-torsion term derived in \cite{Mora} and two surface terms related to the presence of voids. The first part corresponds to the energy contribution of the limiting void $I$, whereas the second part  is associated to discontinuities  or kinks of the deformation, represented by $J_y$ and $J_R$, respectively. This term is due to the fact that voids may collapse to single points and hence enters the  energy with a factor $2$, see Figure~\ref{fig:dimred}.

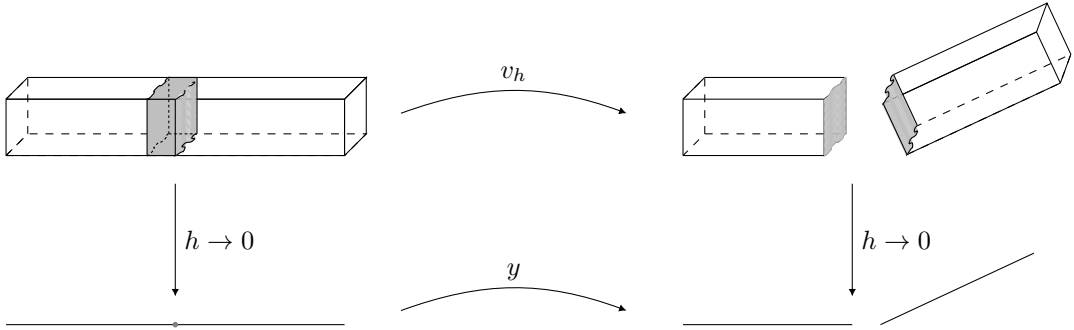
\begin{figure}[htp]
\begin{tikzpicture}[scale=.75]

\tikzset{>={Latex[width=1mm,length=1mm]}};

\draw[->] (7,.75) to[out=20,in=160] (11,.75);
 
\draw(9,1.75) node[anchor=north]{$v_h$};

\draw[->] (7,-2.75) to[out=20,in=160] (11,-2.75);
  
\draw(9,-1.75) node[anchor=north]{$y$};

\draw[dashed](0,1,0)--++(0,0,-1)--++(0,-1,0)--++(0,0,1);

\draw[dashed](0,0,0)--++(6,0,0)--++(0,0,-1)--++(-6,0,0)--++(0,0,1);

\draw(6,0,0)--++(0,1,0)--++(0,0,-1)--++(0,-1,0)--++(0,0,1);

\draw(0,1,0)--++(6,0,0)--++(0,0,-1)--++(-6,0,0)--++(0,0,1);

\draw(0,0,0)--++(0,1,0)--++(6,0,0)--++(0,-1,0)--++(-6,0,0);

\pgfmathsetmacro{\thick}{50};

\foreach \j in {0,...,\thick}{
\draw[gray!50!white,opacity=.2](3,0,0) plot[ smooth, tension=2] coordinates { (3,\j/\thick,0) (3.03,\j/\thick,-.1) (3.05,\j/\thick,-.3) (2.95,\j/\thick,-.5)(3.05,\j/\thick,-.7)(2.97,\j/\thick,-.85)(3,\j/\thick,-1)};

\draw[gray!50!white,opacity=.11](3,0,0) plot[ smooth, tension=2] coordinates {(2.5,\j/\thick,0)(2.52,\j/\thick,-.2)(2.47,\j/\thick,-.5)(2.52,\j/\thick,-.7)(2.49,\j/\thick,-.9)(2.5,\j/\thick,-1)};
}

\draw[black,opacity=.5,fill opacity=.2,fill=gray!50!white](3,0,0)--++(0,1,0)--++(-.5,0,0)--++(0,-1,0)--++(.5,0,0);
\draw[black,opacity=.5](3,0,-1)--++(0,1,0);

\foreach \j in {0,\thick}{
\draw[black,opacity=.5](3,0,0) plot[ smooth, tension=2] coordinates { (3,\j/\thick,0) (3.03,\j/\thick,-.1) (3.05,\j/\thick,-.3) (2.95,\j/\thick,-.5)(3.05,\j/\thick,-.7)(2.97,\j/\thick,-.85)(3,\j/\thick,-1)};
}

\draw[black,opacity=.5,fill opacity=.2,fill=gray!50!white](3,1,0)--++(-.5,0,0)-- plot[ smooth, tension=2] coordinates {(2.5,1,0)(2.52,1,-.2)(2.47,1,-.5)(2.52,1,-.7)(2.49,1,-.9)(2.5,1,-1)}--++(.5,0,0);

\draw[black,opacity=.5,dash pattern = on 1 pt off 1 pt]  plot[ smooth, tension=2] coordinates {(2.5,0,0)(2.52,0,-.2)(2.47,0,-.5)(2.52,0,-.7)(2.49,0,-.9)(2.5,0,-1)};

\draw[black,opacity=.5,dash pattern = on 1 pt off 1 pt](2.5,0,-1)--++(0,1,0);

\draw[black,opacity=.5,dash pattern = on 1 pt off 1 pt](2.5,0,-1)--++(.5,0,0);

\draw[->](3,-.5)--++(0,-2);

\draw(3,-1.5) node[anchor=west]{$h\to 0$};

\draw(0,-3)--++(6,0);
\draw[fill=gray,gray] (3,-3)circle(.04);

\begin{scope}[shift={(12,0)}]

\draw(0,-3)--++(3,0);

\draw(0,-3)++(3.5,0)--++(25:3);

\draw[->](3,-.5)--++(0,-2);
\draw(3,-1.5) node[anchor=west]{$h\to 0$};

\draw[dashed](0,1,0)--++(0,0,-1)--++(0,-1,0)--++(0,0,1);

\draw[dashed](0,0,-1)--++(2.5,0,0);


\draw(0,1,0)--++(2.5,0,0);
\draw(0,1,0)--++(0,0,-1);
\draw(0,1,-1)--++(2.5,0,0);

\draw(0,0,0)--++(2.5,0,0);
\draw(0,0,0)--++(0,1,0);

\draw[black,opacity=.5] (2.5,0,0)--++(0,1,0);

\draw[black,opacity=.5] (2.5,0,-1)--++(0,1,0);

\draw[black,opacity=.5] plot coordinates {(2.5,1,0)(2.52,1,-.2)(2.47,1,-.5)(2.52,1,-.7)(2.49,1,-.9)(2.5,1,-1)};

\draw[black,opacity=.5]  plot[ smooth, tension=2] coordinates {(2.5,0,0)(2.52,0,-.2)(2.47,0,-.5)(2.52,0,-.7)(2.49,0,-.9)(2.5,0,-1)};

\foreach \j in {0,...,\thick}{

\draw[gray!50!white,opacity=.2]  plot[ smooth, tension=2] coordinates {(2.5,\j/\thick,0)(2.52,\j/\thick,-.2)(2.47,\j/\thick,-.5)(2.52,\j/\thick,-.7)(2.49,\j/\thick,-.9)(2.5,\j/\thick,-1)};

}

\end{scope}

\begin{scope}[shift={(13.25,-1.25)},rotate=25]


\draw[dashed](3,0,0)--++(3,0,0)--++(0,0,-1)--++(-3,0,0);

\draw(6,0,0)--++(0,1,0)--++(0,0,-1)--++(0,-1,0)--++(0,0,1);

\draw(3,1,0)--++(3,0,0)++(0,0,-1)--++(-3,0,0)++(0,0,1);

\draw(3,0,0)++(0,1,0)--++(3,0,0)++(0,-1,0)--++(-3,0,0);

\foreach \j in {0,...,\thick}{
\draw[gray!50!white,opacity=.2](3,0,0) plot[ smooth, tension=2] coordinates { (3,\j/\thick,0) (3.03,\j/\thick,-.1) (3.05,\j/\thick,-.3) (2.95,\j/\thick,-.5)(3.05,\j/\thick,-.7)(2.97,\j/\thick,-.85)(3,\j/\thick,-1)};

}

\draw[black,opacity=.5,fill opacity=.2,fill=gray!50!white](3,0,0)--++(0,1,0);
\draw[black,opacity=.5](3,0,-1)--++(0,1,0);

\foreach \j in {0,\thick}{
\draw[black,opacity=.5](3,0,0) plot[ smooth, tension=2] coordinates { (3,\j/\thick,0) (3.03,\j/\thick,-.1) (3.05,\j/\thick,-.3) (2.95,\j/\thick,-.5)(3.05,\j/\thick,-.7)(2.97,\j/\thick,-.85)(3,\j/\thick,-1)};
}

\end{scope}

\end{tikzpicture}
\caption{A collapsing void leading to a discontinuity for $J_y$ and $J_R$.}
\label{fig:dimred}
\end{figure}

\EEE

\textbf{Main results:}
Keeping in mind \eqref{order_1_domain}, \eqref{admissible_configurations_h_level}, \eqref{limiting_admissible_pairs}, and setting
\begin{equation*}
V_I:=I\times (-1/2,1/2)^2\in \mathcal{P}(\Omega)\ \ \text{for } I\in \mathcal{P}(0,L)\,,
\end{equation*}
our main results in this paper are summarized as follows.

\begin{theorem}\label{compactness_thm}
$\text{(\underline {Compactness})}$ Let $(h_j)_{j\in \N}\subset (0,\infty)$ with $h_j\searrow 0$  and $(y_{h_j}, V_{h_j}) \in \hat{\mathcal{A}}_{h_j}$ be  such that 
\begin{equation}\label{uniform_rescaled_energy_bound}
\sup_{j\in \mathbb{N}}\E^{h_j}(y_{h_j},V_{h_j})<+\infty\,. 
\end{equation}
Then, there exists $\big((y\vert \, d_2\vert \, d_3),I\big)\in \mathcal{A}$ such that up to a non-relabeled subsequence,
\begin{align}\label{compactness_properties}
\begin{split}
\rm{(i)}\quad&\chi_{V_{h_j}}\longrightarrow \chi_{V_I}\ \text{ in } L^1(\Omega)\,,\\
\rm{(ii)}\quad& y_{h_j}\longrightarrow \bar{y}  \text{ in  } L^1(\Omega;\R^3)  \,,\\
\rm{(iii)}\quad& \chi_{\Omega\setminus V_{h_j}}\nabla_{h_j}y_{h_j}\rightharpoonup \chi_{\Omega\setminus V_I} \bar{R}  \ \text{ weakly in } L^2(\Omega;\R^{3\times 3})\,,
\end{split}
\end{align}
where $ \bar{y} $ and $ \bar{R} $ are  meant here with the conventions made in \eqref{eq: convention1}--\eqref{eq: convention2}.
\end{theorem}

\begin{definition}\label{type_of_convergence}
 We say that $(y_{h_j},V_{h_j})\overset{\tau}{\longrightarrow}\big((y\vert \, d_2\vert \, d_3),I\big)$ if and only if  \eqref{compactness_properties} holds. 
\end{definition}

Since \eqref{admissible_configurations_h_level} implies that $\sup_{j\in \N}\|y_{h_j}\|_{L^\infty(\Omega)}\leq M$, the convergence in \eqref{compactness_properties}(ii) actually holds in $L^p(\Omega;\R^3)$ for every $p\in [1,+\infty)$. 
We are now ready to state the main $\Gamma$-convergence result.

\begin{theorem}\label{main_gamma_convergence_thm}$(\underline{ {\Gamma}\text{-convergence}})$ 
Let $(h_j)_{j\in \N}\subset (0,\infty)$ with $h_j\searrow 0$. The sequence of functionals $(\E^{h_j})_{j\in \N}$ $\Gamma(\tau)$-converges to the functional $\E^{0}$, i.e., the following two inequalities hold true.\\[-7pt]

\noindent $\rm{(i)}$ $(\underline{ {\Gamma}\text{-liminf inequality}})$ Whenever $(y_{h_j},V_{h_j})\overset{\tau}{\longrightarrow}\big((y\vert \, d_2\vert \, d_3),I\big)$, then
\begin{equation}\label{gamma_liminf_inequality}
\E^0\big((y\vert \, d_2\vert \, d_3), I\big)\leq \liminf_{j\to +\infty}\E^{h_j}(y_{h_j},V_{h_j})\,.
\end{equation} 
$\rm{(ii)}$ $(\underline{ {\Gamma}\text{-limsup inequality}})$ For every $\big((y\vert \, d_2\vert \, d_3), I\big)\in \mathcal{A}$ there exists a sequence $(y_{h_j},V_{h_j})_{j\in \N}$ with  $ (y_{h_j},V_{h_j})\in \hat{\mathcal{A}}_{h_j}$ for each $j \in \mathbb{N}$, such that $(y_{h_j},V_{h_j})\overset{\tau}{\longrightarrow}\big((y\vert \, d_2\vert \, d_3),I\big)$, and 
\begin{equation}\label{gamma_limsup_inequality}
\limsup_{j\to +\infty} \E^{h_j}(y_{h_j},V_{h_j})\leq \E^0\big((y\vert \, d_2\vert \, d_3), I\big)\,. 
\end{equation} 
\end{theorem}

\begin{remark}[Extensions and variants]\label{choice_of_model}
\normalfont (i) We could consider more general perimeter energies of the form
$$\beta_h\int_{\partial E\cap \Omega_h}\varphi(\nu_E)\, \mathrm{d}\H^2\,,$$
where $\lim_{h\to 0} (h^{-2}\beta_h)=\beta>0$ and $\varphi$ is a norm in $\R^3$. For simplicity of the exposition, we have chosen $\beta_h = h^2$ and $\varphi$ to be the standard Euclidean norm in $\R^3$. The general case is completely analogous in its treatment, up to a prefactor $\varphi(e_1)$ appearing in the last two terms in \eqref{limiting_one_dimensional_energy}.

(ii) Regarding the choice of the curvature regularization, let us mention that, in view of the results in \cite[Theorem 2.1]{KFZ:2021}, any choice of the form 
$$h^2\kappa_h\int_{\partial E\cap \Omega_h}|\bm{A}|^q\,\mathrm{d}\H^2$$
with $q\geq 2$ would be possible, up to adjusting the condition for $\kappa_h$ in \eqref{rate_1_gamma_h} (which will then depend also on $q$).  The choice $q \ge 2$, however, is essential, see \cite[Lemma 2.11 and Example 2.12]{KFZ:2021}. For simplicity, we have chosen the exponent $q=2$, which corresponds to a curvature regularization of Willmore type.   

(iii)  Let us also remark that clamped boundary conditions and body forces can be included into the $\Gamma$-convergence statement. We refrain here from giving the details, but refer the interested reader to \cite[Corollaries 2.4, 2.5]{schmidt2017griffith} for results in this direction. 
\end{remark}

\begin{remark}[Discussion on the limiting model] \label{rem:discussion} {\normalfont \NNN  The limiting model  includes the {Griffith-Euler-} \\ Bernoulli theory for brittle beams derived in \cite{schmidt2017griffith}, which corresponds to an  energy of the form
\begin{align}\label{def:GEB}
\mathcal{E}_{\mathrm{GEB}}(y):=\int_0^L |\kappa_y|^2\,\mathrm{d}x_1 + \mathcal{H}^0(J_y\cup J_{{y}'})\,,
\end{align}
where $y\in P\text{-}W^{2,2}((0,L);\mathbb{R}^2)$ is an arclength parametrization and $\kappa_y$ denotes the corresponding curvature. This follows from our model for $I=\emptyset$, deformations $y=(y_1,y_2,0)$, and  Frenet frames  $R=(y'|\,d_2|\,d_3)$ with $d_2=(-y_2',y_1',0)$ and $d_3=(0,0,1)$. This functional is related to a one-dimensional version of the Blake-Zisserman model \cite{Blake-Zisserman}, where $y$ is scalar and $\kappa_y$ is replaced by $y''$. Our model (with $I=\emptyset$) and the model in \eqref{def:GEB} have a similar response to boundary value problems.  In particular, prescribing boundary conditions $y(0)$ and $y(L)$ which are not compatible with smooth elements in $\mathcal{A}$, e.g.~if $|y(0)-y(L)| >L$, we necessarily have $J_y\cup J_{y'} \neq \emptyset$. Even if the boundary conditions can be achieved by smooth elements in $\mathcal{A}$,  cracks may be favorable whenever all curves connecting $y(0)$ and $y(L)$ have large curvature,  e.g.~if $|y(0)-y(L)| \ll L$.

Adding a volume constraint of the form $h^{-2}\mathcal{L}^3(E_h) \to 0$ in our $3$d model, we can easily recover \eqref{limiting_one_dimensional_energy} without voids, i.e., $I= \emptyset$. If we allow for voids in the limit, the  interpretation of the model is a bit more delicate, as the non-smoothness could be introduced through cracks, kinks, or voids. In the extreme case, even everything could be covered by void. To avoid the latter phenomenon, one could add a volume constraint of the form $\mathcal{L}^3(E_h) \leq \alpha h^2$ for $\alpha \in (0,L)$ or introduce body forces (both of which can be incorporated in the $\Gamma$-convergence result). With a body force of type
\begin{align*}
\int_{(0,L)\setminus I} f(x)\cdot y(x_1)\,\mathrm{d}x_1\,,
\end{align*}
cracks may be energetically favorable compared to voids. In fact, with lateral stretched boundary conditions (e.g.\ $y(0) = 0$, $y(L) = (L',0,0)$ with $L' > L$) and $f \equiv -e_3$ (corresponding to gravity force), one can check that it is convenient not to introduce void but a crack, with $J_y$ close to $0$ or $L$. }
\end{remark}

\begin{example}[$L^2$ vs $L^\infty$-bound on the curvature]\label{example:scaleandcardinality} {\normalfont Recall \eqref{initial_energy} and  \eqref{rescaled_energy}. The following example shows that we can exhibit configurations $(v_h, E_h) \in \mathcal{A}_h$ with
\begin{align*}
\sup_{h >0} h^{-4} \mathcal{F}^h(v_h, E_h) <+\infty\,,
\end{align*}
 where $E_h$ consists of balls which concentrate on arbitrarily small scales (independently of $h$), and whose number is diverging (with $h$).  As a preparation, let $r>0$ and observe that for $E:=B_r\subset\subset \Omega_h$, where $B_r$ is a ball of radius $r$, the second fundamental form of $\partial E$ satisfies $|{\bf A}|=\sqrt{2} r^{-1}$, and the surface energy contribution is
\begin{align}\label{ineq:energyball}
h^{-2}\left(\mathcal{H}^2(\partial B_r) + \kappa_h \int_{\partial B_r} |{\bf A}|^2\,\mathrm{d}\mathcal{H}^2\right) =  4\pi \big( h^{-2} r^2 + 2 h^{-2}\kappa_h \big)\,.
\end{align}
In the setting of \cite{SantiliSchmidt2022}  (cf.~Remark~3.1 \MMM therein) \EEE and for void sets consisting of a disjoint union of balls compactly contained in $\Omega_h$, an $L^\infty$-bound of the form $|{\bf A}| \le  Ch^{-1}$ implies a lower bound of order $h$ for the radius of each of those balls. The energy bound $h^{-2}\mathcal{H}^2(\partial E_h \cap \Omega_h) \leq C$ on the perimeter energy  implies now that such voids can only consist of finitely many disjoint balls \MMM whose \EEE cardinality  depends only on the a priori $L^\infty$-bound and the energy bound.

Instead, in our setting we can construct an example of a  sequence of voids  with the aforementioned requirements. We perform this construction for
\begin{align}\label{eq:ratekappaexample}
\kappa_h\to 0 \quad \text{such that} \quad h^{-2} \kappa_h \to 0\,.
\end{align}
This rate of convergence is possible   as $\kappa_h$ only needs to satisfy \eqref{rate_1_gamma_h}, take  e.g.\ $\kappa_h=h^{\AAA 51/25\EEE}$.  Let $N_h \in \mathbb{N}$, let $(x_{i,h})_{i=1}^{N_h} \subset \mathbb{R}^3$, $(r_{i,h})_{i=1}^{N_h} \subset (0,+\infty) $, be such that $B_{r_{i,h}}(x_{i,h}) \subset\AAA \subset\GGG \Omega_h$ for all $i$, with $r_{i,h} \leq h N_h^{-1/2}$ for all $i$, $B_{r_{i,h}}(x_{i,h}) \cap B_{r_{j,h}}(x_{j,h}) =\emptyset$ for $i\neq j$.  Set $E_h :=\bigcup_{i=1}^{N_h} B_{r_{i,h}}(x_{i,h})$ and  $ v_h  =\mathrm{id}$  (for definiteness only as this is not the point of this example). Then, by \eqref{ineq:energyball}, we have
\begin{align*}
\begin{split}	
h^{-4} \mathcal{F}^h(v_h, E_h) 
 &= 4\pi \sum_{i=1}^{N_h} h^{-2} r_{i,h}^2 + \AAA 8\pi N_h h^{-2} \kappa_h\EEE\,.
\end{split}
\end{align*}
By \eqref{eq:ratekappaexample} and  $r_{i,h} \leq h N_h^{-1/2}$ for all $i$, we can suitably choose $N_h \to +\infty$  to obtain a sequence of void sets $(E_h)_{h>0}$ with equi-bounded surface energy, that has a diverging number of components with no (not even diverging with $h$) $L^\infty$-control on the curvature.}
\end{example}

\section{\NNN Blockwise \EEE Sobolev modification of deformations}\label{preparatory_modifications}

This section is devoted to two preliminary propositions which are vital in the proofs of the compactness Theorem \ref{compactness_thm} in Section \ref{compactness} and the $\Gamma$-liminf inequality of Theorem \ref{main_gamma_convergence_thm} in Section \ref{gamma_liminf}. Our reasoning relies on the approximation of a sequence of deformations with equibounded energy by mappings which are \NNN blockwise \EEE Sobolev. This will allow us to use the results of \cite[Theorems 2.1 and 3.1]{Mora} in subsets of the domain where the  modified functions are weakly differentiable. In order to control the  surface contributions due to voids correctly, our arguments will also include estimates on the jump set of the \NNN blockwise \EEE Sobolev approximations. \NNN The main gain of our construction is the fact that,  \NNN in contrast to the jump set of  original deformations $(v_h)_{h>0}$, we are able to control the geometry of the jump set of the newly constructed sequence, namely the new jump set is contained in finitely many vertical planes. That is the reason why we call this modification \emph{blockwise} Sobolev approximation. \EEE

In this and in the following sections, we will use  the continuum  subscript $h>0$ instead of the sequential subscript notation $(h_j)_{j\in \N}$ for notational convenience. Before we can state the main results of this section, we need to collect some more notation. Recalling the definition of $\Omega_h$ in \eqref{eq: Omegah}, for $\rho\in (0,1)$ we define the slightly smaller reference domain
\begin{equation}\label{Omega_h_local}
\Omega_{h,\rho}:
=(\rho h,L-\rho h)\times(-\tfrac{h}{2}+ \tfrac{1}{2}\rho h,\tfrac{h}{2}-\tfrac{1}{2}\rho h)^2\,.
\end{equation}
For $T \in \N$, $T \gg 1$, we  cover  the domain $\Omega_{h}$ with \textit{T-cuboids}, namely
$$Q_{h}(i)   :=\big[(i-1)Th,iT h\big)\times(-\tfrac{h}{2},\tfrac{h}{2})^2\,,  $$ 
for $i=1,\dots, N := \lfloor{\tfrac{L}{Th}\rfloor+1}$, and let
\begin{equation}\label{overlapping_rectangles}
\mathcal{Q}_{h} := \lbrace Q_{h}(i)\colon  i=1 
,\dots, N  \rbrace  \,. 
\end{equation} 
While we will  eventually send $\rho \to 0$ in Section \ref{compactness},  $T$  is fixed throughout the paper. Therefore, we refrain from including $T$ in the notation of  $\mathcal{Q}_h$.    For $x \in \R^3$, $l \ge 0$ we introduce the \emph{stripes}
\begin{align}\label{eq: D not}
S^l_{h}(x)  :=  (x-l,x+l) \times (-\tfrac{h}{2} , \tfrac{h}{2}   )^2 \,. 
\end{align}
Similarly to \eqref{Omega_h_local}, for $\rho\in (0,1)$  we also introduce the smaller stripes
\begin{align}\label{eq: D not2}
S^l_{h,\rho}(x)  :=  (x-l + \rho l,x+l -\rho l ) \times (-\tfrac{h}{2} + \tfrac{1}{2}\rho h, \tfrac{h}{2} - \tfrac{1}{2} \rho h )^2    \,.   
\end{align}
For every measurable set $K\subset \R^3$ and $\gamma >0$ we introduce the localized surface energy
\begin{equation}\label{F_surf_energy}
\mathcal{G}^\gamma_{\rm{surf}}(E;K):=\mathcal{H}^2(\partial E\cap K)+ \gamma \int_{\partial E\cap K}|\bm A|^2\,\mathrm{d}\mathcal{H}^2\,, 
\end{equation}  
where for later purposes we use a general parameter $\gamma$ in place of $\kappa_h$. Then, given an infinitesimal sequence  $(\epsilon_h)_{h>0} \subset (0,+\infty)$, we define the total rescaled  energy  by 
\begin{align}\label{eq: G energy}
\mathcal{G}^h(v,E):= \frac{1}{h^2 \epsilon_h} \int_{\Omega_h\setminus \overline{E}} W(\nabla v)\,\mathrm{d}x+ \frac{1}{h^2} \mathcal{G}^{\kappa_h}_{\rm{surf}}(E;\Omega_h)  \,
\end{align}
for $(v,E)\in \mathcal{A}_h$. Note that for $\epsilon_h=h^2$ we have $\mathcal{G}^h(v,E) = h^{-4}\mathcal{F}^h(v,E)$ with $ \mathcal{F}^h$ as defined in \eqref{initial_energy}. In this section, we treat a more general scaling $\epsilon_h$ of the elastic energy in order to distinguish more clearly the scalings related to  the volume of $\Omega_h$ and that of the average elastic energy per unit volume.   

 \NNN In the next proposition, \EEE we assume that $T\gg 1$ is chosen big enough, see \eqref{eq: T1} \NNN for details. We also \NNN recall the role of $M\gg 1$ in \eqref{initial_admissible_configurations}. \EEE

\begin{proposition}[\NNN Blockwise \EEE Sobolev approximation of deformations]\label{Lipschitz_replacement}
Let $(\epsilon_h)_{h>0} \subset (0,+\infty)$ be a sequence satisfying $\limsup_{h \to 0} \epsilon_h h^{-2} < +\infty$, and let $0<\rho \le \rho_0$ for some universal $\rho_0>0$. Then, there exists a constant $C:=C(T,M)>0$ such that  for every sequence $(v_h,E_h)_{h>0}$ with $(v_h,E_h) \in \mathcal{A}_h$  and  
\begin{equation}\label{Lipschitz_replacement_energy_bound}
\sup_{h>0} \mathcal{G}^h(v_h,E_h)<+\infty\,,
\end{equation}
there exist sequences $(w_h)_{h>0}$ and $(R_h)_{h >0}$ with $w_h \in  SBV^2(\Omega_{h,\rho};\R^3)$ and $R_h \in  SBV^2(\Omega_{h,\rho};\R^{3\times 3})$  satisfying  the following properties:
\begin{align}\label{almost_the same_and_control_of_energy}
\begin{split}
\rm{(i)}\quad& \|w_h\|_{L^\infty(\Omega_{h,\rho})}\leq C, \quad  \|R_h\|_{L^\infty(\Omega_{h,\rho})}\leq  C\,,\\[2pt]
\rm{(ii)}\quad&  J_{w_h} \cup J_{R_h} \subset  \Omega_{h,\rho}  \cap \bigcup_{ Q_{h} \in \mathcal{Q}_{v_h}} \partial Q_{h} \quad  \text{for some } \mathcal{Q}_{v_h} \subset \mathcal{Q}_{h} \text{ with } \# \mathcal{Q}_{v_h} \le  C\,,\\[2pt]
\rm{(iii)}\quad& h^{-2} \L^3\big(\{x\in \Omega_{h,\rho}\colon w_h(x)\neq  v_h  (x)\}\big) \to 0,\quad 
h^{-2} \L^3\big(\Omega_{h,\rho} \cap\big\{|\nabla v_h - R_h |>\theta_h\big\}\big)  \to 0 \,,\\[2pt]
\rm{(iv)}\quad&\int_{\Omega_{h,\rho} } 
\mathrm{dist}^2(\nabla w_h,SO(3))\,\mathrm{d}x\leq   {C} h^2 \epsilon_h\,,\quad  \int_{\Omega_{h,\rho}} |\nabla R_h|^2\,  {\rm d}x \le  {C}\epsilon_h \,,
\end{split}
\end{align}	
where $(\theta_h)_{h>0} \subset (0,+\infty)$ is a sequence with  $\theta_h\to 0$ and $\theta_h \epsilon_h^{-1/2} \to \infty$. 
\end{proposition}

The result allows us to approximate $v_h$ by a \NNN blockwise \EEE Sobolev function $w_h$ by changing the mapping on an asymptotically vanishing portion of the volume, see \eqref{almost_the same_and_control_of_energy}(iii). The important point is that the elastic energy of $w_h$ is still of the same order, see  \eqref{almost_the same_and_control_of_energy}(iv). In the next sections we will also need some  control on the second gradient of $v_h$, which a priori might not exist. This is achieved by a second sequence of  functions $(R_h)_{h>0}$ which has  bounded derivative in $L^2$ and suitably approximates $\nabla v_h$, see again \eqref{almost_the same_and_control_of_energy}(iii),(iv).

The approximation also delivers a control on the jump set, see \eqref{almost_the same_and_control_of_energy}(ii), which corresponds to the fact that in the limit we expect functions which jump at most a finite number of times, see \eqref{limiting_admissible_pairs}, \eqref{SBV_2_isom}. The most delicate part in the derivation of the $ \Gamma$-liminf inequality for the surface energy in \eqref{limiting_one_dimensional_energy} is the  correct factor $2$ in front of $\H^0\big((J_{{y}}\cup J_{{R}})\setminus \partial^*I\big)$. This will be achieved by a contradiction based fundamentally on the following lemma: suppose that along a sequence the surface energy in a set $S^{2l}_{h}(x)$ (see \eqref{eq: D not})  was less than $\sim 2h^2$, i.e., so small such that the void cannot cut through the thin rod $\Omega_h$, see Figure \ref{fig:cases}. Then, the jump height of the sequences $ (w_h)_{h>0}$, $ (R_h)_{h>0}$ is small, see \eqref{eq: small jump} below. Later in Section \ref{gamma_liminf} this will allow us to exclude that the limiting functions $y$ and $R$ jump in the set $S^l_{h,\rho}(x)\subset S^{2l}_{h}(x)$.

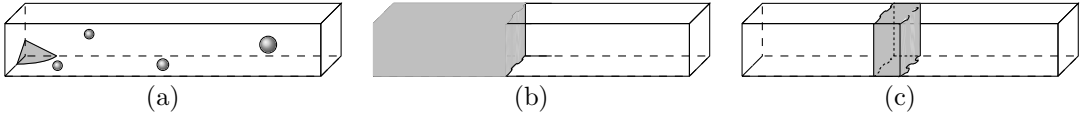
\begin{figure}[htp]
\begin{tikzpicture}[scale=.7]

\begin{scope}[shift={(7,0)}]

\draw(3,0) node[anchor=north]{${\rm (b)}$};

\draw[dashed](0,1,0)--++(0,0,-1)--++(0,-1,0)--++(0,0,1);

\draw[dashed](0,0,0)--++(6,0,0)--++(0,0,-1)--++(-6,0,0)--++(0,0,1);

\draw(6,0,0)--++(0,1,0)--++(0,0,-1)--++(0,-1,0)--++(0,0,1);

\draw(0,1,0)--++(6,0,0)--++(0,0,-1)--++(-6,0,0)--++(0,0,1);

\draw(0,0,0)--++(0,1,0)--++(6,0,0)--++(0,-1,0)--++(-6,0,0);

\pgfmathsetmacro{\thick}{50};



\draw[black,opacity=.5,fill opacity=.2,fill=gray!50!white](3,0,0)++(0,1,0)++(-.5,0,0)--++(0,-1,0)++(.5,0,0);

\foreach \j in {0,...,\thick}{

\draw[gray!50!white,opacity=.11]  plot[ smooth, tension=2] coordinates {(2.5,\j/\thick,0)(2.52,\j/\thick,-.2)(2.47,\j/\thick,-.5)(2.52,\j/\thick,-.7)(2.49,\j/\thick,-.9)(2.5,\j/\thick,-1)};
}

\draw[black,opacity=.5,fill opacity=.2](3,1,0)--++(-.5,0,0)-- plot[ smooth, tension=2] coordinates {(2.5,1,0)(2.52,1,-.2)(2.47,1,-.5)(2.52,1,-.7)(2.49,1,-.9)(2.5,1,-1)}--++(.5,0,0);

\draw[black,opacity=.5,fill opacity=.2,fill=gray!50!white]  plot[ smooth, tension=2] coordinates {(2.5,0,0)(2.52,0,-.2)(2.47,0,-.5)(2.52,0,-.7)(2.49,0,-.9)(2.5,0,-1)};

\draw[black,opacity=.5,fill opacity=.2](2.5,0,-1)--++(0,1,0);

\draw[black,opacity=.5,fill opacity=.2](2.5,0,-1)--++(.5,0,0);

\draw[gray!50!white,opacity=.2,fill opacity=.2,fill=gray!50!white](0,0,0)--++(2.5,0,0)--++(0,1,0)--++(-2.5,0,0)--++(0,-1,0);

\draw[gray!50!white,opacity=.2,fill opacity=.2,fill=gray!50!white](0,1,0)--++(2.5,0,0) plot[ smooth, tension=2] coordinates {(2.5,1,0)(2.52,1,-.2)(2.47,1,-.5)(2.52,1,-.7)(2.49,1,-.9)(2.5,1,-1)}--++(-2.5,0,0)--++(0,0,1);

\end{scope}

 \begin{scope}[shift={(14,0)}]
 
 \draw(3,0) node[anchor=north]{${\rm (c)}$};

\draw[dashed](0,1,0)--++(0,0,-1)--++(0,-1,0)--++(0,0,1);

\draw[dashed](0,0,0)--++(6,0,0)--++(0,0,-1)--++(-6,0,0)--++(0,0,1);

\draw(6,0,0)--++(0,1,0)--++(0,0,-1)--++(0,-1,0)--++(0,0,1);

\draw(0,1,0)--++(6,0,0)--++(0,0,-1)--++(-6,0,0)--++(0,0,1);

\draw(0,0,0)--++(0,1,0)--++(6,0,0)--++(0,-1,0)--++(-6,0,0);

\pgfmathsetmacro{\thick}{50};

\foreach \j in {0,...,\thick}{
\draw[gray!50!white,opacity=.2](3,0,0) plot[ smooth, tension=2] coordinates { (3,\j/\thick,0) (3.03,\j/\thick,-.1) (3.05,\j/\thick,-.3) (2.95,\j/\thick,-.5)(3.05,\j/\thick,-.7)(2.97,\j/\thick,-.85)(3,\j/\thick,-1)};

\draw[gray!50!white,opacity=.11](3,0,0) plot[ smooth, tension=2] coordinates {(2.5,\j/\thick,0)(2.52,\j/\thick,-.2)(2.47,\j/\thick,-.5)(2.52,\j/\thick,-.7)(2.49,\j/\thick,-.9)(2.5,\j/\thick,-1)};
}

\draw[black,opacity=.5,fill opacity=.2,fill=gray!50!white](3,0,0)--++(0,1,0)--++(-.5,0,0)--++(0,-1,0)--++(.5,0,0);
\draw[black,opacity=.5](3,0,-1)--++(0,1,0);

\foreach \j in {0,\thick}{
\draw[black,opacity=.5](3,0,0) plot[ smooth, tension=2] coordinates { (3,\j/\thick,0) (3.03,\j/\thick,-.1) (3.05,\j/\thick,-.3) (2.95,\j/\thick,-.5)(3.05,\j/\thick,-.7)(2.97,\j/\thick,-.85)(3,\j/\thick,-1)};
}

\draw[black,opacity=.5,fill opacity=.2,fill=gray!50!white](3,1,0)--++(-.5,0,0)-- plot[ smooth, tension=2] coordinates {(2.5,1,0)(2.52,1,-.2)(2.47,1,-.5)(2.52,1,-.7)(2.49,1,-.9)(2.5,1,-1)}--++(.5,0,0);

\draw[black,opacity=.5,dash pattern = on 1 pt off 1 pt]  plot[ smooth, tension=2] coordinates {(2.5,0,0)(2.52,0,-.2)(2.47,0,-.5)(2.52,0,-.7)(2.49,0,-.9)(2.5,0,-1)};

\draw[black,opacity=.5,dash pattern = on 1 pt off 1 pt](2.5,0,-1)--++(0,1,0);

\draw[black,opacity=.5,dash pattern = on 1 pt off 1 pt](2.5,0,-1)--++(.5,0,0);

\end{scope}

\begin{scope}[shift={(0,0)}]

\draw(3,0) node[anchor=north]{${\rm (a)}$};

\pgfmathsetmacro{\thicka}{120};

\foreach \j in {0,...,\thicka}{

\draw[gray!50!white,opacity=.11] (0,.3*\j/\thicka,-.6-.4*\j/\thicka) arc [start angle=-90,end angle=-20,x radius=0.8-.8*\j/\thicka,y radius=0.24-.24*\j/\thicka];
}

\draw[black,opacity=.5] (0,0,-.6)--(0,.3,-1);

\draw[black,opacity=.5] (.6,0,-1)--(0,.3,-1);

\draw[black,opacity=.5] (0,0,-.6) arc [start angle=-90,end angle=-20,x radius=0.8,y radius=0.24];


\draw[black,opacity=.5] (1.6,.8) circle (.09cm);

\shade[ball color=gray!50!white,opacity=0.2] (1.6,.8)circle (.08cm);

\draw[black,opacity=.5] (1,.2) circle (.09cm);

\shade[ball color=gray!50!white,opacity=0.2] (1,.2)circle (.08cm);

\draw[black,opacity=.5] (5,.6) circle (.16cm);

\shade[ball color=gray!50!white,opacity=0.2] (5,.6)circle (.15cm);

\draw[black,opacity=.5] (3,.22) circle (.11cm);

\shade[ball color=gray!50!white,opacity=0.2] (3,.22)circle (.1cm);

\draw[dashed](0,1,0)--++(0,0,-1)--++(0,-1,0)--++(0,0,1);

\draw[dashed](0,0,0)--++(6,0,0)--++(0,0,-1)--++(-6,0,0)--++(0,0,1);

\draw(6,0,0)--++(0,1,0)--++(0,0,-1)--++(0,-1,0)--++(0,0,1);

\draw(0,1,0)--++(6,0,0)--++(0,0,-1)--++(-6,0,0)--++(0,0,1);

\draw(0,0,0)--++(0,1,0)--++(6,0,0)--++(0,-1,0)--++(-6,0,0);

\end{scope}

\end{tikzpicture}
\caption{The void set is depicted in gray. In the situation of \eqref{eq: small jump/vol} or Remark \ref{rem: main rem}(i), only case (a) can occur, whereas (b),(c) are impossible.}
\label{fig:cases}
\end{figure}

\NNN Recall the stripes $S_h^{l}(x)$ and $S_{h,\rho}^{l}(x)$ defined in \eqref{eq: D not} and \eqref{eq: D not2}, respectively.\EEE

 \begin{proposition}[Jumps of \NNN blockwise \EEE Sobolev modifications]\label{prop: 2nd main} 
Let $(v_h,E_h)\in \mathcal{A}_h$ be a sequence from Proposition \ref{Lipschitz_replacement} and let $w_h\in SBV^2(\Omega_{h,\rho};\R^3)$, $R_h\in SBV^2(\Omega_{h,\rho};\R^{3\times 3})$ \EEE be the corresponding functions satisfying \eqref{almost_the same_and_control_of_energy}. Then, there exist open sets $E^*_h$ with $E_{h} \subset E^*_h \subset \Omega_h$, $\partial E^*_h\cap \Omega_{h}$ is a union of finitely many \NNN $C^2$\EEE-regular submanifolds, that satisfy 
\begin{align}\label{eq: void-newXXX}
h^{-3} \mathcal{L}^3(E^*_h\EEE\setminus E_{h}) \to 0,  \quad \quad      \liminf_{h \to 0} h^{-2}\H^2(\partial E^*_h\cap \Omega_h) \le \liminf_{h \to 0} h^{-2} \mathcal{G}^{\kappa_h}_{\rm{surf}}(E_h;\Omega_h) \,,
\end{align}
 such that  the following holds for any $l \ge  6Th$,  $x \in (2l,L-2l)$: For any  \NNN stripe \EEE $S_{h}^{2l}(x) $ with 
\begin{align}\label{eq: small jump/vol}
\begin{split}
\rm{(i)} \quad & \frac{1}{ ((1-\rho)h)^2} \H^2\big(  \partial E^*_h \cap S_{h}^{2l}(x)  \big) <2\,,\\[2pt] 
\rm{(ii)} \quad & \frac{\mathcal{L}^3(E_h^* \cap S_{h}^{2l}(x))}{\mathcal{L}^3(S_{h}^{2l}(x))}  \le \frac{1}{9}\,, 
\end{split}
\end{align}
 it holds that
\begin{align}\label{eq: small jump}
\frac{1}{h^2}\int_{S^l_{h,\rho}(x)   \cap  J_{w_h}   }  \sqrt{|[w_h]|} \, {\rm d} \mathcal{H}^2 + \frac{1}{h^2}\int_{S^l_{h,\rho}(x)   \cap  J_{R_h}   }  \sqrt{|[R_h]|} \, {\rm d} \mathcal{H}^2  \to 0\,.
\end{align}
 \end{proposition}

\begin{remark}\label{rem: main rem}
{\normalfont
(i)  One can also prove a variant of Proposition \ref{prop: 2nd main}: if $2$ in the right hand side of \eqref{eq: small jump/vol}(i) is replaced by $1$, then assumption \eqref{eq: small jump/vol}(ii) is not needed,  cf.\ Figure \ref{fig:cases}. 
(ii)  In \eqref{eq: small jump}, $\sqrt{|[w_h]|}$ and $\sqrt{|[R_h]|}$  can be replaced by $|[w_h]|^{1-\beta}$
and  $|[R_h]|^{1-\beta}$ for any $\beta \in (0,1)$, up to adjusting the condition for $\kappa_h$ in \eqref{rate_1_gamma_h} (which will then depend also on $\beta$). \EEE We omit details as the choice $\beta = \frac{1}{2}$ is enough for our purposes. 

We refer to Remark \ref{rem: the lambda case} below for a short comment  how to prove (i) and (ii). 
}
\end{remark}

The rest of this section is  entirely devoted to the proofs of Propositions \ref{Lipschitz_replacement}--\ref{prop: 2nd main}. The proofs of our main compactness and $\Gamma$-convergence results then start in Section \ref{compactness}. In the proofs, we will send the parameters $h, \rho $ to zero (in this order). In order to avoid overburdening of notation, generic constants which are independent of $h, \rho$ but may depend on the fixed parameters $ T, L$ are denoted by $C$. We will use a subscript notation whenever we want to highlight the dependence of a particular constant on a specific parameter.

\subsection{Rigidity results}\label{se: rigi}

This subsection is devoted to recalling some rigidity results which are the basis for our proofs.

\emph{Geometric rigidity in variable domains:}  We first recall the result \NNN \cite[Theorem 2.1]{KFZ:2021}\EEE. For convenience, we will directly formulate it on the set  $\Omega_h$ and its subset  $\Omega_{h,\rho}$, see \eqref{eq: Omegah} and  \eqref{Omega_h_local}.  The behavior of deformations $v$ on (connected components of)  $\Omega_h \setminus \overline{E}\EEE$ might not be  rigid. We refer to \cite[Example~2.6]{KFZ:2021} for an explanation in that direction. A key observation in \cite{KFZ:2021} is that rigidity estimates can be obtained outside of a thickened version of the voids.  We start by formulating this result on the modification of the void sets.

\begin{proposition}[Thickening of sets]\label{prop:setmodification} 
Let $h,\rho>0$, let $\gamma \in (0,1)$. Then, there  exist a universal constant $C_0>0$,  $\eta_0 = \eta_{0}(\rho) \in (0,1)$, and for each $\eta\in (0,\eta_0]$ the following holds:\\
Given   $E\in \mathcal{A}_{\rm reg}(\Omega_h)$, we can find an open set $E_{h,\eta,\gamma}$ such that $E  \subset  E_{h,\eta,\gamma}  \subset \Omega_h$, $\partial E_{h,\eta,\gamma}\cap \Omega_h$ is a union of finitely many \NNN$C^2$\EEE-regular  submanifolds, and   
\begin{align}\label{eq: partition-new}
\begin{split}
{\rm (i)} & \ \    \mathcal{L}^3(E_{h,\eta,\gamma}\setminus E) \le h\eta\EEE \gamma^{1/{ 2\EEE}}  \mathcal{G}_{\rm surf}^{\gamma h^2}(E;\Omega_h),  \EEE \quad \quad \quad  \dist_\mathcal{H}(E, E_{h,\eta,\gamma}) \le h\eta\EEE \gamma^{1/{2\EEE}}\,,  \EEE 
\\
{\rm (ii)} & \ \      \H^2(\partial E_{h,\eta,\gamma} \cap \Omega_h )   \leq (1+C_0\eta\EEE) \,  \mathcal{G}_{\rm surf}^{\gamma h^2}(E;\Omega_h)\,.
\end{split}
\end{align}
\end{proposition}

On the complement $\Omega_{h,\rho}\setminus \overline{E_{h,\eta,\gamma}}$ quantitative piecewise rigidity estimates  hold, as the following result shows. Recall the notation $S^{l}_{h}(x)$ in \eqref{eq: D not}.   

\begin{theorem}[Geometric rigidity in variable domains]\label{prop:rigidity}  
Let $h,\rho>0$, let $\gamma \in (0,1)$ and $l >0$. Then, there exist a universal constant $C_0>0$,   $\eta_0 = \eta_0(\rho)>0$,   and for each $\eta\in (0,\eta_0]$ there exists  $C_\eta = C_\eta(\eta,   \frac{l}{h})>0$ with $C_\eta \to \infty$ as $\eta \to 0$, $\tfrac{l}{h} \to 0$, or $\tfrac{l}{h} \to \infty$ such that the following holds:\\
\noindent For every  $E \in \mathcal{A}_{\rm reg}(\Omega_h)$, denoting by  $E_{h,\eta,\gamma}$ the set of Proposition \ref{prop:setmodification}, for every $U = S^{l}_{h}(x) \subset \Omega_h$ and $\tilde{U} = S^{l}_{h,\rho}(x) \subset \Omega_{h,\rho}$,  for the connected components $(\tilde U_j)_j$  of $\tilde U  \setminus \overline{{E_{h,\eta,\gamma}}}$ and for every $y  \in \NNN W^{1,2}\EEE(\Omega_h \setminus \overline{E};\R^3)$ there exist  corresponding rotations $(R_j)_j \subset SO(3)$ and   vectors $(b_j)_j\subset \R^3$ such that 
\begin{align}\label{eq: main rigitity}
{\rm (i)} & \ \  \sum\nolimits_j \int_{
\tilde U_j}\big|{\rm sym}\big((R_j)^T \nabla y-\mathrm{Id}\big)\big|^2\,\mathrm{d}x 
\leq C_0 \big(1 +  C_\eta \gamma^{-15/2} h^{-3} \varepsilon   \big)\int_{U\setminus  \overline{E}} \mathrm{dist}^2(\nabla y,SO( 3))\,\mathrm{d}x\,,\notag\\
{\rm (ii)} & \ \   \sum\nolimits_{j}\int_{
 \tilde U_j}    \big|(R_j)^T \nabla y-\mathrm{Id}\big|^2\,\mathrm{d}x \leq C_\eta  \gamma^{-3} \int_{ U\setminus  \overline{E} \EEE} \mathrm{dist}^2(\nabla y,SO( 3))\,\mathrm{d}x\,,
\notag\\
{\rm (iii)} & \ \   \sum\nolimits_{j}\int_{
\tilde U_j}  \frac{1}{h^2} \big|  y- (R_jx+b_j) \big|^2 \,\mathrm{d}x \leq C_\eta  \gamma^{-\AAA 5\EEE} \int_{ U\setminus  \overline{E} \EEE} \mathrm{dist}^2(\nabla y,SO( 3))\,\mathrm{d}x\,,
\end{align}
 where for brevity $\eps := \int_{U \setminus \overline{E}} \dist^2(\nabla y, SO( 3)) \, {\rm d}x$. \EEE
\end{theorem}

\begin{proof}[Proof of Proposition \ref{prop:setmodification}  and Theorem \ref{prop:rigidity}]
The result is essentially given in \cite[Theorem 2.1]{KFZ:2021}. We explain here the adaptations necessary to the present version of the result, in particular the scaling in terms of the small parameter $h>0$.

We apply \cite[Theorem 2.1]{KFZ:2021}  for $d=3$,  $q=2$, $\gamma \in (0,1)$ and $\varphi \equiv  \|\cdot\|_2$ on the sets $h^{-1} \Omega_h \subset \R^3$, $\tilde{\Omega} := h^{-1} \Omega_{h,\rho}$ and $h^{-1} E$. The constant $\eta_0$ therein depends only on ${\rm dist}(h^{-1} \partial \Omega_h,h^{-1}\Omega_{h,\rho})$ and  can  thus be chosen  depending only on $\rho$, see \eqref{eq: Omegah} and  \eqref{Omega_h_local}.  Then, the result  first provides a set  $E_{\eta,\gamma} $ with $h^{-1} E  \subset  E_{\eta,\gamma}    \subset h^{-1}\Omega_h$ such that by \cite[(2.2)]{KFZ:2021}   and a scaling argument the set  $E_{h,  \eta,\gamma} := h E_{  \eta,\gamma}$ satisfies \eqref{eq: partition-new}. Here, we particularly note that a change of variables implies 
$$\mathcal{G}_{\rm surf}^{\gamma}(h^{-1} E; h^{-1}\Omega_h) = h^{-2} \mathcal{G}_{\rm surf}^{\gamma h^2}(E;\Omega_h)\,. $$
Then, \eqref{eq: main rigitity} follows from a  localized version of     \cite[(2.3)]{KFZ:2021}, see \cite[Remark 2.10]{KFZ:2021}, first applied on the sets  $h^{-1}U$, $h^{-1} \tilde{U}$ and the function $ w_h(x) = \frac{1}{h}  y(hx)$ for $x \in h^{-1}(\Omega_h \setminus \overline{E})$,   and then again rescaled. The factors $h^{-3}$ and $h^{-2}$ in  \eqref{eq: main rigitity}$\rm{(i),(iii)}$, respectively, ensure that all inequalities in  \eqref{eq: main rigitity} are scaling invariant in the sense that the constants are independent of $h$. The factors $\gamma^{-15/2}$, \AAA $\gamma^{-3}$ and $\gamma^{-5}$ \EEE follow from the choice of $d=3$ and $q=2$.  We further  observe that the constant $C_\eta$ depends on $\eta$ and  $\mathcal{L}^3(h^{-1} U)$, see \cite[Remark 2.10]{KFZ:2021}. As  $\mathcal{L}^3(h^{-1} U) = \frac{ 2\EEE l}{h}$,    we indeed get that $C_\eta$  depends on $\eta$  and   the ratio  of $l$ and $h$.  
\end{proof}

In the proofs below, we will apply this rigidity result  on the $T$-cuboids $\mathcal{Q}_{h}$ introduced in  \eqref{overlapping_rectangles} or in finite unions of such cuboids. For these sets, we observe that the constant $C_\eta$ depends only on $\eta$ and $T$ (as the corresponding $l$ is approximately   $Th$).  Moreover, we will choose  $\eta$ and  $\gamma$ depending on the regime of the elastic energy $\eps$ such that $C_\eta \gamma^{-15/2}h^{-3}\varepsilon \le 1$ and $C_\eta  \gamma^{\AAA -5\EEE}  \le \eps^{-\theta}$ for some $\theta >0$ small. Thus, we obtain a sharp control on symmetrized gradients in terms of $\eps$ (see \eqref{eq: main rigitity}$\rm{(i)}$), while the rigidity estimate in \eqref{eq: main rigitity}$\rm{(ii)}$ and the Poincar\'e-type estimate \eqref{eq: main rigitity}$\rm{(iii)}$ yield control of order $\eps^{1-\theta}$, hence being  suboptimal in the exponent.

\emph{Korn and Poincar\'e inequalities:} The issue of the suboptimal exponent can be remedied provided that the surface measure of the void set is small. This relies on delicate Korn and Poincar\'e inequalities in the space  $GSBD^2$, see  \cite{DalMaso:13} for the definition of this space. We formulate the result of \cite[Theorem~1.1, Theorem~1.2]{Cagnetti-Chambolle-Scardia} in a simplified setting which does not involve functions in $GSBD^2$ but only $SBV^2$-functions. In the following, we say that $a\colon \R^{3} \to \R^{3}$ is an \emph{infinitesimal rigid motion} if $a$ is affine with ${\rm sym}(\nabla a)= 0$.

\begin{theorem}[Korn inequality for functions with small jump set]\label{th: kornSBDsmall}
Let $U \subset \R^{3}$ be a bounded Lipschitz domain. Then, there exists a constant $c = c(U)>0$ such that for all  $u \in SBV^2(U;\R^{3})$  there  exists  a set of finite perimeter $\omega \subset U$ with 
\begin{align}\label{eq: R2main}
\mathcal{H}^{ 2}(\partial^* \omega) \le c\mathcal{H}^{ 2}(J_u)\,, \ \ \ \ \mathcal{L}^{ 3}(\omega) \le c(\mathcal{H}^{ 2}(J_u))^{3/2}\,,
\end{align}
and an infinitesimal rigid motion $a$ such that
\begin{align}\label{eq: main estmain}
 (\mathrm{diam}(U))^{-1}\Vert u - a \Vert_{L^{2}(U \setminus \omega)} + \Vert \nabla u - \nabla a \Vert_{L^{2}(U \setminus \omega)}\le c \Vert {\rm sym}(\nabla u) \Vert_{L^2(U)}.
\end{align}
Moreover, there exists $v \in W^{1,2}(U;\R^{3})$ such that $v\equiv  u$ on $U \setminus \omega$ and
\begin{align*}
\Vert {\rm sym}(\nabla v) \Vert_{L^2(U)} \le c \Vert {\rm sym}(\nabla  u)\Vert_{L^2(U)}.
\end{align*}
Furthermore\EEE, if $u \in L^\infty(U;\R^{ 3\EEE})$ one has  $\|v\|_{L^\infty(U)}\leq  \| u \|_{L^\infty(U)}$.  
\end{theorem}

This follows from \cite{Cagnetti-Chambolle-Scardia} (for $d=3$ and $p=2$) by the fact that $SBV^2 \subset GSBD^2$.   Note that in  \cite[Theorem 1.1]{Cagnetti-Chambolle-Scardia}  $\mathcal{L}^3(\omega) \le c(\mathcal{H}^{2}(J_u))^{3/2}$ has not been  stated explicitly,  but it readily follows from $\mathcal{H}^{2}(\partial^* \omega) \le c\mathcal{H}^{2}(J_u)$ and the isoperimetric inequality. The result is indeed only relevant if $\mathcal{H}^{2}(J_u)$ is small since otherwise $\omega = U$ is possible and the statement is empty. In a similar fashion to the reasoning in  Theorem \ref{prop:rigidity}, it is a standard matter to see that the constant in \eqref{eq: R2main}--\eqref{eq: main estmain} is invariant under translation and rescaling of the domain.

\emph{Difference of affine maps:} To estimate the difference of rigid motions, we   make use of the following elementary lemma. By $B_r(x) \subset \R^3$ we denote the open ball  centered at $x \in \R^3$ with radius $ r>0$. 
\begin{lemma}[Estimate on  affine maps]\label{lemma: rigid motions}
Let $\delta >0$. Then there exists a constant $C>0$ only depending on  $\delta$ such for every $G \in \R^{3\times  3}$, $b \in \R^{3}$, $x\in \R^3$, and $E \subset B_r(x)$ for some $r >0$ with $\mathcal{L}^{ 3}(E) \ge \delta r^{ 3}$ we have
$$\Vert G \cdot  + b \Vert_{L^\infty(B_r(x))} \le Cr^{- 3} \mathcal{L}^{ 3}(E)^{1/2} \Vert G \cdot  + b \Vert_{L^2(E)}, \quad |G| \le C  r^{-4}\mathcal{L}^{3}(E)^{ 1/2} \Vert G \cdot  + b \Vert_{L^2(E)}\,.   $$
\end{lemma}

\begin{proof}
For $r = 1$ and $x=0$, the result is a special case of \cite[Lemma 3.4]{Friedrich-Solombrino}, applied  (for $d=3$) to $\psi(t) := t^2$. In particular, in  \cite[(3.4)]{Friedrich-Solombrino} we also use H\"older's inequality to get the control in terms of  the quantity \EEE $\mathcal{L}^{3}(E)^{ 1/2} \Vert G \cdot  + b \Vert_{L^2(E)}$. For general $r>0$ and $x \in \R^3$, the estimates follow from a standard scaling and translation argument. 
\end{proof}

\subsection{Isoperimetric inequalities on cuboids}\label{sec: isoperimetric}

In this subsection, we present a special case of a relative isoperimetric inequality in cuboids  that are long in one direction, where the isoperimetric constant is independent of the length.  Such an inequality is possible for sets that have small relative perimeter as, in this case, isoperimetric sets will concentrate at one of the corners or at one of the short edges of the long cuboid. Indeed, under the small perimeter constraint, the relative boundary cannot span a cross section of the cuboid, see Figure~\ref{fig:isoperimetric}. \EEE  As the result may be interesting in its own right, it is formulated in arbitrary space dimension on the cuboids $S^l_{ \sigma \EEE}(x_0):=x_0+(-l,l)\times (-\tfrac{ \sigma \EEE}{2},\tfrac{ \sigma \EEE}{2})^{d-1}$, consistent with the notation \eqref{eq: D not}. Afterwards, we will present two consequences which will be used in the sequel.

\begin{proposition}[Relative isoperimetric inequality on cuboids]\label{nice_new_isoperimetric_inequality}
Let $l, \sigma >0$ with   $l/ \sigma \geq 1$,   and  
$x_0\in \R^d$. Then,  there exists a dimensional constant $C_{\rm iso}\ge 1$ independent of $l$ and $ \sigma \EEE$  such that for every set of finite perimeter $P\subset S^l_{ \sigma \EEE}(x_0)$ with
\begin{equation}\label{small_perimeter_condition}
\mathcal{H}^{d-1}\big(\partial^*P\cap S^l_{ \sigma \EEE}(x_0)\big)<  \sigma \EEE^{d-1}\,,
\end{equation}
it holds that  
\begin{equation}\label{long_relative_isoperimetric_inequality}
\min\{\L^d(P),\L^d(S^l_{ \sigma \EEE}(x_0)\setminus P)\}\leq C_{\rm iso} \sigma \EEE \H^{d-1}(\partial^* P\cap S^l_{ \sigma \EEE}(x_0))\,.
\end{equation}
\end{proposition}

\begin{figure}[htp]
\begin{tikzpicture}

\draw(3,0) node[anchor=north]{${\rm (a)}$};

\pgfmathsetmacro{\thick}{50};

\draw[dashed](0,1,0)--++(0,0,-1)--++(0,-1,0)--++(0,0,1);

\draw[dashed](0,0,0)--++(6,0,0)--++(0,0,-1)--++(-6,0,0)--++(0,0,1);

\draw(6,0,0)--++(0,1,0)--++(0,0,-1)--++(0,-1,0)--++(0,0,1);

\draw(0,1,0)--++(6,0,0)--++(0,0,-1)--++(-6,0,0)--++(0,0,1);

\draw(0,0,0)--++(0,1,0)--++(6,0,0)--++(0,-1,0)--++(-6,0,0);

\pgfmathsetmacro{\thicka}{120};

\foreach \j in {0,...,\thicka}{

\draw[gray!50!white,opacity=.2] (0,.5*\j/\thicka,-.4-.6*\j/\thicka) arc [start angle=-90,end angle=0,x radius=0.8-.8*\j/\thicka,y radius=0.24-.24*\j/\thicka];
}

\draw[black,opacity=.5] (0,0,-.4)--(0,.5,-1);

\draw[black,opacity=.5] (.57,0,-1)--(0,.5,-1);

\draw[black,opacity=.5] (0,0,-.4) arc [start angle=-90,end angle=0,x radius=0.8,y radius=0.24];

\begin{scope}[shift={(8,0)}]

\draw[dashed](0,1,0)--++(0,0,-1)--++(0,-1,0)--++(0,0,1);

\draw[dashed](0,0,0)--++(6,0,0)--++(0,0,-1)--++(-6,0,0)--++(0,0,1);

\draw(6,0,0)--++(0,1,0)--++(0,0,-1)--++(0,-1,0)--++(0,0,1);

\draw(0,1,0)--++(6,0,0)--++(0,0,-1)--++(-6,0,0)--++(0,0,1);

\draw(0,0,0)--++(0,1,0)--++(6,0,0)--++(0,-1,0)--++(-6,0,0);

\pgfmathsetmacro{\thick}{50};



\draw[black,opacity=.5,fill opacity=.2,fill=gray!50!white](3,0,0)++(0,1,0)++(-.5,0,0)--++(0,-1,0)++(.5,0,0);

\foreach \j in {0,...,\thick}{

\draw[gray!50!white,opacity=.2]  plot[ smooth, tension=2] coordinates {(2.5,\j/\thick,0)(2.52,\j/\thick,-.2)(2.47,\j/\thick,-.5)(2.52,\j/\thick,-.7)(2.49,\j/\thick,-.9)(2.5,\j/\thick,-1)};
}

\draw[black,opacity=.5,fill opacity=.2](3,1,0)--++(-.5,0,0)-- plot[ smooth, tension=2] coordinates {(2.5,1,0)(2.52,1,-.2)(2.47,1,-.5)(2.52,1,-.7)(2.49,1,-.9)(2.5,1,-1)}--++(.5,0,0);

\draw[black,opacity=.5,fill opacity=.2,fill=gray!50!white]  plot[ smooth, tension=2] coordinates {(2.5,0,0)(2.52,0,-.2)(2.47,0,-.5)(2.52,0,-.7)(2.49,0,-.9)(2.5,0,-1)};

\draw[black,opacity=.5,fill opacity=.2](2.5,0,-1)--++(0,1,0);

\draw[black,opacity=.5,fill opacity=.2](2.5,0,-1)--++(.5,0,0);

\draw[gray!50!white,opacity=.2,fill opacity=.2,fill=gray!50!white](0,0,0)--++(2.5,0,0)--++(0,1,0)--++(-2.5,0,0)--++(0,-1,0);

\draw[gray!50!white,opacity=.2,fill opacity=.2,fill=gray!50!white](0,1,0)--++(2.5,0,0) plot[ smooth, tension=2] coordinates {(2.5,1,0)(2.52,1,-.2)(2.47,1,-.5)(2.52,1,-.7)(2.49,1,-.9)(2.5,1,-1)}--++(-2.5,0,0)--++(0,0,1);

\draw(3,0) node[anchor=north]{${\rm (b)}$};

\end{scope}

\end{tikzpicture}
\caption{Void set contained in a thin rod with (a) relative perimeter less than $ \sigma^{d-1}$, (b) relative perimeter bigger than $ \sigma^{d-1}$.}
\label{fig:isoperimetric}
\end{figure}
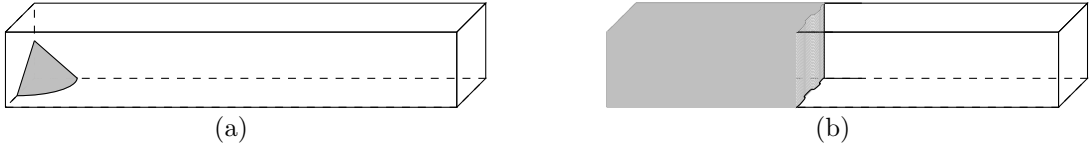

\EEE

\begin{proof}
Without loss of generality, after translation and uniform rescaling, we can assume that $x_0=0$, $ \sigma=1$, and  can without restriction reduce to showing the following assertion on the cuboid   $Q_{l}:=S^l_{1}(0)=(-l,l)\times(-\tfrac{1}{2},\tfrac{1}{2})^{d-1}$.  There exists a dimensional constant $c_d>0$ such that  for each $l \geq 1$ and for every set of finite perimeter $P\subset Q_{l}$ with 
\begin{equation*}
\L^d(P)\leq l\, \quad \text{and} \quad \H^{d-1}(\partial^* P\cap Q_{l})  <1\,,
\end{equation*}
there holds that
\begin{equation}\label{rescaled_long_relative_isoperimetric}
\L^d(P)\leq c_d\H^{d-1}(\partial^* P\cap Q_{l})\,.
\end{equation}
Note that we can assume that $l \gg 1$ as, given $l_0 >1$, we have that for all $1\leq l \leq l_0$ the statement follows directly from the classical relative isoperimetric inequality
\begin{align*}
\min\{\L^d(P),\L^d( Q_l
\setminus P)\}\leq C_{\rm iso}(l)  \H^{d-1}(\partial^* P\cap Q_l
)\,,
\end{align*}
and the fact that that $C_{\rm iso}(l)  \leq C_0 l_0$, where $C_0>0$ is a dimensional constant. 

We prove the assertion by induction on the dimension $d$, the case $d=1$ being a trivial statement. Assume now that \eqref{rescaled_long_relative_isoperimetric} is true for some $d\geq 1$, and for the inductive step let us prove it in dimension $d+1$. For this purpose, let $P\subset Q^{d+1}_{l}:=(-l,l)\times(-\tfrac{1}{2},\tfrac{1}{2})^d$ be a set of finite perimeter, with
\begin{equation}\label{inductive hypothesis}
\L^{d+1}(P)\leq l\, \quad \text{and} \quad \H^{d}(\partial^* P\cap Q^{d+1}_{l})  <1\,.
\end{equation}
For every $t\in (-\tfrac{1}{2},\tfrac{1}{2})$, let us set for notational simplicity $Q^d_{l,t}:=Q^{d+1}_{l}\cap \{x_{d+1}=t\}$ and also $P_t:=P\cap Q^d_{l,t}$. By general slicing properties of sets of finite perimeter, $P_t$ is a subset of finite perimeter in $Q^d_{l,t}$ for  $\L^1$-\EEE a.e. $t\in (-1/2,1/2)$. Let now $t_0\in (-1/2,1/2)$ be such that $P_{t_0}$ is of finite perimeter and also, as a consequence of the coarea formula (cf.\ \cite[(18.25)]{maggi2012sets}) and \eqref{inductive hypothesis},  
\begin{align}\label{minimality_condition}
\H^{d-1}(\partial^*P_{t_0}\cap Q^{d}_{l,t_0})&\leq \int_{-1/2}^{1/2}\H^{d-1}(\partial^*P_{t}\cap Q^d_{l,t})\, \mathrm{d}t\le  \H^d(\partial^\ast P \cap Q^{d+1}_{l})  <1\,. 
\end{align}
 By 
\eqref{inductive hypothesis}, we also find  that for $\L^1$-a.e.\ $t\in (-1/2,1/2)$,  
\begin{equation}\label{difference_in volume_slices}
\begin{split}
\big|\mathcal{L}^d(P_{t})-\mathcal{L}^d(P_{t_0})\big|&\leq \mathcal{H}^{d}\Big(\partial^*P\cap\big((-l,l)\times (-1/2,1/2)^{d-1}\times[t_0\wedge t, t_0\vee t ] \big)  \Big)\\&\leq \mathcal{H}^d(\partial^*P \cap Q_l^{d+1}) <1\,.
\end{split}
\end{equation}
 Note that the first inequality in \eqref{difference_in volume_slices} is immediate for smooth sets via a projection argument. In the general case, it can be derived by the density of smooth sets  and Fubini's theorem.  
Therefore, we get
\begin{equation}\label{bound_on_Ld_of_slices}
\mathcal{L}^d(P_{t_0})- 1
< \mathcal{L}^d(P_{t}) <
\mathcal{L}^d(P_{t_0})+ 1 \quad \text{for $\L^1$-a.e. } t\in(-1/2,1/2)\,.
\end{equation}
We now claim that 
\begin{equation}\label{P_t_0_is_the_small_one}
\mathcal{L}^d(P_{t_0})\leq \mathcal{L}^d(Q^d_{l,t_0}\setminus P_{t_0})\,.
\end{equation}
Indeed, if \eqref{P_t_0_is_the_small_one} was not true, then by \eqref{minimality_condition} and the inductive hypothesis we would have
\begin{align}\label{the_wrong_isoperimetric}
\mathcal{L}^d(Q^d_{l,t_0}\setminus P_{t_0})\leq c_d \H^{d-1}(\partial^*P_{t_0}\cap Q^{d}_{l,t_0}) < c_d\,
\end{align}
 Then, by choosing $l\gg 1$, \eqref{bound_on_Ld_of_slices} together with \eqref{the_wrong_isoperimetric} would imply that
\begin{align}\label{the_wrong_volume_slice}
\mathcal{L}^d(P_{t})  > \L^d(P_{t_0})-1  > 2l-c_d-1> l  \quad \text{for $\L^1$\EEE-a.e. }t\in (-1/2,1/2)\,.
\end{align}
Thus, by \eqref{the_wrong_volume_slice} and Fubini's theorem, we would get
\begin{equation*}
\L^{d+1}(P)=\int_{-1/2}^{1/2}\L^d(P_t)\,\mathrm{d}t> l\,,
\end{equation*}
contradicting the first assumption in \eqref{inductive hypothesis}. Therefore, indeed \eqref{P_t_0_is_the_small_one} holds true. By our inductive hypothesis and \eqref{minimality_condition} this yields
\begin{equation*}
\mathcal{L}^d(P_{t_0})\leq c_d \H^{d-1}(\partial^*P_{t_0}\cap Q^{d}_{l,t_0})\leq c_d\H^d(\partial^*P\cap Q_{l}^{d+1})\,.
\end{equation*}
The last inequality together with \eqref{difference_in volume_slices} implies that 
\begin{equation*}
\mathcal{L}^d(P_{t})\leq (c_d+1) \H^{d}(\partial^*P\cap Q^{d+1}_{l})\quad \text{for $\L^1$\EEE-a.e. }t\in(-1/2,1/2)\,.
\end{equation*}
Therefore, using Fubini's theorem again, we get
\begin{equation*}
\L^{d+1}(P)=\int_{-1/2}^{1/2}\L^d(P_t)\,\mathrm{d}t\leq (c_d+1) \H^{d}(\partial^*P\cap Q^{d+1}_{l})\,,
\end{equation*}
finishing the induction and hence the proof. \qedhere
\end{proof}

We proceed with two corollaries:  Corollary \ref{cor: dom1} and Corollary \ref{cor: dom2} describe how a long cuboid can be partitioned by a void set.  If the void set has relative perimeter  less than the area of the cross section $\sigma^{d-1}$, then there is a very large dominant component and some small components whose volume can be controlled by the relative perimeter of the void set. The same is true if the void set has relative perimeter between $ \sigma^{d-1}$ and $2 \sigma^{d-1}$ but small volume. If we drop the volume assumption, there may be two different large components - one consisting of the void set and a large complementary component.  If the void set has relative perimeter bigger than $2 \sigma^{d-1}$, then, even if the void set has small volume, it may separate the cuboid into two large complementary components. Some indicative   cases   are illustrated in Figure~\ref{fig:cases}: (a) Void set with perimeter less than $ \sigma^{d-1}$ or small volume and perimeter less than $2 \sigma^{d-1}$. (b) Void set with perimeter less than $2 \sigma^{d-1}$ with large volume. (c) Void set with perimeter bigger than $2 \sigma^{d-1}$ with   small  volume.

\begin{corollary}[Dominant component 1]\label{cor: dom1}
There exists $T_0 \in \N$ with the following property. Let  $l, \sigma\EEE>0$ with $l/ \sigma  \ge T_0$. Let $(P_j)_{j \ge 1}$ be a Caccioppoli partition of $S^l_{ \sigma}(x_0)$ with
\begin{align}\label{small_perimeter_condition-new}
\mathcal{H}^{d-1}\Big(\bigcup_{j \ge 1}\partial^*P_j\cap S^l_{ \sigma }(x_0)\Big) <  \sigma\EEE^{d-1} 
\end{align}
and $\L^d(P_1) \ge \L^d(P_j)$ for all $j \ge 2$.  Then,
\begin{align}\label{eq: DDD1}
\L^d( S^l_{ \sigma }(x_0) \setminus P_1  ) \le  C_{\rm iso} \sigma \H^{d-1}\big( \partial^*P_1 \cap S^l_{ \sigma\EEE }(x_0)\big)\, \quad \text{and} \quad  \L^d(P_1) > \frac{1}{2}\L^d(S^l_{ \sigma}(x_0))\,,  
\end{align}
where $C_{\rm iso}$ is the constant in \eqref{long_relative_isoperimetric_inequality}. 
\end{corollary}

\begin{proof} 
In view of \eqref{small_perimeter_condition-new}, \eqref{small_perimeter_condition} holds for each $P_j$ and Proposition \ref{nice_new_isoperimetric_inequality} is applicable for each $P_j$. To prove the statement, it  suffices to show that 
$$\L^d(P_1) > \L^d(S^l_{ \sigma \EEE}(x_0)\setminus P_1)\,.   $$
Assume by contradiction that this was false.  By $\L^d(P_1) \ge \L^d(P_j)$ for all $j \ge 2$, this would imply $\L^d(P_j) \le \L^d(S^l_{ \sigma \EEE}(x_0)\setminus P_j)$ for all $j \ge 1$.  But then we calculate using \eqref{long_relative_isoperimetric_inequality} and \eqref{small_perimeter_condition-new},
\begin{align*}
2l \sigma^{d-1} &=\L^d (S^l_{ \sigma}(x_0)) = \sum_{j \ge 1} \L^d (P_j) = \sum_{j \ge 1}  \min\{\L^d(P_j),\L^d(S^l_{ \sigma}(x_0)\setminus P_j)\} \\ & \le \sum_{j \ge 1}   C_{\rm iso}\sigma \H^{d-1}(\partial^* P_{ j}\cap S^l_{ \sigma}(x_0)) =  2C_{\rm iso}\sigma\H^{d-1}\big(\bigcup_{j\geq 1}\partial^* P_{j}\cap S^l_{\sigma}(x_0)\big) < 2C_{\rm iso}   \sigma^d\,, 
\end{align*}
where we also used the local structure of Caccioppoli partitions, see \cite[Theorem 4.17]{Ambrosio-Fusco-Pallara:2000}. \EEE 
By choosing $T_0\in \N$ large enough depending only on $C_{\rm iso}$ such that $ l/\sigma \ge T_0 > C_{\rm iso}$, this yields a contradiction. 
\end{proof}

\begin{corollary}[Dominant component 2]\label{cor: dom2}
There exists $T_0 \in \N$  with the following property. Let   $l, \sigma>0$ with $l/ \sigma \ge T_0$. Let $E\in \mathcal{P}(S^l_{ \sigma}(x_0))$  and let $(P_j)_{j \ge 1}$ be the connected components of  $S^l_{ \sigma\EEE}(x_0) \setminus {E}$ in the sense that $(P_j)_{j \ge 1} \cup \lbrace E\rbrace$ forms a Caccioppoli partition of $S^l_{ \sigma\EEE}(x_0)$ with 
\begin{align}\label{loc str}
\mathcal{H}^{d-1}\big( (\partial^* P_j   \setminus \partial^* E) \cap S^l_{ \sigma\EEE}(x_0) \big) = 0 \quad \text{ for all $j \ge 1$}\,.
\end{align}
 Suppose  that 
\begin{align}\label{small_perimeter_condition-new-new}
\mathcal{H}^{d-1}\Big(\bigcup_{j \ge 1}\partial^* P_j\cap S^l_{ \sigma\EEE}(x_0)\Big) < 2  \sigma\EEE^{d-1},  \quad \mathcal{L}^d(E) \le \frac{1}{4} \L^d( S^l_{ \sigma\EEE}(x_0)) 
\end{align}
and $\L^d(P_1) \ge \L^d(P_j)$ for all $j \ge 2$. Then,
\begin{align}\label{eq: DDD2}
\L^d( S^l_{ \sigma\EEE}(x_0) \setminus P_1  ) \le  C_{\rm iso} \sigma\EEE \H^{d-1}\Big( \bigcup_{j \ge 1}\partial^* P_j   \cap S^l_{ \sigma\EEE}(x_0)\Big) + \mathcal{L}^d(E) \, \ \text{and} \  \L^d(P_1) > \frac{1}{2}\L^d(S^l_{ \sigma\EEE}(x_0))\,,  
\end{align}
where $C_{\rm iso}$ is the constant in \eqref{long_relative_isoperimetric_inequality}. 
\end{corollary}

\begin{proof}
 As $\L^d(P_j) \le \L^d(P_1)$, we first observe that 
\begin{align}\label{two times} 
\L^d(P_j) \le \frac{1}{2}\L^d(S^l_{ \sigma}(x_0)) \quad \text{for all $j \ge 2$}\,.
\end{align}
 The sets $(\partial^* P_j \cap S^l_{ \sigma\EEE}(x_0) )_{j \ge 1}$ are pairwise disjoint up to $\mathcal{H}^{d-1}$-negligible sets by \eqref{loc str} and the local structure of Caccioppoli partitions, see \cite[Theorem 4.17]{Ambrosio-Fusco-Pallara:2000}. Therefore, by \eqref{small_perimeter_condition-new-new} we get
\begin{equation}\label{splitting_of_perimeter}
\sum_{j \ge 1} \mathcal{H}^{d-1}\Big(\partial^* P_j\cap S^l_{ \sigma\EEE}(x_0)\Big) =   \mathcal{H}^{d-1}\Big(\bigcup_{j \ge 1}\partial^* P_j\cap S^l_{ \sigma\EEE}(x_0)\Big)< 2  \sigma\EEE^{d-1}\,. 
\end{equation}
This implies that at least one of the following two cases holds: 
 $${\rm (a)} \ \  \mathcal{H}^{d-1}\Big(\partial^* P_1\cap S^l_{ \sigma\EEE}(x_0)\Big) <  \sigma^{d-1}, \quad \quad {\rm (b)} \ \  \sum_{j \ge 2} \mathcal{H}^{d-1}\Big(\partial^* P_j\cap S^l_{ \sigma}(x_0)\Big) <   \sigma^{d-1} \,.  $$
 We first assume that (a) holds. An  application of  Proposition \EEE \ref{nice_new_isoperimetric_inequality} yields
\begin{align}\label{eq 1:hopefully_last_isop}
\min\{\L^d(P_1),\L^d(S^l_{ \sigma\EEE}(x_0)\setminus P_1)\}\leq C_{\rm iso} \sigma\EEE \H^{d-1}(\partial^* P_1\cap S^l_{ \sigma\EEE}(x_0)) \le  C_{\rm iso}  \sigma\EEE^d \,.
\end{align}
Then, in the case $\L^d(P_1)\geq \L^d(S^l_{ \sigma\EEE}(x_0)\setminus P_1)$, we find
$${\L^d(S^l_{ \sigma}(x_0) \setminus P_1)   \le C_{\rm{iso}}\sigma\H^{d-1}(\partial^* P_1\cap S^l_{ \sigma}(x_0)) \EEE \le  C_{\rm iso}  \sigma^d\,.}  $$
This shows the first part of  \eqref{eq: DDD2}. The second part follows by choosing $T_0 \in \N$ large enough depending on $C_{\rm iso}$ noting that $\L^d(S^l_{ \sigma\EEE}(x_0))=2l\sigma^{d-1} \EEE\ge 2T_0  \sigma\EEE^d $.  

We show that the case $\L^d(P_1)< \L^d(S^l_{ \sigma\EEE}(x_0)\setminus P_1)$ leads to a contradiction. Indeed, if that was the case,  by \eqref{eq 1:hopefully_last_isop} we would have 
\begin{equation}\label{eq 2:hopefully_last_isop}
\L^d(P_j) \le \L^d(P_1) \le C_{\rm iso}   \sigma\EEE^d \quad \text{for all $j \ge 2$}\,.
\end{equation}
From this we derive that
\begin{align}\label{eq: small compi}
\L^d(P_j) \leq  C_{\rm iso}  \sigma \H^{d-1}(\partial^* P_j\cap S^l_{ \sigma\EEE}(x_0)) \quad \text{for all $j \ge 1$} \,. 
\end{align}
Indeed, for $j=1$ this is a consequence of \eqref{eq 1:hopefully_last_isop}. For $j \ge 2$ instead, if $\H^{d-1}(\partial^* P_j\cap S^l_{ \sigma\EEE}(x_0)) \ge  \sigma\EEE^{d-1}$, this follows from \eqref{eq 2:hopefully_last_isop}. If  $\H^{d-1}(\partial^* P_j\cap S^l_{ \sigma}(x_0)) <  \sigma\EEE^{d-1}$, \eqref{small_perimeter_condition} holds and the estimate follows from an application of Proposition~\ref{nice_new_isoperimetric_inequality} and \eqref{two times}.

Now, by \eqref{splitting_of_perimeter} and  \eqref{eq: small compi} we obtain
$$\L^d(S^l_{ \sigma}(x_0)\setminus \overline{E}) = \sum_{j \ge 1} \L^d(P_j) \le C_{\rm iso}   \sigma \sum_{j \ge 1} \H^{d-1}(\partial^* P_j\cap S^l_{ \sigma}(x_0)) \le 2C_{ d}{ \sigma}^d\,. $$
 Using  that  $\L^d(S^l_{ \sigma}(x_0)) \ge 2T_0  \sigma^d $, by choosing $T_0\in \N$ large enough, this would imply  $$\L^d(S^l_{ \sigma\EEE}(x_0)\setminus \overline{E}) < \frac{3}{4}\L^d( S^l_{ \sigma\EEE}(x_0))\,.$$ This however contradicts the fact that  $\L^d(E) \le \frac{1}{4} \L^d( S^l_{ \sigma\EEE}(x_0)) $, see \eqref{small_perimeter_condition-new-new}.

We are left with case (b). Here, we can again apply  Proposition \EEE \ref{nice_new_isoperimetric_inequality} on each $P_j$, $j \ge 2$, to find
\begin{align*}
\L^d(P_j) =  \min\{\L^d(P_j),\L^d(S^l_{ \sigma\EEE}(x_0)\setminus P_j)\}\leq C_{\rm iso} \sigma\EEE \H^{d-1}(\partial^* P_{ j\EEE}\cap S^l_{ \sigma\EEE}(x_0))\,,  
\end{align*} 
 where  the first identity follows from \eqref{two times}.  Now, by using \eqref{splitting_of_perimeter} we estimate
\begin{align*}
\L^d( S^l_{ \sigma}(x_0) \setminus P_1  ) &\le  \mathcal{L}^d(E) + \sum_{j \ge 2} \L^d(P_j) \le  \mathcal{L}^d(E) +  \sum_{j \ge 2} C_{\rm iso} \sigma\EEE\H^{d-1}\Big( \partial^* P_j   \cap S^l_{ \sigma\EEE}(x_0)\Big) \\
&  \le  C_{\rm iso} \sigma\H^{d-1}\Big( \bigcup_{j \ge 1}\partial^* P_j   \cap S^l_{ \sigma}(x_0)\Big) + \mathcal{L}^d(E)\,.
\end{align*}
This shows the first part of \eqref{eq: DDD2}. The second part again follows for some $T_0 \in \N$ large enough, using that $\mathcal{L}^d(E) \le \frac{1}{4} \L^d( S^l_{ \sigma}(x_0))$.
\end{proof}

\subsection{Local estimates and Sobolev extension on cuboids}\label{sec: locest}

In the following, we set up the necessary notation and definitions for the remainder of Section  \ref{preparatory_modifications}. We introduce the thickened void set and partition our reference domain $\Omega_{h, \rho}$ into cuboids, where we partition with respect to the surface area of the boundary of the thickened void.
We let $(v_h,E_h)_{h>0}$ be a sequence of admissible deformations and void sets in the  thin rod $\Omega_h$, where for convenience we use a continuum index $h$ in the notation for the sequences. Recalling \eqref{eq: G energy}, we suppose that 
\begin{equation}\label{uniform_energy_bound}
\sup_{h>0}  \mathcal{G}^{h}(v_h,E_h)< +\infty\,.
\end{equation}
We fix $0<\rho \le \rho_0:= 1 -(19/20)^{1/3}$ as in Proposition \ref{Lipschitz_replacement}. Recall the sequence $(\kappa_h)_{h>0}$ as in \eqref{rate_1_gamma_h}. For technical reasons, we need to assume that $(\kappa_h)_{h>0}$ converges to zero sufficiently fast. Therefore, we introduce
\begin{align}\label{barkappa}
\bar{\kappa}_h := \min \lbrace \kappa_h, h^2 \rbrace \,,
\end{align}
and observe that
\begin{align}\label{barkappa2}
 \mathcal{G}^{\bar{\kappa}_h}_{\rm surf}(E_h;\Omega_h)   \le  \mathcal{G}^{\kappa_h}_{\rm surf}(E_h;\Omega_h)\,.
\end{align}
Recall $T$ as introduced  before \eqref{overlapping_rectangles}.  From now on, we will tacitly assume that $T$ is chosen sufficiently large  such that Corollaries \ref{cor: dom1}--\ref{cor: dom2} are applicable. After possibly increasing $T$, we can assume that
\begin{align}\label{eq: T1}
T \ge 80 C_{\rm iso}\,,
\end{align}
where $C_{\rm iso} \ge 1$     is the constant in Proposition~\ref{nice_new_isoperimetric_inequality}. Let  $\eta_0 = \eta_{0}(\rho) \in (0,1)$ be the constant in Proposition \ref{prop:setmodification}. In view of \eqref{rate_1_gamma_h} and \eqref{barkappa}, we can  choose a sequence $(\eta_h)_{h >0} \subset (0,\eta_0)$ converging to zero sufficiently slow such that the constant $C_{\eta_h}$ in \eqref{eq: main rigitity}, applying   Theorem \ref{prop:rigidity} for $\rho$, $l=3Th$,  $\eta = \eta_h$, and $\gamma = \bar{\kappa}_h/h^2$, satisfies  
\begin{equation}\label{parameters_for_uniform_bounds}
\limsup_{h\to 0 } \,   C_{\eta_h}\Big(\frac{h^2}{\bar{\kappa}_h}\Big)^{\AAA 5\EEE} h^{2/5} <+\infty  \,.
\end{equation}
Then, by Proposition \ref{prop:setmodification} applied for $\rho$, $\eta = \eta_h$, and $\gamma=\bar{\kappa}_h/h^2$, for all $h>0$ we can find open sets $E^*_h$ with $E_{h} \subset E^*_h \subset \Omega_h$ such that  $\partial E^*_h\cap \Omega_{h}$ is a union of finitely many \NNN$C^2$\EEE-regular submanifolds and 
\begin{align}\label{eq: void-new}
\begin{split}
{\rm{(i)}} & \ \   h^{-3} \mathcal{L}^3(E^*_h\EEE\setminus E_{h}) \to 0,  \quad \quad  h^{-1}\dist_\mathcal{H}(E^*_h,E_h) \to 0 \quad \text{ as $h \to 0$}\,,
\\[2pt]
{\rm{(ii)}} & \ \    \liminf_{h \to 0 }    h^{-2}  \H^2(\partial E^*_h\cap \Omega_h)\leq    \liminf_{h \to 0 } \, h^{-2}\mathcal{G}^{\bar{\kappa}_h}_{\rm surf}(E_h;\Omega_h) \le \liminf_{h \to 0 } \, h^{-2}\mathcal{G}^{\kappa_h}_{\rm surf}(E_h;\Omega_h) \,.
\end{split}
\end{align}
Here, we used \eqref{eq: partition-new}$\rm{(i),(ii)}$, $ \eta_h \to 0$, \eqref{barkappa}, and that $h^{-2} \mathcal{G}_{\rm surf}^{\gamma h^2}(E_h;\Omega_h) = h^{-2}\mathcal{G}_{\rm surf}^{\overline{\kappa}_h}(E_h;\Omega_h) $ is uniformly bounded by  \eqref{eq: G energy},  \eqref{uniform_energy_bound}, and  \eqref{barkappa2}.  This is the sequence of sets in Proposition \ref{prop: 2nd main}  and we note that \eqref{eq: void-new} implies \eqref{eq: void-newXXX}.  In the rigidity estimate \eqref{eq: main rigitity}, the behavior of the deformation inside $E_h^*$  cannot be controlled. Thus, in a similar fashion to \eqref{initial_admissible_configurations}, for definiteness we can assume that the deformation is the identity inside $E_h^*$, i.e., we introduce the modification  $v_h^*\colon \Omega_h \to \R^3$ by
\begin{align}\label{eq: *mod}
v_h^*(x) := \begin{cases}  v_h(x) &  \text{ if } x \in \Omega_h \setminus E_h^*, \\
{\rm id} & \text{ if } x \in E_h^*\,.
\end{cases}
\end{align}
Note that by \eqref{eq: void-new} we get 
\begin{align}\label{eq: *mod2}
h^{-3} \mathcal{L}^3( \lbrace  v_h \neq {v}^*_h \rbrace ) \le h^{-3}  \mathcal{L}^3(E^*_h\setminus E_{h}) \to 0 \quad \text{ as $h \to 0$}\,.
\end{align}

Recall the definition of the $T$-cuboids in the family  $\mathcal{Q}_{h}$ in \eqref{overlapping_rectangles}. For $i=2,\ldots,N-1$, we also introduce the $3T$-cuboids by
\begin{equation*}
Q^3_{h}(i):=Q_{h}(i-1)\cup Q_{h}(i)\cup Q_{h}(i+1)\,.
\end{equation*}
Our idea is to apply Theorem \ref{prop:rigidity} for $U := Q^3_{h}(i)$. To this end, we also need the slightly smaller cuboids, defined by 
\begin{equation}\label{three_times_the_cuboid-rho}
Q^3_{h, \rho}(i):=  x_i + (1-\rho) \big(Q^3_h(i)  - x_i\big)  \subset \Omega_{h,\rho}\,,
\end{equation}
where $x_i = ((i-1/2)Th,0,0)$ denotes the center of the cuboid $Q_h(i)$. As we suppose that $ 0<\rho \le 1 -( 19/ 20)^{1/3}$, \eqref{three_times_the_cuboid-rho} implies that  
\begin{align}\label{eq: rho small}
\mathcal{L}^3(Q^3_{h, \rho}(i)) \ge \frac{ 19}{20} \mathcal{L}^3(Q^3_{h}(i))\,.
\end{align}
We also introduce the (small) parameter
\begin{align}\label{eq: T2}
\alpha = \Big(\frac{T}{10 c_T} \Big)^{2/3}\,,
\end{align}
where $ c_T:=c(T)>0$ denotes the constant of Theorem \ref{th: kornSBDsmall} applied on the cuboid $(0,3T) \times (-\frac{1}{2}, \frac{1}{2})^2$.      We will distinguish three classes of cuboids: first,  we consider    the family  of indices associated to \textit{good cuboids}, defined by
\begin{equation}\label{good_cuboids}
 I^h_{\rm g} :=\Big\{i=2,\dots, N-1\colon\  
\H^2(\partial E^*_h\cap  Q^3_{h,\rho}(i))\leq \alpha h^2
\Big\}\,.
\end{equation} 
This will be the family of cuboids for which Theorem \ref{th: kornSBDsmall} can be applied without introducing a too large exceptional set, cf.\  \eqref{eq: R2main}. Next, we   collect the family of \emph{bad cuboids} in the index set
\begin{align}\label{very_good_mildly_good_cuboids}
 I^h_{\rm b}  : & =\Big\{i\in \{2,\dots,N-1\}\setminus I^h_{\rm g}\colon \H^2\big(\partial E^*_h\cap  Q^3_{h,\rho}(i) \big) <(1-\rho)^2h^2 \Big\} \notag \\ 
& \  \cup \Big\{i\in \{2,\dots,N-1\}\setminus I^h_{\rm g}\colon   \H^2\big(\partial E^*_h\cap  Q^3_{h,\rho}(i)) <   2 (1-\rho)^2h^2\,,\ \mathcal{L}^3(E_h^* \cap Q^3_{h,\rho}(i)) \le    2C_{\rm iso}    h^3  \Big\}\,.
\end{align} 
For the  cuboids in $I^h_{\rm b}$, it might not be possible to apply Theorem \ref{th: kornSBDsmall}, but due to the relative isoperimetric inequality, see Corollaries \ref{cor: dom1}--\ref{cor: dom2}, we can still find a dominant component   which will allow us to compare rigid motions on adjacent cuboids via Lemma \ref{lemma: rigid motions}.  

On the remaining set of indices 
\begin{align}\label{eq: ugly}
 I^h_{\rm u} := \lbrace i = 1,\ldots, N \colon i \notin I^h_{\rm g} \cup I^h_{\rm b}\rbrace\,,
\end{align}
 corresponding to the so called \emph{ugly cuboids}, where the thickened void may cut through the rod and thus the behavior of $v_h^*$ cannot be controlled.  
 
\NNN Note that for each $i = 2 ,\ldots, N-1$, we have 
\begin{align}\label{ 5number}
\#\big\{j \in \{2 ,\ldots, N-1\}\colon Q^3_{h,\rho}(i) \cap Q^3_{h,\rho}(j)\neq \emptyset\big\}\leq 5\,.
\end{align}
 By \eqref{good_cuboids}\NNN--
 \eqref{ 5number}, \eqref{eq: void-new}(ii), and \eqref{uniform_energy_bound}, for $h>0$ small enough, we obtain
\begin{align*}
\alpha  \#\big( I^h_{\rm b} \cup  I^h_{\rm u} \big) \leq h^{-2}\sum_{i \in I^h_{\rm b} \cup  I^h_{\rm u} } \H^2(\partial E^*_h\cap  Q^3_{h,\rho}(i)) \leq Ch^{-2} \mathcal{H}^2(\partial E^*_h\cap \Omega_h) \leq Ch^{-2}\mathcal{G}^{\kappa_h}_{\rm surf}(E_h;\Omega_h)\leq C\,.
\end{align*}
Thus, we deduce that
\begin{equation}\label{cardinality_of_bad_cuboids}
\# \big( I^h_{\rm b} \cup  I^h_{\rm u} \big) = \# \big( \lbrace 1,\ldots, N\rbrace \setminus I^h_{\rm g}  \big)   \leq C
\end{equation} \EEE
for $C=C(\alpha) >0$, i.e., there are only a bounded number of indices in $I^h_{\rm b}\cup I^h_{\rm u}$ independently of $h$. 

We now formulate a local rigidity estimate on cuboids. As a final preparation, we introduce the localized elastic energy by
\begin{align}\label{eq: epsinot}
\eps_{i,h} := \int_{Q^3_{h}(i) \setminus \overline{E_h}} {\rm dist}^2(\nabla v_h, SO(3)) \, {\rm d}x\,, 
\end{align}
and use  \eqref{eq: nonlinear energy}(iv), \eqref{eq: G energy},  \eqref{uniform_energy_bound}, and \eqref{ 5number} to find 
\begin{equation}\label{localized_elastic_energy_on_good_cuboids}
\sum_{i=2}^{N-1}\eps_{i,h} \leq C \int_{\Omega_h\setminus \overline{E_h}}\mathrm{dist}^2(\nabla v_h,SO(3))\,\mathrm{d}x\leq  Ch^2 \epsilon_h \,.
\end{equation}

\begin{proposition}[Local rigidity estimate and Sobolev approximation]\label{summary_estimates_proposition}
Let $0 < \rho \le \rho_0$. There exists a constant $C = C(T)>0$ independent of $h$ such that for all $h>0$ and for every $i\in  I^h_{\rm g} \cup I^h_{\rm b}$ there exists a set of finite perimeter $D^3_{i,h}\subset  Q^3_{h,\rho}(i)  $ satisfying  
\begin{equation}\label{big_volume_of_dominant_set}
 \L^3\big(Q^3_{h,\rho}(i)\setminus D^3_{i,h} \big)  \leq Ch\H^{2}\big(\partial E_h^*\cap Q^3_{h}(i)\big), \quad  \L^3(   Q^3_{h}(i)  \setminus D^3_{i,h})\leq \frac{1}{5}\L^3(Q^3_{h }(i))\,,
\end{equation}	 
and a corresponding rigid motion $r_{i,h}(x):=R_{i,h}x+ b_{i,h}$, where $R_{i,h}\in SO(3)$ and $b_{i,h}\in \R^3$ with $|b_{i,h}| \le CM$ (see \eqref{initial_admissible_configurations} for the definition of $M$)    such that 
\begin{equation}\label{L2_gradient_estimates}
h^{-2}\int_{D^3_{i,h}}\big|v^*_h(x)-r_{i,h}(x)\big|^2\,\mathrm{d}x+\int_{D^3_{i,h}}\big|\nabla v_h^*(x)- R_{i,h}\big|^2\, \mathrm{d}x\leq
 {C}\eps_{i,h}^{9/10}\,,
\end{equation}
where $v_h^*$ is defined in \eqref{eq: *mod}.

Moreover, for $i \in I^h_{\rm g}$ there exists a Sobolev map   $z_{i,h}\in W^{1,2}(Q^3_{h,\rho}(i);\R^3)$  such that 
\begin{align}\label{properties_of_Sobolev_replacement}
\begin{split}
\rm{(i)}&\quad z_{i,h}\equiv v^*_h \quad \text{on }\ D^3_{i,h}\,,\\
\rm{(ii)}&\quad h^{-2}\int_{Q^3_{h,\rho}(i)}\big|z_{i,h}(x)-r_{i,h}(x)\big|^2\,\mathrm{d}x+\int_{Q^3_{h,\rho}(i)}\big|\nabla z_{i,h}(x)- R_{i,h}\big|^2\, \mathrm{d}x\leq
{C}\eps_{i,h}\,, \\
\rm{(iii)} & \quad \Vert z_{i,h} \Vert_{L^\infty(Q^3_{h,\rho}(i))} \le  CM\,.
\end{split}
\end{align}
\end{proposition}

In the following, we refer to $D^3_{i,h}$ as the \emph{dominant component} since $\L^3(Q^3_{h}(i)\setminus D^3_{i,h})$ is small, see \eqref{big_volume_of_dominant_set}. Accordingly, $r_{i,h}$ denotes the \emph{dominant rigid motion} which approximates $v_h^*$ in $Q^3_{h, \rho \EEE}(i)$. Note that $D^3_{i,h} \subset E_h^*$ is also possible which means that the void has a large volume inside $Q^3_{h,\rho}(i)$.

Observe that the estimate \eqref{L2_gradient_estimates} is actually better for $i \in I^h_{\rm g}$ as $\eps_{i,h}^{9/10}$ can be replaced by $\eps_{i,h}$. This follows directly from \eqref{properties_of_Sobolev_replacement}.  
This improvement is  possible due to the application of a Korn-Poincar\'e inequality in case of void sets with small surface measure, see Theorem \ref{th: kornSBDsmall}.  We also note that the choice of the exponent  $9/10$ is for definiteness only and can be enhanced to any exponent smaller than $1$, provided the sequence $(\kappa_h)_{h>0}$ in \eqref{rate_1_gamma_h}  is chosen appropriately. Before starting with the proof, let us recall that we use the notation $C>0$ for generic constants which are independent of $h,  \rho$ but may depend on the fixed parameters  $T, L  $.

 \begin{proof}[Proof of Proposition \ref{summary_estimates_proposition}] 
We use Theorem \ref{prop:rigidity} for $\rho>0$, $l =3Th$, $\gamma :=\bar{\kappa}_h/h^2$ with $\bar{\kappa}_h$ from \eqref{barkappa}, and the sequence $\eta_h\to 0$ such that \eqref{parameters_for_uniform_bounds} holds.  We apply the rigidity result to $v^*_h$ in the cuboid  $U := Q^3_{h}(i)$ for $i\in I^h_{\rm g} \cup I^h_{\rm b}$ and the compactly contained cuboid $\tilde{U} := Q^3_{h,\rho}(i)$. We denote by
\begin{equation*}
\mathcal P_{i,h}:=  \Big\{(P^j_{i,h})_j  \text{ the connected components of } Q^3_{h,\rho}(i)\setminus   \overline{E^*_h} \Big\} \cup \lbrace E_h^* \rbrace\,,
\end{equation*}
where   the enumeration is such that $\L^3(P^1_{i,h})$ is always maximal.  

Recall the definitions in \eqref{good_cuboids}--\eqref{very_good_mildly_good_cuboids}. In the case $i \in I^h_{\rm g}$ or  in the case that $i \in I^h_{\rm b}$ with $\H^2\big(\partial E^*_h\cap  Q^3_{h,\rho}(i) \big) < (1-\rho)^2h^2$ we can apply Corollary \ref{cor: dom1} on $Q^3_{h,\rho}(i)$ to obtain a dominant component. If $i \in I^h_{\rm b}$ with $\H^2\big(\partial E^*_h\cap  Q^3_{h,\rho}(i) \big) \ge(1-\rho)^2h^2$ instead, we can apply Corollary \ref{cor: dom2} on $Q^3_{h,\rho}(i)$, where we note that  the volume condition in \eqref{small_perimeter_condition-new-new}  is indeed satisfied by the definition of $I^h_{\rm b}$,  \eqref{eq: rho small},  and the fact that $T \ge 80 C_{\rm iso}$, see \eqref{eq: T1}. 

In both cases,  using that  $\bigcup_{j \ge 1} \partial  P^j_{i,h}  \cap Q^3_{h,\rho}(i) = \partial E^*_h\cap  Q^3_{h,\rho}(i)$, we get a dominant component $P^1_{i,h} \subset Q^3_{h,\rho}(i)$ which by \eqref{eq: DDD1} or \eqref{eq: DDD2}, respectively, and \eqref{very_good_mildly_good_cuboids} satisfy 
\begin{align}\label{eq: QQQ}
\L^3(Q^3_{h,\rho}(i)\setminus P^1_{i,h}) \le C_{\rm iso}h\H^{2}(\partial E_h^*\cap Q^3_{h}(i)) \le C_{\rm iso}h^3   
\end{align}
or 
\begin{align}\label{eq: QQQ2}
\L^3(Q^3_{h,\rho}(i)\setminus P^1_{i,h}) \le C_{\rm iso}h\H^{2}(\partial E_h^*\cap Q^3_{h}(i))  + \L^3(E_h^* \cap Q^3_{h,\rho}(i))  \le 2C_{\rm iso}h^3  + 2C_{\rm iso} h^3 = 4C_{\rm iso} h^3\,.  
\end{align}
Therefore, in both cases, we get by \eqref{eq: T1} and \eqref{eq: rho small} that 
\begin{align}\label{eq: first on domi} 
\L^3(Q^3_{h}(i)\setminus P^1_{i,h})  \le 4C_{\rm iso}h^3  + \L^3(Q^3_{h}(i) \setminus Q^3_{h,\rho}(i))\le \frac{1}{20} Th^3 + \frac{1}{20}  \mathcal{L}^3(Q^3_{h}(i)) \le  \frac{1}{10}  \mathcal{L}^3(Q^3_{h}(i))\,,
\end{align}
and moreover
\begin{align}\label{eq: first on domiXXX}
 \L^3\big(Q^3_{h,\rho}(i)\setminus P^1_{i,h} \big) \le Ch\H^{2}(\partial E_h^*\cap Q^3_{h}(i))\,.
\end{align}
Indeed, in the first case this directly follows from \eqref{eq: QQQ}. In the second case, it follows from \eqref{eq: QQQ2} and the fact $\H^2\big(\partial E^*_h\cap  Q^3_{h,\rho}(i) \big) \ge (1-\rho)^2h^2 \ge \tfrac{1}{4}h^2$ (as $ 0<\rho\le \tfrac{1}{2}$), where  the absolute constant $C>0$   needs to be chosen sufficiently large. 
 
We now distinguish the cases
$${\rm (a)} \ \ P^1_{i,h} = E_h^*, \quad {\rm (b)} \ \ P^1_{i,h} \cap  E_h^* = \emptyset, \  i \in I^h_{\rm b}, \quad  {\rm (c)} \ \ P^1_{i,h} \cap  E_h^* = \emptyset, \ i \in I^h_{\rm g}\,. $$

\emph{Case ${\rm(a):}$} If $P^1_{i,h} = E_h^*$, we define  ${D}^3_{i,h} := P^1_{i,h}$,  $R_{i,h}:= {\rm Id}$, $b_{i,h}:= 0$, and $z_{i,h}\in W^{1,2}(Q^3_{h,\rho}(i);\R^3)$ by $z_{i,h}:= {\rm id}$. Then, \eqref{big_volume_of_dominant_set} holds by \eqref{eq: first on domi}--\eqref{eq: first on domiXXX} and \eqref{L2_gradient_estimates}--\eqref{properties_of_Sobolev_replacement} are trivially satisfied (recall \eqref{eq: *mod}). Note that in this case we can define a Sobolev modification $z_{i,h}$ also if $i \in I^h_{\rm b}$. 

\emph{Preparations for ${\rm(b)}$ and ${\rm(c):}$} We proceed with preparations for (b) and (c). Suppose that $P^1_{i,h}\cap  E_h^* = \emptyset$. Then,   \eqref{eq: main rigitity}  in Theorem \ref{prop:rigidity} provides a rotation $R^1_{i,h} \in SO(3)$ and $b_{i,h}^1 \in \R^3$ such that 
\begin{align*}
\hspace{-1em}\begin{split}
{{\rm (i)}} & \ \  \int_{P^1_{i,h}}\big|{\rm sym}\big((R^1_{i,h})^T \nabla v^*_h-\mathrm{Id}\big)\big|^2\,\mathrm{d}x 
\leq C\big(1 +  C_{\eta_h} (h^{-2} \bar{\kappa}_h)^{-15/2}h^{-3}\eps_{i,h}\big)\eps_{i,h}\,,
\\
{{\rm (ii)}} & \ \     h^{-2}\int_{P^1_{i,h}}|v^*_h-(R^1_{i,h}x+b^1_{i,h})|^2\,\mathrm{d}x
+\int_{P^1_{i,h}}\big|(R^1_{i,h})^T \nabla  v^*_h-\mathrm{Id}\big|^2\,\mathrm{d}x
 \leq  C_{\eta_h}
 (h^{-2}\bar{\kappa}_h)^{\AAA-5\EEE}\eps_{i,h}\,,
\end{split}
\end{align*}
where \EEE we use the notation in \eqref{eq: epsinot} and  recall that we set $\gamma = \bar{\kappa}_h/h^2\AAA\in (0,1]\EEE$. By the choice of $(\eta_h)_{h >0}$ before \eqref{parameters_for_uniform_bounds}, $\limsup_{h \to 0} \epsilon_h h^{-2} < +\infty$, and \eqref{localized_elastic_energy_on_good_cuboids} we obtain
\begin{align}\label{eq: main rigidity}
\hspace{-1em}\begin{split}
{{\rm (i)}} & \ \    \int_{P^1_{i,h}}\big|{\rm sym}\big((R^1_{i,h})^T \nabla v^*_h-\mathrm{Id}\big)\big|^2\,\mathrm{d}x 
\leq C_0\eps_{i,h}\,,
\\
{{\rm (ii)}} & \ \    \Big( h^{-2}\int_{P^1_{i,h}}|v^*_h-(R^1_{i,h}x+b^1_{i,h})|^2\,\mathrm{d}x
+\int_{P^1_{i,h}}\big|(R^1_{i,h})^T \nabla  v^*_h-\mathrm{Id}\big|^2\,\mathrm{d}x
\Big) \leq C_0 h^{-2/5} \eps_{i,h}\,,
\end{split}
\end{align}
for a universal constant $C_0>0$.   We now show that
\begin{align}\label{eq: bbb}
|b^1_{i,h}| \le  CM\,,
\end{align}  
 for a constant $C>0$ that depends on $L,T>0$, but is independent of $h>0$. As $\|v_h\|_{L^\infty(\Omega_h)} \leq M$ for some $M \ge 1$, the triangle inequality implies
\begin{align}\label{eq: argu1}
\L^3(P^1_{i,h})|b_{i,h}^{1}|^2 \le  C \int_{P^1_{i,h}}|v_{h}^*(x)-  (R^1_{i,h}x+b^1_{i,h}) |^2\,\mathrm{d}x + C\L^3(P^1_{i,h}) \big( \|v_h\|_{L^\infty(\Omega_h)}^2 + ({\rm diam}(\Omega_h))^2   \big)\,.  
\end{align}
Thus, by  \eqref{localized_elastic_energy_on_good_cuboids}, \eqref{eq: first on domi}, and \eqref{eq: main rigidity}(ii)  we get 
\begin{align}\label{eq: argu2}
|b_{i,h}^{1}|^2 \le Ch^{-3} h^2  h^{-2/5} \eps_{i,h}  +   C(M^2 + C) \le C(M^2 + C)\,, 
\end{align}
 and thus $|b_{i,h}^{1}| \le CM$. After these preparations, we continue with the cases (b) and (c).

\emph{Case ${\rm(b):}$} We first suppose that  $i \in I^h_{\rm b}$. We set $D^3_{i,h}:= P_{i,h}^1$. Then, \eqref{big_volume_of_dominant_set} follows from \eqref{eq: first on domi}--\eqref{eq: first on domiXXX}  and   \eqref{L2_gradient_estimates} follows from \eqref{eq: main rigidity}(ii) by setting $R_{i,h}:= R_{i,h}^1$ and $b_{i,h} := b^1_{i,h}$, where we use $\eps_{i,h}^{1/10} \le C(h^4)^{1/10}$ by \eqref{localized_elastic_energy_on_good_cuboids} and $\limsup_{h \to 0} \epsilon_h h^{-2} < +\infty$. Observe that $|b_{i,h}| \le C M$ by \eqref{eq: bbb}.  

 \emph{Case ${\rm(c):}$} Let us now assume that $i \in I^h_{\rm g}$. We will use Theorem \ref{th: kornSBDsmall} to obtain a Sobolev function which satisfies    \eqref{properties_of_Sobolev_replacement}. First, let us introduce the function  $u_{i,h}\in SBV^2(Q^3_{h,\rho}(i);\R^3)$ by 
\begin{equation}\label{rotated_sbv_displacements}
u_{i,h}(x):=  \chi_{P^1_{i,h}}(x)\big[(R^1_{i,h})^T v_h^*(x)-x-(R^1_{i,h})^T b^1_{i,h}\big]\,,
\end{equation}
and note that by its definition  $J_{u_{i,h}}\subset \partial E^*_h\cap Q^3_{h,\rho}(i)$.    Now,  \eqref{localized_elastic_energy_on_good_cuboids}, \eqref{eq: main rigidity},  \eqref{rotated_sbv_displacements}, as well as $\limsup_{h \to 0} \epsilon_h h^{-2} < +\infty$ imply the bounds 
\begin{align}\label{cor_main rigidity}
\begin{split}
{\rm{(i)}} & \ \ \int_{Q^3_{h,\rho}(i)}|{\rm sym}\big(\nabla u_{i,h}\big)|^2\,\mathrm{d}x 
\leq C\eps_{i,h}\,,
\\
{\rm{(ii)}} & \ \  h^{-2}\int_{Q^3_{h,\rho}(i)} |u_{i,h}|^2\,\mathrm{d}x+\int_{Q^3_{h,\rho}(i)} |\nabla u_{i,h}|^2\,\mathrm{d}x \leq C\eps_{i,h}^{9/10}\,.
\end{split}
\end{align}
By the   scaling invariance   of  Theorem \ref{th: kornSBDsmall} we note that the constant therein 
is given by $ c_T$ appearing in  \eqref{eq: T2}.   Theorem \ref{th: kornSBDsmall}  for the map $u_{i,h}$ and the definition of $I^h_{\rm g}$ in \eqref{good_cuboids}   provide a set of finite perimeter $\omega_{i,h}\subset Q^3_{h,\rho}(i)$ satisfying
\begin{align}\label{set_of_finite_perimeter_control_area_volume}
\L^3(\omega_{i,h})\leq   c_T \big(\mathcal{H}^2(J_{u_{i,h}}) \big)^{3/2} \le  c_{T}\big(\mathcal{H}^2(\partial E^*_h\cap Q^3_{h,\rho}(i))\big)^{3/2}\leq  c_{T}\alpha^{1/2}h\mathcal{H}^2(\partial E^*_h\cap Q^3_{h,\rho}(i))\le  c_{T}\alpha^{3/2}h^3
\end{align}
and a Sobolev map   $\zeta_{i,h}\in W^{1,2}(Q^3_{h,\rho}(i);\R^3)$  such that  
\begin{align}\label{properties_of_Sobolev_replacementXXX}
\begin{split}
{\rm (i)}&\quad \zeta_{i,h}\equiv  u_{i,h} \quad \text{on }\ Q^3_{h,\rho}(i) \setminus  \omega_{i,h}\,,\\
{\rm (ii)} & \quad \Vert {\rm sym}(\nabla  \zeta_{i,h}) \Vert_{L^2(Q^3_{h,\rho}(i) ) } \le c_{T}\EEE \Vert {\rm sym}(\nabla  u_{i,h}) \Vert_{L^2(Q^3_{h,\rho}(i) ) }\,,\\
{\rm (iii)} & \quad \Vert \zeta_{i,h} \Vert_\infty \le \Vert  u_{i,h}  \Vert_\infty\le CM\EEE\,,
\end{split}
\end{align}
where the last estimate in $\rm{(iii)}$ follows from $\Vert v_h^* \Vert_\infty \le M$, \eqref{eq: bbb}, and the definition of $u_{i,h}$ in \eqref{rotated_sbv_displacements}.
In view of  \eqref{eq: T2} and \eqref{set_of_finite_perimeter_control_area_volume}, we get
$$\L^3(\omega_{i,h}) \le \frac{1}{10} Th^3 \le  \frac{1}{10}  \L^3(Q^3_h(i))\,.$$
We define the dominant component 
\begin{align}\label{eq: last dom def}
D^3_{i,h}:= P^1_{i,h} \setminus \omega_{i,h}\,
\end{align}
 and observe by \eqref{eq: first on domi}--\eqref{eq: first on domiXXX} and    \eqref{set_of_finite_perimeter_control_area_volume}  that  \eqref{big_volume_of_dominant_set} holds.

By the classical Korn's inequality  in $W^{1,2}$ we find   $A_{i,h}\in \R^{3\times 3}_{\mathrm{skew}}$  such that 
\begin{equation}\label{Korn_Poincare_inequality_first}
 \int_{Q^3_{h,\rho}(i)}|\nabla \zeta_{i,h}-A_{i,h}|^2 \,\mathrm{d}x
\leq C_{T}\int_{Q^3_{h,\rho}(i)}|\mathrm{sym}(\nabla u_{i,h})|^2 \,\mathrm{d}x \le CC_T \eps_{i,h}\,,
\end{equation}
where we used \eqref{properties_of_Sobolev_replacementXXX}(ii) and the last step follows from \eqref{cor_main rigidity}(i).  Therefore, setting  $$ z_{i,h}:= R^1_{i,h} \zeta_{i,h} + R^1_{i,h}{\rm id} + b_{i,h}^1 \in  W^{1,2}(Q^3_{h,\rho}(i);\R^3)$$ we observe by \eqref{rotated_sbv_displacements}, \eqref{properties_of_Sobolev_replacementXXX}(i), and \eqref{eq: last dom def} that $z_{i,h} \equiv v^*_h$ on $D^3_{i,h}$. This yields \eqref{properties_of_Sobolev_replacement}(i).  Moreover, \eqref{properties_of_Sobolev_replacement}(iii) follows from \eqref{properties_of_Sobolev_replacementXXX}(iii) and \eqref{eq: bbb}.

We proceed to show \eqref{properties_of_Sobolev_replacement}(ii). We start with the observation that      \eqref{Korn_Poincare_inequality_first} implies
\begin{equation}\label{ineq_for_skew_part}
\int_{Q^3_{h,\rho}(i)}|\nabla z_{i,h}-R^1_{i,h}(\mathrm{Id}+A_{i,h})|^2 \,\mathrm{d}x\leq C\eps_{i,h}\,.
\end{equation} 
Now, we need to replace $R^1_{i,h}(\mathrm{Id}+A_{i,h})$ suitably by a rotation. We  claim that there exists $R_{i,h}\in SO(3)$ such that
\begin{equation}\label{optimal_rotation_i}
\L^3(Q^3_{h,\rho}(i))|R^1_{i,h}(\mathrm{Id}+A_{i,h})-R_{i,h}|^2\leq C \eps_{i,h}\,.
\end{equation}
In order to show \eqref{optimal_rotation_i}, we argue as follows.  By \eqref{big_volume_of_dominant_set}  together with \eqref{properties_of_Sobolev_replacementXXX}(i), \eqref{cor_main rigidity}(ii), and \eqref{Korn_Poincare_inequality_first} we get 
\begin{align*}
\tfrac{4}{5} \L^3(Q_{h}(i)) |A_{i,h}|^2&\leq\L^3(D^3_{i,h})| A_{i,h}|^2= \int_{D^3_{i,h}}\big|\nabla u_{i,h}+ A_{i,h}-\nabla \zeta_{i,h}\big|^2\, \mathrm{d}x\\
&\leq 2\Big(\int_{D^3_{i,h}}|\nabla u_{i,h}|^2 \,\mathrm{d}x+\int_{D^3_{i,h}}|\nabla \zeta_{i,h} - A_{i,h}|^2\,\mathrm{d}x\Big)\leq C(\eps_{i,h}+\eps_{i,h}^{9/10})\,.
\end{align*}
By using the fact that $\eps_{i,h}\leq Ch^2 \epsilon_h \le Ch^4$, see \eqref{localized_elastic_energy_on_good_cuboids}, and $ \L^3(Q_{h}(i)) =   Th^3$, we obtain an estimate on $A_{i,h}$, namely 
\begin{equation*}
|A_{i,h}|^2\leq Ch^{-3}\eps_{i,h}^{9/10} \le Ch^{-7/5}\eps_{i,h}^{1/2}\,.
\end{equation*}
Therefore, the Taylor expansion (see  \cite[Equation (33)]{friesecke2002theorem})
$${\rm dist}(G,SO(3)) = |{\rm sym}(G) - {\rm Id}| + {\rm O}(|G - {\rm Id}|^2) $$
  allows us to estimate
\begin{align*}
\mathrm{dist}^2\big((\mathrm{Id}+A_{i,h}),SO(3)\big) \leq C|A_{i,h}|^4\leq Ch^{-14/5}\eps_{i,h}\,,
\end{align*}
i.e., there exists indeed $R_{i,h}\in SO(3)$ for which
\begin{equation*}
|R^1_{i,h}(\mathrm{Id}+A_{i,h})- R_{i,h}|^2  \leq Ch^{1/5}h^{-3}\eps_{i,h}\leq Ch^{-3}\eps_{i,h}  \le C (\L^3(Q^3_{h,\rho}(i)))^{-1}  \eps_{i,h}    \,.
\end{equation*}
This proves \eqref{optimal_rotation_i}. Hence, in view of \eqref{ineq_for_skew_part} and\eqref{optimal_rotation_i} we get 
\begin{equation*}
\int_{Q^3_{h,\rho}(i)}|\nabla z_{i,h}- R_{i,h}|^2\,\mathrm{d}x\leq C\eps_{i,h}\,,
\end{equation*}
which yields the second part of \eqref{properties_of_Sobolev_replacement}(ii). Finally, the Poincar\'e  inequality on $W^{1,2}(Q^3_{h,\rho}(i);\R^3)$ also implies that   there exists a vector $b_{i,h} \in \R^3$ such that the rigid motion 
 $r_{i,h}(x):=  R_{i,h}x+ b_{i,h}$ satisfies
\begin{equation*}
h^{-2} \int_{Q^3_{h,\rho}(i)}|z_{i,h}(x)- r_{i,h}(x)|^2\,\mathrm{d}x\leq C\eps_{i,h}\,.
\end{equation*}
This concludes the proof of  \eqref{properties_of_Sobolev_replacement}(ii). Eventually, in the case   $i\in I^h_{\rm g}$, we note  that estimate \eqref{L2_gradient_estimates} is an immediate  consequence of \eqref{properties_of_Sobolev_replacement}. Therefore, by repeating exactly  the argument in \eqref{eq: argu1}--\eqref{eq: argu2} with $b_{i,h}$ in place of $b_{i,h}^1$ we also get that $|b_{i,h}| \le CM$. This concludes the proof.     
\end{proof}

As a consequence, we can estimate the difference of two dominant rigid motions on adjacent cuboids.  

\begin{corollary}[Difference of rigid motions]\label{difference_rigid-motions}
Suppose $i,i+1 \in I^h_{\rm g} \cup I^h_{\rm b}$.  The rigid motions $r_{i,h}$,  $r_{i+1,h}$ given in  Proposition \ref{summary_estimates_proposition} satisfy 
 \begin{equation}\label{general_differences_rotations}
h^{-2}\|r_{i,h}-  r_{i+1,h}\|_{L^{\infty}(Q^3_{h}(i)  \cup Q^3_{h}(i+1) )   }^2+\big|R_{i,h}-R_{i+1,h}\big|^2\leq  Ch^{-3}(\eps_{i,h}^{9/10}+\eps_{i+1,h}^{9/10})\,.
\end{equation}
If $i,i+1 \in I^h_{\rm g}$, the better estimate
  \begin{equation}\label{general_differences_rotations2}
h^{-2}\|r_{i,h}-  r_{i+1,h}\|_{L^{\infty}(Q^3_{h}(i)  \cup Q^3_{h}(i+1) )   }^2+\big|R_{i,h}-R_{i+1,h}\big|^2\leq  Ch^{-3}(\eps_{i,h}+\eps_{i+1,h})
\end{equation}
holds. 
\end{corollary}

 \begin{proof}
 By \eqref{L2_gradient_estimates} and the triangle inequality we have 
\begin{align}\label{eq: the diff}
\int_{D^3_{i,h} \cap D^3_{i+1,h}}\big|r_{i,h}-r_{i+1,h}\big|^2\,\mathrm{d}x &\le 2 \int_{D^3_{i,h}}\big|v^*_h(x)-r_{i,h}(x)\big|^2\,\mathrm{d}x + 2\int_{D^3_{i+1,h}}\big|v^*_h(x)-r_{i+1,h}(x)\big|^2\,\mathrm{d}x \notag \\ 
& \le Ch^2(\eps_{i,h}^{9/10} + \eps_{i+1,h}^{9/10})\,.
\end{align}
Note that $\mathcal{L}^3(Q^3_{h}(i)  \cap Q^3_{h}(i+1)) = 2Th^3$ and  $\mathcal{L}^3( Q^3_{h}(j)  \setminus  D^3_{j,h}  ) \le \frac{3}{5}Th^3$ by \eqref{big_volume_of_dominant_set} for $j=i,i+1$. This yields 
$$\mathcal{L}^3(D^3_{i,h} \cap D^3_{i+1,h}) \ge \mathcal{L}^3(Q^3_{h}(i)  \cap Q^3_{h}(i+1)) -\mathcal{L}^3( Q^3_{h}(i)    \setminus  D^3_{i,h}  ) -\mathcal{L}^3( Q^3_{h}(i+1) \setminus  D^3_{i+1,h}  ) \ge  \frac{4}{5}Th^3\,.$$
 Moreover, we  observe that $Q^3_{h}(i)  \cup Q^3_{h}(i+1)$ is contained in a ball of radius $r = cTh$ for a universal constant $c>0$. This along with \eqref{eq: the diff} and Lemma \ref{lemma: rigid motions} shows  \eqref{general_differences_rotations}. Estimate \eqref{general_differences_rotations2} follows in the same fashion noting that with \eqref{properties_of_Sobolev_replacement}(ii) in place of \eqref{L2_gradient_estimates}  the exponents $9/10$ in \eqref{eq: the diff} can be replaced by $1$. 
\end{proof}

\subsection{Construction of \NNN blockwise \EEE Sobolev modifications and proofs of the propositions}\label{sec: global_constructions_proofs}

This subsection is devoted to the construction of $w_h$ and $R_h$, as well as to the proofs of  Proposition~\ref{Lipschitz_replacement} and Proposition \ref{prop: 2nd main}.

We start with the construction of $(w_h)_{h>0}$ and $(R_h)_{h>0}$. To this end, let $\psi^h \in C^\infty(\R^3)$ be a cut-off function satisfying $\psi^h(x) = \psi^h(x_1,0,0)$ for $x \in \R^3$, $0 \le \psi^h \le 1$, $\psi^h \equiv  1$  on $\lbrace x_1 \le -h \rbrace$, and  $\psi^h \equiv  0$  on $\lbrace x_1 \ge h \rbrace$ such that
\begin{align}\label{eq: h onfty}
\Vert \nabla \psi^h \Vert_\infty \le Ch^{-1}\,.
\end{align} Recalling \eqref{overlapping_rectangles}, for each $i=1,\ldots, N-1$, we set $\psi^h_{i,i+1}(x) = \psi^h(x- iTh e_1)$. For $i=1,\ldots,N-1$, we also define the sets 
\begin{align*}
\Psi^h_{i,i+1} :=\begin{cases}
\big((    iTh - h, iTh + h   ) \times \R^2\big) \cap \Omega_{h,\rho} & \text{ if } i,i+1 \in I^h_{\rm g},\\
\emptyset & \text{ else}\,,
\end{cases}
\end{align*}
i.e., $\lbrace \psi^h_{i,i+1} \in (0,1) \rbrace \cap \Omega_{h,\rho} \subset \Psi_{i,i+1}^h$, provided that $i,i+1 \in I^h_{\rm g}$. Note that the sets $(\Psi^h_{i,i+1})_i$ are pairwise disjoint by \eqref{eq: T1}. Moreover, since $ 0< \rho \le 1 -(19/20)^{1/3} \le 0.017$, we get that 
\begin{align}\label{eq: well defined}
 \big(Q_h(i) \cup  \Psi_{i-1,i}^h \cup  \Psi_{i,i+1}^h \big) \cap \Omega_{h,\rho} \subset Q^3_{h,\rho}(i) \quad \text{ for all $i=2,\ldots,N-1$},
 \end{align}
cf.\ \eqref{three_times_the_cuboid-rho}. To see this, by \eqref{eq: T1}, it suffices to note that $(T - \frac{3}{2} T\rho - 1)h \ge C_{\rm iso} (79-120\rho) h >0$.

We now define the sequences $(w_h)_{h>0}$  and $(R_h)_{h>0}$. First, we construct $w_h \in  SBV^2(\Omega_{h,\rho};\R^3)$ as follows. We set
\begin{align}\label{eq: construcc1}
w_h  :=  {\rm id} \quad\ \ \text{ on } \  Q_h(i) \cap \Omega_{h, \rho} \ \text{ for all } i \in I^h_{\rm u}\,,
\end{align}
and
\begin{align}\label{eq: construcc2}
w_h  := r_{i,h} \quad \text{ on }   Q_h(i) \cap \Omega_{h, \rho} \ \text{ for all } i \in I^h_{\rm b}\,,
\end{align}
where $ r_{i,h} $ denotes the rigid motion given in  \eqref{L2_gradient_estimates}. Eventually, recalling the definition of the Sobolev maps  $z_{i,h}\in W^{1,2}(Q^3_{h,\rho}(i);\R^3)$ in Proposition \ref{summary_estimates_proposition}, given $i \in I^h_{\rm g} \subset \lbrace 2,\ldots,N-1\rbrace$, and $x \in Q_h(i) \cap \Omega_{h, \rho}$,  we define
\begin{align}\label{eq: construcc3}
\begin{split}
w_h(x) := \begin{cases}  z_{i,h}(x) & \text{ if } x \in Q_h(i) \setminus (\Psi^h_{i-1,i} \cup \Psi^h_{i,i+1})\,,\\
\psi^h_{i-1,i}(x)z_{i-1,h}(x)  + (1-\psi^h_{i-1,i}(x))z_{i,h}(x)   & \text{ if } x \in Q_h(i) \cap \Psi^h_{i-1,i}\,, \\
\psi^h_{i,i+1}(x)z_{i,h}(x)  +  (1-\psi^h_{i,i+1}(x))z_{i+1,h}(x)  & \text{ if } x \in Q_h(i) \cap \Psi^h_{i,i+1}\,,
\end{cases}
\end{split}
\end{align}
where the second and third part of the definition might be empty if $\Psi^h_{i-1,i}  = \emptyset$ or $\Psi^h_{i,i+1}  = \emptyset$, respectively.  Note that this is well defined by \eqref{eq: well defined}, and the fact that $z_{i-1,h}$ or $z_{i+1,h}$ exist if $\Psi^h_{i-1,i}  \neq \emptyset$ or $\Psi^h_{i,i+1}  \neq \emptyset$, respectively.

In the absence of information on the second derivatives of $(v_h)_{h>0}$, we construct another sequence of functions $(R_h)_{h>0}$ with  $R_h \in SBV^2(\Omega_{h,\rho};\R^{3\times 3})$ which approximate $\nabla v_h$  and whose derivative can be controlled. We define 
\begin{align}\label{eq: construcc4}
R_h  :=  \nabla w_{h} \quad \text{ on }   Q_h(i) \cap \Omega_{h, \rho} \ \text{ for all } i \in I^h_{\rm b} \cup I^h_{\rm u}\,,
\end{align}
and for $x \in Q_h(i) \cap \Omega_{h, \rho}$, $i \in I^h_{\rm g}$, we let 
\begin{align}\label{eq: construcc5}
\begin{split}
R_h(x) := \begin{cases}  R_{i,h}
& \text{ if } x \in Q_h(i) \setminus (\Psi^h_{i-1,i} \cup \Psi^h_{i,i+1})\,,\\
\psi^h_{i-1,i}(x) R_{i-1,h}
+ (1-\psi^h_{i-1,i}(x)) R_{i,h}
& \text{ if } x \in Q_h(i) \cap \Psi^h_{i-1,i}\,, \\
\psi^h_{i,i+1}(x) R_{i,h}
+ (1-\psi^h_{i,i+1}(x)) R_{i+1,h}
& \text{ if } x \in Q_h(i) \cap \Psi^h_{i,i+1}\,,
\end{cases}
\end{split}
\end{align}
where $R_{i,h}$ are given by Proposition \ref{summary_estimates_proposition}.

Note that the construction implies that indeed $w_h \in  SBV^2(\Omega_{h,\rho};\R^3)$, $R_h \in  SBV^2(\Omega_{h,\rho};\R^{3\times 3})$, and the jump sets satisfy
\begin{align}\label{eq: whjum}
J_{w_h} \cup  J_{R_h} \subset \Omega_{h, \rho} \cap \bigcup_{i \in I^h_{\rm b} \cup I^h_{\rm u}} \partial Q_h(i)\,. 
\end{align}
We are now ready to give the proofs of the propositions.  

\begin{proof}[Proof of Proposition \ref{Lipschitz_replacement}]
First,  \eqref{almost_the same_and_control_of_energy}(i) follows from the construction \eqref{eq: construcc1}--\eqref{eq: construcc5}, the uniform control in  \eqref{properties_of_Sobolev_replacement}(iii), the bound $|b_{i,h}|\le CM$ for $i \in I^h_{\rm b}$,  and the fact that $SO(3) \subset \R^{3 \times 3}$ is compact. To see \eqref{almost_the same_and_control_of_energy}(ii), we use \eqref{eq: whjum} and the fact that $\# ( I^h_{\rm b} \cup  I^h_{\rm u} ) \le C$, see \eqref{cardinality_of_bad_cuboids}.

We proceed to show \eqref{almost_the same_and_control_of_energy}(iii),(iv) for $w_h$ and defer the proof for $R_h$ to the end.  Regarding \eqref{almost_the same_and_control_of_energy}(iii), we note that by the definition of $w_h$ and \eqref{properties_of_Sobolev_replacement}(i),
\begin{equation*}
\{w_h \neq  v_h\}\subset  B_h:= \bigcup_{i \in I^h_{\rm b} \cup I^h_{\rm u}} Q^3_{h}(i)  \cup \bigcup_{i\in I^h_{\rm g}} \big(Q^3_{h,\rho}(i)\setminus D^3_{i,h}    \big)  \cup \{v_h \neq  v^*_h\}\,.
\end{equation*}	
Since $\# ( I^h_{\rm b} \cup  I^h_{\rm u} ) \le C$, using also \eqref{uniform_energy_bound}, \eqref{eq: void-new}(ii), and  \eqref{big_volume_of_dominant_set}, we find
\begin{align}\label{measure_estimate_for_replacement}
\begin{split}
\L^3(\{w_h \neq  v_h\})   & \le \L^3(B_h) \leq CTh^3 + Ch\sum_{i\in I^h_{\rm g}} \H^{2}(\partial E_h^*\cap Q^3_{h}(i)) + \L^3(\{v_h \neq  v^*_h\}) \\
&\leq CTh^3+Ch \H^2(\partial E_h^*\cap \Omega_{h}) +  \L^3(\{v_h \neq  v^*_h\})\leq Ch^3 +   \L^3(\{v_h \neq  v^*_h\})\,,
\end{split}
\end{align}	
where  we also used the fact that each cuboid $Q^3_{h}(i)$ overlaps only with   neighboring ones, cf.\ \eqref{ 5number}. This along with \eqref{eq: *mod2} shows \eqref{almost_the same_and_control_of_energy}(iii) for $w_h$. 

Eventually, we show \eqref{almost_the same_and_control_of_energy}(iv) for $w_h$.  First,  by \eqref{eq: construcc1}--\eqref{eq: construcc2} we observe that
\begin{align}\label{eq: partis enough}
 \int_{\Omega_{h,\rho} } 
\mathrm{dist}^2(\nabla w_h,SO(3))\,\mathrm{d}x =  \sum_{i \in I^h_{\rm g}} \int_{ Q_h(i) \cap \Omega_{h,\rho} } 
\mathrm{dist}^2(\nabla w_h,SO(3))\,\mathrm{d}x\,. 
\end{align}
Moreover, by \eqref{eq: construcc3} and  \eqref{properties_of_Sobolev_replacement}(ii) we compute, for $ i \in I^h_{\rm g}$,
\begin{align*}
 \int_{ (Q_h(i) \cap \Omega_{h,\rho}) \setminus (\Psi^h_{i-1,i} \cup \Psi^h_{i,i+1}) } 
\mathrm{dist}^2(\nabla w_h,SO(3))\,\mathrm{d}x & \le  \int_{ Q^3_{h,\rho}(i) } 
\mathrm{dist}^2(\nabla z_{i,h},SO(3))\,\mathrm{d}x  \\
&\le  \int_{ Q^3_{h,\rho}(i) } 
|\nabla z_{i,h} - R_{i,h}|^2 \,\mathrm{d}x  \le C \eps_{i,h} \,. 
\end{align*}
This along with \eqref{localized_elastic_energy_on_good_cuboids} shows
\begin{align}\label{eq: the half}
\sum_{i \in I^h_{\rm g}} \int_{ (Q_h(i) \cap \Omega_{h,\rho}) \setminus (\Psi^h_{i-1,i} \cup \Psi^h_{i,i+1}) } 
\mathrm{dist}^2(\nabla w_h,SO(3))\,\mathrm{d}x \le C\sum_{i \in I^h_{\rm g}}  \eps_{i,h} \le Ch^2 \epsilon_h\,.
\end{align}
For all $\Psi^h_{i,i+1} \neq \emptyset$, i.e., $i,i+1 \in I^h_{\rm g}$, we estimate using \eqref{properties_of_Sobolev_replacement}(ii), \eqref{eq: well defined}, and \eqref{eq: construcc3}
\begin{align}\label{on_V_l_h_rho_i_elastic_energy}
\begin{split}
\hspace{-2em}\int_{ \Psi^h_{i,i+1}} \mathrm{dist}^2(\nabla w_h,SO(3))&= \int_{ \Psi^h_{i,i+1}} \mathrm{dist}^2\big(\nabla \big(z_{i+1,h}+\psi_{i,i+1}^h(z_{i,h}-z_{i+1,h})\big), SO(3)\big)\\
&\leq C\int_{ Q^3_{h,\rho}(i+1) }\mathrm{dist}^2(\nabla z_{i+1,h},SO(3))\\
&\ \ +C\int_{ \Psi^h_{i,i+1}}(|\nabla\psi_{i,i+1}^h|^2|z_{i,h}-z_{i+1,h}|^2+|\psi_{i,i+1}^h|^2|\nabla z_{i,h}-\nabla z_{i+1,h}|^2)\\
&\leq C\eps_{i+1,h}+C\int_{ \Psi^h_{i,i+1}}\big(h^{-2}|z_{i,h}-z_{i+1,h}|^2+|\nabla z_{i,h}-\nabla z_{i+1,h}|^2\big)\,,
\end{split}
\end{align}
where in the last step we used that $0 \le \psi_{i,i+1}^h \le 1$ and $\Vert \nabla \psi_{i,i+1}^h \Vert_{\infty} \le C h^{-1}$, see  \eqref{eq: h onfty}.
Since $\Psi^h_{i,i+1} \subset Q^3_{h,\rho}(i),  Q^3_{h,\rho}(i+1)$ by \eqref{eq: well defined},   we compute by \eqref{properties_of_Sobolev_replacement}(ii), \eqref{general_differences_rotations2}, and the triangle inequality 
\begin{align*}
\int_{ \Psi^h_{i,i+1}} |z_{i,h}-z_{i+1,h}|^2 \, {\rm d}x  &\le  C \int_{ Q^3_{h,\rho}(i)} |z_{i,h}-r_{i,h}|^2 \, {\rm d}x + C\int_{Q^3_{h,\rho}(i+1)} |z_{i+1,h}-r_{i+1,h}|^2 \, {\rm d}x \\ &  \ \ \ \  + C\int_{Q^3_{h,\rho}(i) \cup Q^3_{h,\rho}(i+1)} |r_{i,h}-r_{i+1,h}|^2 \, {\rm d}x  \le Ch^2(\eps_{i,h} + \eps_{i+1,h})\,,
\end{align*}
where we also used that $\mathcal{L}^3(Q^3_{h,\rho}(i) \cup Q^3_{h,\rho}(i+1)) \le CTh^3$. In a similar fashion, \eqref{properties_of_Sobolev_replacement}(ii) and  \eqref{general_differences_rotations2}  also imply 
\begin{align*}
\int_{ \Psi^h_{i,i+1}} |\nabla z_{i,h}- \nabla z_{i+1,h}|^2 \, {\rm d}x  &\le  C \int_{ Q^3_{h,\rho}(i)} |\nabla z_{i,h}-R_{i,h}|^2 \, {\rm d}x + C\int_{Q^3_{h,\rho}(i+1)} |\nabla z_{i+1,h}-R_{i+1,h}|^2 \, {\rm d}x  \\ &  \ \ \ \  + Ch^3|R_{i,h} - R_{i+1,h}|^2    \le C(\eps_{i,h} + \eps_{i+1,h})\,.
\end{align*}
The last two estimates along with \eqref{on_V_l_h_rho_i_elastic_energy} and \eqref{localized_elastic_energy_on_good_cuboids} show
$${\sum_{i,i+1 \in I^h_{\rm g}}\int_{ \Psi^h_{i,i+1}} \mathrm{dist}^2(\nabla w_h,SO(3))  \,\mathrm{d}x \le C \EEE \sum_{i=2}^{N-1} \eps_{i,h} \le C h^2 \epsilon_h\,.}$$
This together with \eqref{eq: partis enough}--\eqref{eq: the half} \EEE concludes the proof of the first inequality in  \eqref{almost_the same_and_control_of_energy}(iv). 

We now continue with the proof of  \eqref{almost_the same_and_control_of_energy}(iii) for $R_h$. By \eqref{measure_estimate_for_replacement} the set $B_h$ satisfies $h^{-2}\L^3(B_h) \to 0 $ as $h \to 0$.    To obtain an estimate on the complement  $\Omega_{h,\rho} \setminus B_h$, we recall the definition of  $w_h$ and $R_h$ in  \eqref{eq: construcc3} and \eqref{eq: construcc5}, respectively. In particular, as $z_{i,h} = z_{i+1,h}$ on $\Psi_{i,i+1}^h \setminus B_h$, see \eqref{properties_of_Sobolev_replacement}(ii),  we have $ \nabla v_h=\nabla w_h =  \nabla z_{i,h} = \nabla z_{i+1,h} =  \psi^h_{i,i+1}\nabla z_{i,h}+(1-\psi^h_{i,i+1})\nabla z_{i+1,h}$ on $\Psi_{i,i+1}^h \setminus B_h$.  Therefore, one can check that 
\begin{align*}
\int_{\Omega_{h,\rho} \setminus B_h} |\nabla v_h - R_h |^2    \, {\rm d} x & = \int_{\Omega_{h,\rho} \setminus B_h} | \nabla w_h - R_h |^2    \, {\rm d} x \le C\sum_{i \in I^h_{\rm g}} \int_{Q^3_{h,\rho}(i)} | \nabla z_{i,h} - R_{i,h} |^2   \, {\rm d} x \le Ch^2\epsilon_h\,,
\end{align*} 
where the last step follows from \eqref{properties_of_Sobolev_replacement}(ii) and  \eqref{localized_elastic_energy_on_good_cuboids}. Let $(\theta_h)_{h>0} \subset (0,+\infty)$ be an infinitesimal sequence such that  $\theta_h \epsilon_h^{-1/2} \to \infty$.  Then, by using  $h^{-2}\L^3(B_h) \to 0 $ we compute 
\begin{align*}
h^{-2} \L^3\big(\Omega_{h,\rho} \cap\big\{|\nabla v_h - R_h |>\theta_h\big\}\big) & \le h^{-2} \L^3\big((\Omega_{h,\rho} \setminus B_h ) \cap\big\{|\nabla v_h - R_h |>\theta_h\big\}\big)  +   h^{-2} \L^3(B_h) \\ 
& \le  h^{-2} \theta_h^{-2} \int_{\Omega_{h,\rho} \setminus B_h} |\nabla v_h - R_h |^2    \, {\rm d} x  +   h^{-2} \L^3(B_h) \\
& \le  C \epsilon_h \theta_h^{-2}    +   h^{-2} \L^3(B_h) \to 0\,.
\end{align*}
 This shows  \eqref{almost_the same_and_control_of_energy}(iii) for $R_h$. 
 
We finally show the second estimate in \eqref{almost_the same_and_control_of_energy}(iv). We observe $\nabla R_h = 0$ on $\Omega_{h,\rho} \setminus \bigcup_{i} \Psi_{i,i+1}^h$ (recall \eqref{eq: construcc4}, \eqref{eq: construcc5}). For all $\Psi^h_{i,i+1} \neq \emptyset$, i.e., $i,i+1 \in I^h_{\rm g}$,  by \eqref{general_differences_rotations2} and the fact that $\L^3(\Psi^h_{i,i+1}) \le 2h^3$, we compute
\begin{align*}
\int_{ \Psi^h_{i,i+1}} |\nabla R_h|^2 \, {\rm d}x  &= \int_{ \Psi^h_{i,i+1}} |\nabla \psi^h_{i,i+1}|^2 | R_{i,h} - R_{i+1,h}|^2 \, {\rm d}x    \\
& \le Ch^{-2}  \int_{ \Psi^h_{i,i+1}}  | R_{i,h} - R_{i+1,h}|^2 \, {\rm d}x  \le Ch^{-2}  (\eps_{i,h}+\eps_{i+1,h})\,,
\end{align*}
 where we again used that  $\Vert \nabla \psi_{i,i+1}^h \Vert \le C h^{-1}$.  Summing over all $i\in I^h_{\rm g}$ and using \eqref{localized_elastic_energy_on_good_cuboids} we conclude 
$$\sum_{i \in I^h_{\rm g}}\int_{ \Psi^h_{i,i+1}} |\nabla R_h|^2 \, {\rm d}x  \le C\epsilon_h\,. $$
This concludes the proof of the second estimate in \eqref{almost_the same_and_control_of_energy}(iv).  
\end{proof}

We close this section with the proof of Proposition \ref{prop: 2nd main}.
\begin{proof}[Proof of Proposition \ref{prop: 2nd main}]  
Fix the \NNN stripe \EEE $S_{h}^{2l}(x) $ as in the statement. We start the proof with the following observation:  Assumption \eqref{eq: small jump/vol} implies that
\begin{align}\label{iii}
i \in I^h_{\rm g} \cup I^h_{\rm b} \ \ \text{ for each  } \ \ i\in I^h  := \lbrace i\colon Q^3_h(i) \cap  S^l_{h,\rho}(x) \neq \emptyset\rbrace\,.
\end{align}
Indeed, since   $l \ge 6Th$, we first get that  $Q^3_h(i) \subset   S^{3l/2}_{h}(x)$ for all $i \in I^h  $, see \eqref{eq: D not}--\eqref{eq: D not2}. Then, in view of  \eqref{eq: small jump/vol}(i), we find  $\H^2\big(\partial E^*_h\cap  Q^3_{h,\rho}(i)) <   2 (1-\rho)^2h^2$ for all $i \in  I^h $. Thus, recalling \eqref{good_cuboids}--\eqref{very_good_mildly_good_cuboids}, to show that $i \in I^h_{\rm g} \cup I^h_{\rm b}$  it suffices to check that 
\begin{align}\label{eq: EEE}
\L^3\big( E^*_h \cap S^{3l/2}_{h}(x) \big) \le 2C_{\rm iso}h^3\,.
\end{align}
Thus, let us check \eqref{eq: EEE}.  
Again using  \eqref{eq: small jump/vol}(i) we can find a partition $S^{2l}_{h}(x) = U_- \cup U_+$ (up to a set of negligible $\L^3$- measure) with disjoint open cuboids $U_-,U_+$ such that 
\begin{align}\label{eq: ole} 
 \H^2\big(  \partial E^*_h \cap U_\pm \big) <  h^2\,.
\end{align}
($\pm$ is a shorthand for $+$ or $-$.) If  $U_\pm \cap  S^{3l/2}_{h
}(x) = \emptyset$, the set is irrelevant for showing \eqref{eq: EEE}. Thus, we  suppose that  $U_\pm \cap   S^{3l/2}_{ h
}(x) \neq \emptyset$. Then, Proposition \ref{nice_new_isoperimetric_inequality} implies 
$$\min\{\L^3(E_h^* \cap U_\pm),\L^3(U_\pm\setminus E_h^*)\}\leq C_{\rm iso}h \H^{2}(\partial E_h^*\cap U_\pm)\,. $$
(Note that the proposition is applicable as $U_\pm$ contains at least one $T$-cuboid.) By \eqref{eq: ole} we get  
\begin{align}\label{YYY}
\min\{\L^3(E_h^*\cap U_\pm),\L^3(U_\pm\setminus E_h^*)\}\leq C_{\rm iso}h^3\,. 
\end{align}
For $U_\pm \cap   S^{3l/2}_{ h 
}(x) \neq \emptyset$ we have $\L^3(U_\pm) \ge \frac{1}{8}\L^3(S^{2l}_{h}(x))$. Then, necessarily, $\L^3(E_h^*\cap U_\pm) \le \L^3(U_\pm\setminus E_h^*)$, since otherwise by \eqref{eq: small jump/vol}(ii)
$${\frac{1}{8} \L^3( S^{2l}_{h}(x) )\EEE \le \L^3(U_\pm)  \le \L^3(U_\pm\setminus E_h^*) + \L^3(U_\pm\cap E_h^*) \le C_{\rm iso}h^3 + \frac{1}{9}\L^3(S^{2l}_{h}(x))  \,.} $$
This yields a contradiction, since 
$$\L^3(S^{2l}_{h}(x)) = 4lh^2 \ge 24Th^3 \ge 1920 
C_{\rm iso} h^3\,,$$ 
see \eqref{eq: T1}. Thus, using \eqref{YYY} we conclude 
$$\L^3\big(E_h^* \cap   S^{3l/2}_{h\EEE
}(x) \big) \le \L^3\big(E_h^*\cap U_- \cap   S^{3l/2}_{h\EEE
}(x)\big) + \L^3\big(E_h^*\cap U_+ \cap   S^{3l/2}_{h\EEE
}(x)\big) \le 2C_{\rm iso}h^3.$$
This shows \eqref{eq: EEE}, and thus \eqref{iii} holds.

We are now ready to verify \eqref{eq: small jump}. In view of \eqref{iii}, \eqref{eq: whjum} yields that $$J_{w_h} \cap S^l_{h,\rho}(x) \subset \bigcup_{i \in I^h_{\rm b}} \partial Q_h(i)  \cap \Omega_{h,\rho}\,.$$ 
Note that in  each cuboid $  Q_h(i)$, $i \in I^h_{\rm b}$, the trace   ${\rm tr}(w_h)$ on $\partial Q_h(i)  \cap \Omega_{h,\rho}$ coincides with $r_{i,h}$ (recall \eqref{eq: construcc2}). If the neighboring cuboid is good, i.e., $i-1 \in I^h_{\rm g}$ or $i+1 \in I^h_{\rm g}$, by \eqref{properties_of_Sobolev_replacement}(ii),  \eqref{eq: construcc3},  and the trace estimate on $Q_h(j)$ (with its scaling), for $j=i-1,i+1$, the trace ${\rm tr}(w_h)$  satisfies
$${\int_{\partial Q_h(i) \cap \partial Q_h(j)  \cap \Omega_{h,\rho} } |{\rm tr}(w_h) - r_{j,h}|^{2} \, {\rm d}\H^2 \le Ch\int_{Q_h(j)} \big(h^{-2}|w_h - r_{j,h}|^{2} + |\nabla w_h-R_{j,h}|^2\big)\,\mathrm{d}x \leq   Ch\eps_{j,h}\,.}$$
Thus, by a discrete H\"older's inequality we find for each $i \in I^h_{\rm b}$ that 
$${\gamma_i := \sum_{j = i-1,i+1} \int_{\partial Q_h(i) \cap \partial Q_h(j)  \cap \Omega_{h,\rho} } |{\rm tr}(w_h) - r_{j,h}|^{1/2} \, {\rm d}\H^2 \le C(h^2)^{3/4} h^{1/4}\big((\eps_{i-1,h})^{1/4} +  (\eps_{i+1,h})^{1/4} \big)\,.} $$
Now, by \eqref{general_differences_rotations} we compute
\begin{align*} 
\int_{J_{w_h}\cap S^l_{h,\rho}(x)}|[w_h]|^{1/2}\,\mathrm{d}\H^2&\leq  \sum_{i\in I^h_{\rm b}}\int_{\partial Q_h(i) \cap \Omega_{h,\rho}}|(w_h)_+-(w_h)_- |^{1/2}\,\mathrm{d}\H^2\nonumber\\
&\leq C\sum_{i\in I^h_{\rm b}}\big(  h^2  \|r_{i-1,h}-r_{i,h}\|^{1/2}_{L^{\infty}(Q^3_{h}(i))}  +  h^2  \|r_{i,h}-r_{i+1,h}\|^{1/2}_{L^{\infty}(Q^3_{h}(i))}   + \gamma_i\big)    \nonumber\\
&\leq  C 
\sum_{i\in I^h_{\rm b}} \Big(h^{7/4}\big(\eps^{9/40}_{i-1,h} + \eps^{9/40}_{i,h}+\eps^{9/40}_{i+1,h}\big)+h^{7/4}\big((\eps_{i-1,h})^{1/4} +  (\eps_{i+1,h})^{1/4} \big) \Big) \,.
\end{align*}
A discrete H\"older's  inequality along with $\# I^h_{\rm b}\leq C$   (recall \eqref{cardinality_of_bad_cuboids})  and \eqref{localized_elastic_energy_on_good_cuboids} then yields
 \begin{align*}
 {\int_{J_{w_h}\cap S^l_{h,\rho}(x)}|[w_h]|^{1/2}\,\mathrm{d}\H^2 \le  C h^{7/4} \Big( \sum_{i=2}^{N-1} \eps_{i,h} \Big)^{9/40}   \le Ch^{11/5} \epsilon_h^{9/40}\,.}
  \end{align*}
This shows the first part of \eqref{eq: small jump}. For the second part, we compute in a similar fashion, again using \eqref{localized_elastic_energy_on_good_cuboids} and \eqref{general_differences_rotations}, and the construction in \eqref{eq: construcc4}--\eqref{eq: construcc5}
\begin{align*} 
\int_{J_{R_h}\cap S^l_{h,\rho}(x)}|[R_h]|^{1/2}\,\mathrm{d}\H^2&\leq  \sum_{i\in I^h_{\rm b}}\int_{\partial Q_h(i) \cap \Omega_{h,\rho}}|(R_h)_+-(R_h)_- |^{1/2}\,\mathrm{d}\H^2\nonumber\\
&\leq C h^2 \sum_{i\in I^h_{\rm b}}\big(  |R_{i-1,h}-R_{i,h}|^{1/2}  +  |R_{i,h}-R_{i+1,h}|^{1/2} \big)   \nonumber\\
&\leq  C h^{2-3/4}
\sum_{i\in I^h_{\rm b}} \big(\eps^{9/40}_{i-1,h} + \eps^{9/40}_{i,h}+\eps^{9/40}_{i+1,h}\big) \le Ch^{17/10}\epsilon_h^{9/40}\,. 
\end{align*}
This along with   $\limsup_{h \to 0} \epsilon_h h^{-2} < +\infty$ concludes the proof. 
\end{proof}

\begin{remark}[Variant of  Proposition \ref{prop: 2nd main}]\label{rem: the lambda case}
{\normalfont

Let us briefly comment on Remark \ref{rem: main rem}. The proof of (i) basically follows from the previous proof by noting that the assumption \eqref{eq: small jump/vol}(i) with  $1$ in place of $2$ (on the right hand side) excludes the presence of ugly cuboids. For (ii), we also follow the estimates above and observe that, in the worst case $\epsilon_h \sim h^2$, the integral over jump heights $|[w_h]|^{1-\beta}$ and  $|[R_h]|^{1-\beta}$ can be estimated by $h^2 h^{13(1-\beta)/10}$ and $h^2 h^{3(1-\beta)/10}$, respectively. }
\end{remark}

\section{Compactness}\label{compactness}

 This section is devoted to the proof of   Theorem \ref{compactness_thm}. We again use  the continuum  subscript $h>0$ instead of the sequential subscript notation $(h_j)_{j\in \N}$ for convenience. We first recall the relevant result from the Sobolev setting.

\begin{lemma}[Compactness in the Sobolev setting]\label{lemma: mora}
Let $\Omega_{\ell_1,\ell_2} := (0,\ell_1) \times (-\ell_2,\ell_2)^2$ for $\ell_1, \ell_2 >0$, and let $(\tilde{w}_h)_{h >0}$ be a bounded sequence in  $W^{1,2}(\Omega_{\ell_1,\ell_2};\R^3)$ such that 
$$\limsup_{h \to 0} \frac{1}{h^2}\int_{ \Omega_{\ell_1,\ell_2}} {\rm dist}^2( \nabla_h \tilde{w}_h,SO(3)) \, {\rm d} x  + \Vert \tilde{w}_h\Vert_{W^{1,2}(\Omega_{\ell_1,\ell_2})}  \le C_0 < +\infty\,. $$
Then, there exist $\bar{y} \in W^{2,2}(\Omega_{\ell_1,\ell_2};\R^3)$ and $\bar{d_2},\bar{d_3} \in W^{1,2}(\Omega_{\ell_1,\ell_2};\R^3)$, all independent of $(x_2,x_3)$, and  a subsequence (not relabeled) such that
\begin{align}\label{eq: weki}
\tilde{w}_h \rightharpoonup \bar{y} \quad \text{weakly in $W^{1,2}(\Omega_{\ell_1,\ell_2};\R^3)$}, \quad \quad  \nabla_h \tilde{w}_h \to \big(\bar{y}_{,1}\big\vert \, \bar{d_2}\big \vert \, \bar{d_3}\big) \  \text{strongly  in $L^2(\Omega_{\ell_1,\ell_2};\R^{3 \times 3})$}\,.
\end{align}
Moreover,  $(\bar{y}_{,1}\big\vert \, \bar{d_2}\big \vert \, \bar{d_3}) \in SO(3)$ a.e.\ in $\Omega_{\ell_1,\ell_2}$, and
\begin{align}\label{eq: constant bound}
\Vert \bar{y} \Vert_{W^{2,2}(\Omega_{\ell_1,\ell_2})} + \Vert \bar{d_2} \Vert_{W^{1,2}(\Omega_{\ell_1,\ell_2})} + \Vert \bar{d_3} \Vert_{W^{1,2}(\Omega_{\ell_1,\ell_2})}  \le C
\end{align}
for a constant $C>0$ only depending on $C_0$.
\end{lemma}

For the proof we refer to \cite[Theorem 2.1]{Mora}. The weak convergence to $\bar{y}$ has not been mentioned in the original statement, but follows directly from weak compactness.  A simple argument shows that $\bar{y}$ is indeed independent of $(x_2,x_3)$. In fact,  \eqref{eq: weki} and \eqref{rescaled_deformation_gradient}   yield $\bar{y}_{,i} \equiv 0$ for $i=2,3$.  Property \eqref{eq: constant bound} has not been stated explicitly in \cite[Theorem 2.1]{Mora}, \NNN but is a consequence of the proof of \cite[Theorem 2.1]{Mora} (cf.\ the last estimate therein) and the equivalent characterization of Sobolev spaces via finite differences, see e.g.~\cite[Theorem 1.36]{Dacorogna}. \EEE We now proceed with the proof of our compactness result.

\begin{proof}[Proof of Theorem \ref{compactness_thm}] 
By the energy bound \eqref{uniform_rescaled_energy_bound} and  \eqref{eq: newenergy} we have that
\begin{equation}\label{energy_bound_in_the_order_1_domain}
h^{-2}\int_{\Omega\setminus \overline{V_h}} W(\nabla_hy_h(x))\,\mathrm{d}x+\int_{\partial V_h\cap \Omega}\big|\big(\nu^1_{V_h}(z), h^{-1}\nu^2_{V_h}(z), h^{-1}\nu^3_{V_h}(z)\big)\big|\, \mathrm{d}\H^2(z)\leq C\,,
\end{equation}
where $\nu_{V_h}(z):=(\nu^1_{V_h}(z),\nu^2_{V_h}(z), \nu^3_{V_h}(z))$ denotes the outward pointing unit normal to $\partial V_h\cap \Omega$ at the point $z$. Note that \eqref{energy_bound_in_the_order_1_domain} implies  
$$\sup_{h>0}\big(\L^3(V_h)+\H^2(\partial V_h\cap \Omega)\big)\leq C\,.$$
Therefore, by a compactness result for sets of finite perimeter (see \cite[Theorem 3.39]{Ambrosio-Fusco-Pallara:2000}), there exists $V\in \mathcal{P}(\Omega)$ such that,  up to a non-relabeled subsequence, we have
 \begin{equation*}
\chi_{V_h}\to \chi_{V}\ \ \text{in\ } L^1(\Omega)\,.
\end{equation*}
By Reshetnyak's lower semicontinuity theorem (cf.\ \cite[Theorem 2.38]{Ambrosio-Fusco-Pallara:2000}) applied to the lower semicontinuous, positively 1-homogeneous, convex function $\phi:\S^2\to [0,+\infty)$ with $\phi(\nu):=|(0,\nu^2,\nu^3)|$, we get, using again  \eqref{energy_bound_in_the_order_1_domain}, that 
\begin{equation*}
\int_{\partial^*V\cap \Omega} |(0,\nu_{V}^2,\nu_{V}^3)| \,\mathrm{d}\H^2\leq \liminf_{h\to 0}\int_{\partial V_h\cap \Omega}|(0,\nu_{V_h}^2,\nu_{V_h}^3)|\,\mathrm{d}\H^2\leq C\liminf_{h\to 0}h=0\,, 
\end{equation*}
where $\nu_{V}$ denotes the  measure-theoretic outer unit normal to $\partial^* V$. This implies $\nu^2_{V}(x)=\nu^3_{V}(x)=0$ for $\H^2$-a.e.\ $x\in \partial^*V\cap\Omega$. We denote by $I:= \lbrace x_1 \in (0,L)\colon \mathcal{H}^2((\lbrace x_1 \rbrace \times \R^2) \cap V)>0\rbrace$ the measure-theoretic projection of $V$ onto the $x_1$-axis. The previous argument shows that indeed $V=V_I:=I\times (-1/2,1/2)^2$ for some $I\in \mathcal{P}(0,L)$, up to a set of negligible $\mathcal{L}^3$-measure. This proves \eqref{compactness_properties}(i). From now on (and also in the next sections) we will without loss of generality consider the representative in the $\L^1$-equivalence class of $I$ which consists of a finite union of open subintervals of $(0,L)$, without mentioning it further.

We now proceed with the compactness for the deformations. Denote by $(v_h,E_h)_{h>0}$ the sequence related to $(y_h,V_h)_{h>0}$ via \eqref{order_1_domain}--\eqref{from_v_to_y}. We first derive a compactness result for the \NNN blockwise \EEE Sobolev  modifications constructed in Proposition \ref{Lipschitz_replacement} and then we will show \eqref{compactness_properties}(ii),(iii) for the sequence $(y_h)_{h>0}$   afterwards. To this end, we fix $\rho >0$ sufficiently small. We apply Proposition \ref{Lipschitz_replacement} on $(v_h,E_h)_{h>0}$ and $\epsilon_h:= h^2$ to find a sequence  $(w_h)_{h>0}$. (Observe that \eqref{uniform_rescaled_energy_bound} implies \eqref{Lipschitz_replacement_energy_bound}.) Then, we consider the sequence $(\tilde w_h)_{h>0}\subset SBV^2(\Omega_{1,\rho};\R^3)$  defined by 
\begin{equation}\label{tilde_y_h}
\tilde w_h(x):=\tilde w_h(x_1,x_2,x_3):=w_h(x_1,hx_2,hx_3)\,.
\end{equation}
After passing to a subsequence, we may assume that there exists $n \in 2\N$ such that the sets  $\mathcal{Q}_{v_h}$ in  \eqref{almost_the same_and_control_of_energy}(ii) satisfy $\#\mathcal{Q}_{v_h} = n/2$ for all $h >0$. Thus, by \eqref{almost_the same_and_control_of_energy}(ii) we get $(x^h_i)_{i=1}^n \subset (0,L)$ such that
$$J_{w_h} \subset  \Omega_{h,\rho}  \cap   \bigcup_{i=1}^n  \, \big(  \lbrace x^h_i \rbrace \times \R^2\big)\,. $$
Up to a further subsequence, we can suppose that for each $i=1,\ldots,n$ we have $x_i^h \to x_i$ as $h \to 0$ for suitable $x_i \in [0,L]$. Thus, fixing an arbitrary  $\delta>0$ and defining the set 
\begin{equation}\label{Omega_delta_set}
\Omega_\rho^\delta:=\Omega_{1,\rho} \setminus \bigcup_{i=1}^n \, \Big([x_i-\delta,x_i+\delta]\times \R^2\Big)\,,
\end{equation}
we find that
$$\tilde{w}_h|_{\Omega_\rho^\delta} \in W^{1,2}(\Omega_\rho^\delta;\R^3) \quad \text{for all $h>0$ small enough}. $$
A change of variables together with \eqref{almost_the same_and_control_of_energy}(iv) (for $\epsilon_h:=h^2$) imply that
\begin{align}\label{elastic_energy_of_tilde_y_h}
h^{-2}\int_{\Omega_{\rho}^\delta}\mathrm{dist}^2(\nabla_h \tilde w_h,SO(3))\,\mathrm{d}x& \le h^{-4}\int_{\Omega_{h,\rho}}\mathrm{dist}^2(\nabla w_h,SO(3))\,\mathrm{d}x\leq C\,, 
\end{align}
for a constant $C>0$  independent of $h$, $\delta$, and $\rho$. This along with  \eqref{almost_the same_and_control_of_energy}(i)  and \eqref{tilde_y_h}  shows that the sequence $(\tilde w_h)_{h>0}$ is equibounded in $W^{1,2}(\Omega_\rho^\delta;\R^3)$, i.e.,
\begin{align}\label{elastic_energy_of_tilde_y_hXXX}
\Vert \tilde w_h \Vert_{W^{1,2}(\Omega_\rho^\delta)} \le C\,.
\end{align}
Therefore,   by Lemma \ref{lemma: mora} applied to   the sequence $(\tilde w_h)_{h>0}$    on the connected components of the fixed domain $\Omega_{\rho}^\delta$, we obtain a map $y_\rho^\delta\in W^{2,2}(\Omega_\rho^\delta;\R^3)$, and $(d_2)_{\rho}^\delta, (d_3)_{\rho}^\delta\in W^{1,2}(\Omega_\rho^\delta;\R^3)$, all independent of the $(x_2,x_3)$-coordinates, such  that
\begin{equation}\label{W12_convergence_of_y_rho_delta}
\tilde w_{h}\rightharpoonup y_{\rho}^\delta \ \  \text{weakly in } W^{1,2}(\Omega_\rho^\delta;\R^3) \quad  \text{and } \quad \nabla_h\tilde w_{h}\to R_\rho^\delta \ \ \text{strongly in } L^2(\Omega_\rho^\delta;\R^{3\times 3})\,,
\end{equation}
where
\begin{equation}\label{R_rho_delta_def}
R_\rho^\delta:=\big({y_{\rho,1}^\delta}\big\vert \, {(d_2)_{\rho}^\delta}\big \vert \, {(d_3)_{\rho}^\delta}\big)\in SO(3) \text{ a.e. in } \Omega_{\rho}^\delta\,.
\end{equation} 
Moreover, by \eqref{eq: constant bound}, \eqref{elastic_energy_of_tilde_y_h}, and \eqref{elastic_energy_of_tilde_y_hXXX} we have
\begin{align}\label{eq: constant bound2}
\Vert y_\rho^\delta \Vert_{W^{2,2}(\Omega_\rho^\delta)} + \Vert (d_2)_{\rho}^\delta \Vert_{W^{1,2}(\Omega_\rho^\delta)} + \Vert (d_3)_{\rho}^\delta \Vert_{W^{1,2}(\Omega_\rho^\delta)}  \le C
\end{align}
for a constant $C>0$ independent of $\rho$ and $\delta$.

We now replace $\tilde{w}_h$ by $y_h$ in \eqref{W12_convergence_of_y_rho_delta}. By \eqref{from_v_to_y}, \eqref{tilde_y_h}, a scaling argument,   and   \eqref{almost_the same_and_control_of_energy}(iii)  we have that 
\begin{align}\label{measure_of_set_of_difference_in_Omega_delta_rho}
\L^3\big(\{x\in \Omega_{\rho}^\delta\colon y_h(x)\neq \tilde w_h(x)\}\big)& \le  h^{-2} \L^3\big(\{x\in \Omega_{h,\rho}\colon v_h(x)\neq w_h(x)\}\big) \to  0 
\end{align}
as $h \to 0$. Thus, from \eqref{W12_convergence_of_y_rho_delta} and \eqref{measure_of_set_of_difference_in_Omega_delta_rho} we deduce that 
\begin{equation}\label{W12_convergence_of_y_rho_delta2}
y_{h}\to y_{\rho}^\delta \ \  \text{ in measure on } \Omega_\rho^\delta \quad  \text{and } \quad \nabla_h y_{h}\to R_\rho^\delta \ \ \text{ in measure on } \Omega_\rho^\delta\,.
\end{equation}
Now, by  \eqref{R_rho_delta_def}, \eqref{eq: constant bound2}, \eqref{W12_convergence_of_y_rho_delta2}, and a  monotonicity argument for $\rho \to 0$ and $\delta \to 0$ we find  $(y\vert \, d_2\vert \, d_3) \in \NNN(P\text{-}W^{2,2}\times P\text{-}W^{1,2}\times P\text{-}W^{1,2}\EEE)((0,L);\R^{3\times 3}\big)$,  such that the corresponding functions  $\bar y, \bar d_2, \bar d_3  \colon \Omega \to \R^3$ defined in \eqref{eq: convention1}  satisfy 
\begin{align}\label{eq: identi} 
\bar y  = y_{\rho}^\delta  \quad \text{ on $\Omega_\rho^\delta$}, \quad \quad \bar R:=(\bar y_{,1}\vert \, \bar d_2\vert \, \bar d_3) = R_{\rho}^\delta  \quad \text{ on $\Omega_\rho^\delta$}\,,  
\end{align}
and
\begin{equation}\label{W12_convergence_of_y_rho_delta3}
y_{h}\to \bar y  \ \  \text{ in measure on } \Omega \quad  \text{and } \quad \nabla_h y_{h}\to \bar R  \ \ \text{ in measure on } \Omega\,.
\end{equation}
Property \eqref{R_rho_delta_def} also implies that $({\bar y\EEE_{,1}}\big\vert \, {\bar d_2}\big \vert \, {\bar d_3}) \in SO(3)$ a.e.\ in $\Omega$, i.e., $(y\vert \, d_2 \vert \,  d_3)\in SBV^2_{\mathrm{isom}}(0,L)$, see \eqref{SBV_2_isom} and the convention introduced right after it. The measure convergence $y_{h}\to \bar y $ in \eqref{W12_convergence_of_y_rho_delta3} together with  $\|y_h\|_{L^\infty(\Omega)} \le  M$ shows \eqref{compactness_properties}(ii).  By \eqref{eq: nonlinear energy}(iv), \eqref{energy_bound_in_the_order_1_domain}, and the fact that $\nabla_hy_h=\mathrm{Id}$ on $V_h$ we have that
\begin{equation*}
\sup_{h>0} \int_{\Omega}|\nabla_hy_h|^2 \, {\rm d}x <+\infty\,.
\end{equation*}
A compactness argument and \eqref{W12_convergence_of_y_rho_delta3} show that $\nabla_h y_{h} \rightharpoonup \bar R $ weakly in $L^2(\Omega;\R^{3\times 3})$. Recalling that  $\chi_{V_{h}}\to \chi_{V_I}$ in $L^1(\Omega)$ by \eqref{compactness_properties}(i), the proof of  \eqref{compactness_properties}(iii) is concluded. 

It finally remains to show that $((y\vert \, d_2\vert \, d_3),I)\in \mathcal{A}$. In fact, we have $\|y\|_{L^\infty(\Omega)}\leq M$ by  $\|y_h\|_{L^\infty(\Omega)} \le  M$  for all $ h>0$ and \eqref{W12_convergence_of_y_rho_delta3}.  Moreover, the second part of \eqref{W12_convergence_of_y_rho_delta3}, \eqref{compactness_properties}(i), and the fact that $y=T_h({\rm id})$, $\nabla_h y_{h} = {\rm Id}$ on $V_h$ (see \eqref{admissible_configurations_h_level}) show that  $\bar R\equiv {\rm Id}$ on $V_I$ and thus  $y(x_1) \equiv x_1$ and  $(y_{,1}\vert \, d_2\vert \, d_3) \equiv \mathrm{Id}$ on $I$.  As above we have already seen that $(y\vert \, d_2 \vert \,  d_3)\in SBV^2_{\mathrm{isom}}(0,L)$, the proof is concluded. 
 \end{proof}

\section{The $\Gamma$-liminf inequality}\label{gamma_liminf}

This section is devoted to the proof of the lower bound \eqref{gamma_liminf_inequality} of Theorem \ref{main_gamma_convergence_thm}, which is split into proving the lower bound for the bulk and the surface part of the energy separately.  

Recalling Definition \ref{type_of_convergence}, we consider a sequence $(y_{h}, V_{h})_{h>0}$ and $((y\vert \, d_2\vert \, d_3), I)\in \mathcal{A}$ such that $(y_{h}, V_{h}) \overset{\tau}{\longrightarrow}((y\vert \, d_2\vert \, d_3), I)$, i.e., \eqref{compactness_properties}(i)--(iii) hold true.  

Regarding the lower bound for the elastic energy,  we will use the following result from the purely elastic case. For its formulation, recall the definition of the elastic part of the limiting energy, as introduced in \eqref{limiting_one_dimensional_energy}--\eqref{linear_elasticity_quadratic_form}, and the convention after \eqref{SBV_2_isom}.

\begin{lemma}[Lower bound in the Sobolev setting]\label{lemma: liminf Mora}
Let $\Omega_{\ell_1,\ell_2} := (0,\ell_1) \times (-\ell_2,\ell_2)^2$ for $\ell_1, \ell_2 >0$, and let $(\tilde{w}_h)_{h >0}$ be a  sequence in  $W^{1,2}(\Omega_{\ell_1,\ell_2};\R^3)$ such that 
\begin{align}\label{eq: weki-new}
\tilde{w}_h \rightharpoonup \bar y  \quad \text{weakly in $W^{1,2}(\Omega_{\ell_1,\ell_2};\R^3)$}, \ \  \nabla_h \tilde{w}_h \to  \bar R \EEE =  \big({{\bar y\EEE}_{,1}}\big\vert \, {\bar d}_2\big \vert \, {\bar  d}_3\big) \  \text{strongly  in $L^2(\Omega_{\ell_1,\ell_2};\R^{3 \times 3})$}\,,
\end{align}
where ${\bar y}$, ${\bar d}_2$, and ${\bar d}_3$ are independent of $(x_2,x_3)$. Then, there exists a sequence of piecewise constant functions $\mathcal{R}_h \colon \Omega_{\ell_1,\ell_2} \to SO(3)$ and a limiting function $G \in L^2(\Omega_{\ell_1,\ell_2};\R^{3 \times 3})$ such that
\begin{align}\label{eq: mora-liminfy-eq}
\begin{split}
{\rm (i)} \quad & G_h := \frac{\mathcal{R}_h^T \nabla_h \tilde{w}_h - \mathrm{Id}}{h} \rightharpoonup G \quad \text{weakly in $L^2(\Omega_{\ell_1,\ell_2};\R^{3\times 3})$}\,, \\ 
{\rm (ii)} \quad & \liminf_{h \to 0} \frac{1}{h^2} \int_{\Omega_{\ell_1,\ell_2} } W(\nabla_h \tilde{w}_h) \, {\rm d}x \ge \frac{1}{2} \int_{\Omega_{\ell_1,\ell_2}} \mathcal{Q}_3(G) \, {\rm d}x\,,\\
{\rm (iii)} \quad &  \frac{1}{2} \int_{\Omega_{\ell_1,\ell_2}} \mathcal{Q}_3(G) \, {\rm d}x \ge \frac{1}{2}(2\ell_2)^{4}\int_{(0,\ell_1)}\Q_2({R}^T{R}_{,1})\,\mathrm{d}x_1\,,
\end{split}
\end{align}
where $R$ is 
defined via \eqref{SBV_2_isom}--\eqref{eq: convention2}.
\end{lemma}

The proof can be found in \cite[Theorem 3.1(i)]{Mora}. In particular, we refer to  \cite[\NNN(3.4)--(3.6), (3.16), and Remark 3.2\EEE]{Mora}. The result is stated there only for cross sections with area $1$, corresponding to $\ell_2= \frac{1}{2}$. However, a standard scaling argument shows that 
\begin{align*}
(2\ell_2)^{4} \Q_2(A) = \Q_2^{\ell_2}(A):= \min_{a\in W^{1,2}\left(\left(-\ell_2,\ell_2\right)^2;\R^3\right)} \int_{\left(-\ell_2,\ell_2\right)^2}\Q_3\left( A \begin{pmatrix}
0 \\
x_{2} \\
x_{3}
\end{pmatrix}\Bigg\vert \, \alpha_{,2}\Bigg \vert \, \alpha_{,3}\right)\,\mathrm{d}x_2\, \mathrm{d}x_3\,,
\end{align*} 
where $\Q_2(A)$ is given by \eqref{def_Q_2}. This implies \eqref{eq: mora-liminfy-eq}(iii) in the present form. 

\NNN The issue in our framework is that Lemma \ref{lemma: liminf Mora} cannot be applied directly since the sequence $(y_h)_{h>0}$ is only in $W^{1,2}(\Omega \setminus \overline{E}_h;\R^3)$ and the geometry of $(E_h)_{h>0}$ cannot be controlled a priori. Therefore, as in the proof of  Theorem \ref{compactness_thm}, we will use the modification $(w_h)_{h>0}$ constructed in Proposition \ref{Lipschitz_replacement}. The advantage here is that, due to \eqref{almost_the same_and_control_of_energy}(ii), the geometry of the jump set of $(w_h)_{h>0}$ is well controlled in the sense that it is   contained in the vertical faces of finitely many cuboids. Therefore, far from these cuboids, we can reduce to the Sobolev setting.  \EEE

\begin{lemma}\label{elastic_lower_bound}
Suppose that  $(y_{h}, V_{h})\overset{\tau}{\longrightarrow}((y\vert \, d_2\vert \, d_3), I)$ for some  $((y\vert \, d_2\vert \, d_3),I)\in \mathcal{A}$. Then, 
\begin{equation}\label{lower_bound}
\liminf_{h\to 0} \Big(h^{-2}\int_{\Omega\setminus \overline{V_h} }W(\nabla_hy_h)\,\mathrm{d}x\Big)\geq 
\frac{1}{2}\int_{(0,L)\setminus I}\Q_2(R^TR_{,1})\,\mathrm{d}x_1\,.
\end{equation} 
\end{lemma}

\begin{proof}
We apply Proposition \ref{Lipschitz_replacement} for $\rho >0$ small and  $\epsilon_h:=h^2$ and the sequence $(v_h, E_h)_{h>0}$ related to $(y_h,V_h)_{h>0}$ via 
\eqref{order_1_domain}--\eqref{from_v_to_y}.  Here, we note that it is not restrictive to assume that $(\E^{h}(y_{h}, V_{h}))_{h>0}$ is bounded, and thus \eqref{Lipschitz_replacement_energy_bound} holds. We denote the resulting sequence by $(w_h)_{h>0}$, and as in the proof of the compactness result, we consider the sequence $(\tilde w_h)_{h>0}\subset SBV^2(\Omega_{1,\rho};\R^3)$  defined by
\begin{equation}\label{eq: rescaliii}
\tilde w_h(x):=\tilde w_h(x_1,x_2,x_3):=w_h(x_1,hx_2,hx_3)\,.
\end{equation}
Similarly to the reasoning in the proof of Theorem \ref{compactness_thm}, see  \eqref{Omega_delta_set}, \eqref{W12_convergence_of_y_rho_delta}, and \eqref{eq: identi}, we can  define a set $\Omega_\rho^\delta$ for $\rho,\delta>0$ with $\mathcal{L}^3(\Omega \setminus \Omega_\rho^\delta) \to 0$ as $\rho,\delta \to 0$ such that
$$\tilde{w}_h|_{\Omega_\rho^\delta} \in W^{1,2}(\Omega_\rho^\delta;\R^3) \quad \text{for all $h>0$ small enough}\,,$$
and 
\begin{equation*}
\tilde w_{h}\rightharpoonup \bar y \ \  \text{weakly in } W^{1,2}(\Omega_\rho^\delta;\R^3) \quad  \text{and } \quad \nabla_h\tilde w_{h}\to \bar R  \ \ \text{strongly in } L^2(\Omega_\rho^\delta;\R^{3\times 3})\,.
\end{equation*} 
This means that \eqref{eq: weki-new} is satisfied  and we can thus apply Lemma \ref {lemma: liminf Mora} on each connected component of $\Omega_\rho^\delta$ to find corresponding $G_h$  and $G$ such that \eqref{eq: mora-liminfy-eq} holds on the set $\Omega_\rho^\delta$. The main part of the proof will consist in confirming that \eqref{eq: mora-liminfy-eq}(ii) also holds with $y_h$ in place of $\tilde{w}_h$. Then, the liminf inequality follows from \eqref{eq: mora-liminfy-eq}(iii).
 
To show \eqref{eq: mora-liminfy-eq}(ii) for $y_h$ in place of $\tilde{w}_h$, we will perform a by now classical linearization argument which we sketch here for convenience:  we consider a sequence of positive numbers $(\lambda_h)_{h>0}\subset (0,\infty)$ with
\begin{equation}\label{lambda_h_sequence}
\lambda_h\to \infty\,,\  h\lambda_h  \to 0 \ \ \text{as } h\to 0\,,
\end{equation}
and define
\begin{equation}\label{set_of_big_gradient}
\Theta_h:= \lbrace x\in \Omega_{\rho}^\delta \colon \tilde{w}_h(x) = y_h(x) \rbrace \cap\{x\in \Omega_{\rho}^\delta\colon |G_h(x)|\leq \lambda_h\}\,.
\end{equation}
Note that $\mathcal{L}^3(\lbrace \tilde{w}_h \neq y_h \rbrace) \to 0$  by \eqref{almost_the same_and_control_of_energy}(iii) and a scaling argument. This together with  the fact that  $\Vert  G_h  \Vert_{L^2(\Omega_{\rho}^\delta)} \le C$, see \eqref{eq: mora-liminfy-eq}(i), $\lambda_h \to + \infty$, and Chebyshev's inequality implies that 
\begin{align}\label{eq: cheby}
\mathcal{L}^3(\Omega_{\rho}^\delta \setminus \Theta_h) \to 0 \quad \text{as $h\to 0$}\,. 
\end{align}
This yields $\chi_{\Theta_h}\to 1$ boundedly in measure in $\Omega_\rho^\delta$ as $h\to 0$. By   \eqref{admissible_configurations_h_level}, $W(\mathrm{Id})=0$, $W \ge 0$, and by the definition of $\Theta_h$ we get 
\begin{align*}
\liminf_{h \to 0} \Big(h^{-2}\int_{\Omega \setminus \overline{V_h}} W(\nabla_hy_h)\, {\rm d}x\Big) & = \liminf_{h \to 0} \Big(h^{-2}\int_{\Omega} W(\nabla_hy_h)\, {\rm d}x\Big)  \\
& \ge  \liminf_{h \to 0} \Big(h^{-2}\int_{\Omega_\rho^\delta} \chi_{\Theta_h} W(\nabla_h \tilde{w}_h)\, {\rm d}x\Big)\,.
\end{align*}
By the regularity and the structural hypotheses on $W$ (recall \eqref{eq: nonlinear energy})  we get 
$$W({\rm Id}+F) = \tfrac{1}{2}\mathcal{Q}_3({\rm sym}(F)) + \Phi(F)\,,$$
where $\Phi\colon\R^{3 \times 3}\to  \R $ is a function  satisfying 
\begin{align}\label{PPPhi}
\sup \big\{ \tfrac{|\Phi(F)|}{|F|^2} \colon \, |F| \le \sigma \big\} \to 0 \quad \text{ as $\sigma \to 0$.}
\end{align} 
Then,  together with the definition of $G_h$ in  \eqref{eq: mora-liminfy-eq}(i) this gives
\begin{align}\label{liminf_first_inequality}
\liminf_{h \to 0} \Big(h^{-2}\int_{\Omega\setminus \overline{V_h}} W(\nabla_hy_h)\, {\rm d}x\Big) & \ge \liminf_{h\to 0}  \Big(h^{-2}\int_{\Omega_\rho^\delta} \chi_{\Theta_h} W({\rm Id} + hG_h) \, {\rm d}x\Big) \nonumber
\\
&\ge   \liminf_{h\to 0} \int_{\Omega_\rho^\delta} \chi_{\Theta_h}\Big( \tfrac{1}{2}\Q_3(\mathrm{sym}(G_h)) +h^{-2} \Phi(hG_h)  \Big) \, {\rm d}x \nonumber
\\
&= \liminf_{h\to 0} \frac{1}{2}\int_{\Omega_\rho^\delta} \chi_{\Theta_h} \Q_3(\mathrm{sym}( G_h)) \,{\rm d}x\,.
\end{align}
 Here, in the last step we used that 
$${\limsup_{h\to 0}  \int_{\Omega_\rho^\delta} \chi_{\Theta_h} h^{-2} |\Phi(hG_h)|  \, {\rm d}x \le \limsup_{h \to 0}\left( \sup\big\{ \tfrac{|\Phi(hG_h)|}{|h G_h|^2} \colon \, |hG_h| \le  h\lambda_h  \big\}   \int_{\Omega_\rho^\delta} \chi_{\Theta_h} |G_h|^2 \, {\rm d}x\right)= 0\,,    }$$
which follows from the fact that $(G_h)_h$ is bounded in $L^2(\Omega_\rho^\delta;\R^{3 \times 3})$,  \eqref{set_of_big_gradient}, \eqref{PPPhi}, and   $h\lambda_h\to 0$ (see  \eqref{lambda_h_sequence}).    Hence,   \eqref{eq: mora-liminfy-eq}(i),(iii), the fact that $\chi_{\Theta_h}\to 1$ boundedly in measure in $\Omega_\rho^\delta$, see \eqref{eq: cheby}, and the convexity of $\Q_3$ imply that 
\begin{equation}\label{liminf_second_inequality}
\liminf_{h\to 0} \frac{1}{2}\int_{\Omega_\rho^\delta} \chi_{\Theta_h}\Q_3(\mathrm{sym}( G_h)) \,{\rm d}x\geq \frac{1}{2}\int_{\Omega_\rho^\delta} \Q_3(  G)\,{\rm d}x\geq \frac{1}{2}(1-\rho)^4\int_{\pi_1(\Omega_\rho^\delta)}\Q_2\big(R^T R_{,1}\big)\,\mathrm{d}x_1\,,\\
\end{equation}
where $\pi_1$ is the projection onto the $x_1$-axis, and $\Omega^\delta_\rho$ is defined in \eqref{Omega_h_local} and \eqref{Omega_delta_set}. As $\mathcal{L}^3(\Omega \setminus \Omega_\rho^\delta) \to 0$ for $\rho,\delta \to 0$, we also get that  $\mathcal{L}^1((0,L) \setminus \pi_1(\Omega_\rho^\delta)) \to 0$ as $\rho,\delta \to 0$. Thus, \eqref{liminf_first_inequality}, \eqref{liminf_second_inequality}, and monotone convergence yield the lower bound \eqref{lower_bound}.  Note that the last integral can also be taken on $(0,L) \setminus I$ only since $\Q_2(0)=0$ and $R = {\rm Id}$ on $I$, see \eqref{limiting_admissible_pairs}, \eqref{SBV_2_isom} and \eqref{def_Q_2}.  This concludes the proof. 
\end{proof}

We now proceed with the lower bound for the surface part of the energy, namely  
\begin{equation}\label{surface_part}
\mathcal{E}^h_{\mathrm{surf}}(V_h):= \mathcal{E}^h(y_h,V_h)-h^{-2}\int_{\Omega\setminus \overline{V_h}}W(\nabla_hy_h) \, \mathrm{d}x=h^{-2}\mathcal{G}^{\kappa_h}_{\mathrm{surf}}(E_h; \Omega_h)\,,
\end{equation}
where we refer to the definitions in  \eqref{rescaled_energy} and  \eqref{F_surf_energy}. Our approach deviates significantly from the proof of lower bounds in relaxation results for energies defined on pairs of functions and sets, cf.\   \cite{BraChaSol07} or \cite{CrismaleFriedrich}. This is mainly due to the fact that the nonlinear geometric rigidity result allows us to control the elastic energy only in a large part of $\Omega \setminus \overline{V_h}$. Our argument to derive the lower bound for the surface energy term related to collapsing voids correctly hinges  on Proposition \ref{prop: 2nd main} along with an argument by contradiction. We again suppose that $I$ is the representative consisting of  a finite union of open  intervals.

\begin{lemma}\label{surface_lower_bound}
Suppose that  $(y_{h}, V_{h})\overset{\tau}{\longrightarrow}((y\vert \, d_2\vert \, d_3), I)$ for some  $((y\vert \, d_2\vert \, d_3),I)\in \mathcal{A}$.  Then,
\begin{equation}\label{surface_part_lower_bound}
\liminf_{h\to 0} \E^h_{\mathrm{surf}}(V_h)\geq \H^0(\partial I\cap (0,L))+2\H^0\big((J_{ y}\cup J_{R})\setminus \partial I\big)\,.
\end{equation} 
\end{lemma}
\begin{proof}

Let $(E_h)_{h>0}$ be the void sets associated to  $(V_h)_{h>0}$ according to \eqref{order_1_domain}. By $(E^*_h)_{h>0}$  we denote the open sets given by Proposition \ref{prop: 2nd main} satisfying $E_{h} \subset E^*_h \subset \Omega_h$ and   \eqref{eq: void-newXXX}. We  also  introduce the rescaled sets 
\begin{equation}\label{enlarged_voids_V_h}
V_h^*:=T_{1/h}(E_h^*)\,,
\end{equation} 
and note by \eqref{eq: void-newXXX}, a scaling argument,  and \eqref{compactness_properties}(i) that
\begin{equation}\label{enlarged_voids_V_h2}
\chi_{V^*_{h}}\longrightarrow \chi_{V_I}\ \text{ in } L^1(\Omega)\,.
\end{equation} 
By \eqref{eq: void-newXXX}  and \eqref{surface_part} we have
\begin{equation}\label{1st_lower_bound_surface}
\liminf_{h \to 0}\E^h_{\mathrm{surf}}(V_h) = \liminf_{h \to 0} h^{-2}\mathcal{G}^{\kappa_h}_{\mathrm{surf}}(E_h; \Omega_h) \geq \liminf_{h \to 0}h^{-2}\H^{2}(\partial E_h^*\cap \Omega_h)\,.
\end{equation}
Since $((y\vert \, d_2\vert \, d_3),I)\in \mathcal{A}$, there exist finitely many $(x_j)_{j=1}^n \subset (0,L)$ such that $$\{x_1,\dots,x_n\}
= (\partial I \cap (0,L)) \cup J_y \cup J_R\,.$$  We choose  $\delta>0$ sufficiently small such that the sets $S^{2\delta}_{h}(x_j)$, $j=1,\ldots,n$, are pairwise disjoint and contained in $\Omega_h$,   cf.\ \eqref{eq: D not}.  Our goal is to prove
\begin{align}\label{eq: maon lsc}
\begin{split}
{\rm (i)} \quad &  \liminf_{h \to 0 }   h^{-2} \H^{2}\big(\partial E_h^*\cap S^{2\delta}_{h}(x_j)\big) \ge 1 \quad \text{ if $x_j \in \partial I \cap (0,L)$}\,,\\
{\rm (ii)} \quad &  \liminf_{h \to 0 }   h^{-2} \H^{2}\big(\partial E_h^*\cap S^{2\delta}_{h}(x_j)\big) \ge 2 \quad \text{ if $x_j \in J_y \setminus \partial I$}\,,\\
{\rm (iii)} \quad &  \liminf_{h \to 0 }   h^{-2} \H^{2}\big(\partial E_h^*\cap S^{2\delta}_{h}(x_j)\big) \ge 2 \quad \text{ if $x_j \in J_R \setminus \partial I$}\,.
\end{split}
\end{align}
Once \eqref{eq: maon lsc} is shown, we can conclude as follows. By \eqref{1st_lower_bound_surface} and the fact that the sets $S^{2\delta}_{h}(x_j) \subset \Omega_h$, $j=1,\ldots,n$, are pairwise disjoint, we get
\begin{align*}
\liminf_{h \to 0}\E^h_{\mathrm{surf}}(V_h) & \ge  \liminf_{h \to 0} \sum_{j=1}^n  h^{-2}\H^{2}(\partial E_h^*\cap S^{2\delta}_{h}(x_j)) \ge \H^0(\partial I \cap (0,L)) + 2  \H^0\big( (J_y \cup J_R) \setminus \partial I \big)\,. 
\end{align*}
This shows \eqref{surface_part_lower_bound}. 

We now proceed with the proof of the properties stated in \eqref{eq: maon lsc}.  We start with (i). By a change of variables and the definition in \eqref{enlarged_voids_V_h} we find  
\begin{align*} 
\liminf_{h \to 0 }   h^{-2} \H^{2}\big(\partial E_h^*\cap S^{2\delta}_{h}(x_j)\big) & = \liminf_{h \to 0}\int_{\partial V^*_h\cap S^{2\delta}_{1}(x_j)}|(\nu^1_{V^*_h},h^{-1}\nu^2_{V^*_h},h^{-1}\nu^3_{V^*_h})| \, {\rm d}\mathcal{H}^2\,
\end{align*}
where we use the notation $\nu_{V^*_h}=(\nu^1_{V^*_h},\nu^2_{V^*_h}, \nu^3_{V^*_h})\in \S^2$ for the outer unit normal  to $V_h^*$.  By \eqref{enlarged_voids_V_h2} and the lower semicontinuity of the perimeter (cf.\ \cite[Proposition 3.38]{Ambrosio-Fusco-Pallara:2000}),
we get
$$ {\liminf_{h \to 0 }   h^{-2} \H^{2}\big(\partial E_h^*\cap  S^{2\delta}_{h}(x_j)\big) \ge \liminf_{h \to 0 }    \H^{2}\big(\partial V_h^*\cap S^{2\delta}_{1}(x_j)\big) \ge  \mathcal{H}^2\big( \partial V_I\cap S^\delta_{1}(x_j)\big)\,. } $$
The fact that $\lbrace x_j \rbrace \times (-\frac{1}{2},\frac{1}{2}) \subset \partial V_I$ yields (i).

We proceed to show  (ii). Suppose the statement was wrong, i.e., there exists $0<\mu <1$ and a subsequence (not relabeled) such that $2 \mu \ge    h^{-2} \H^{2}(\partial E_h^*\cap S^{2\delta}_{h}(x_j))$ for all $h>0$. Choose $\rho >0$ small enough such that $ \frac{\mu}{(1-\rho)^2} <1$. Then, we get
$$ \H^{2}\big(\partial E_h^*\cap S^{2\delta}_{h}(x_j)\big)   \le 2\mu h^2   < 2   ((1-\rho)h)^2 \,. $$
This implies that \eqref{eq: small jump/vol}(i) (for $l=\delta$) holds. The fact that the sets $S^{2\delta}_{1}(x_j)$, $j=1,\ldots,n$, are pairwise disjoint implies that $S^{2\delta}_{1}(x_j) \cap V_I = \emptyset$. Thus, by \eqref{enlarged_voids_V_h2} we get $\lim_{h \to \infty} \mathcal{L}^3(V_h^* \cap S^{2\delta}_{1}(x_j)) = 0$. By the definition of $V_h^*$ and a change of variables we find
$$ \frac{\mathcal{L}^3(E_h^* \cap S^{2\delta}_{h}(x_{j}))}{\mathcal{L}^3(S^{2\delta}_{h}(x_{j}))}  \le \frac{1}{9} $$
for $h>0$ sufficiently small, i.e., \eqref{eq: small jump/vol}(ii) is satisfied. Let $(w_h)_{h>0}$ be the sequence from  Proposition~\ref{Lipschitz_replacement} and let again    $(\tilde w_h)_{h>0}\subset SBV^2(\Omega_{1,\rho};\R^3)$  be the rescaled sequence, see \eqref{eq: rescaliii}. Thus, by a change of variables and by \eqref{eq: small jump}, we find that
\begin{align*}
\int_{S_{1,\rho}^{\delta}(x_j)    \cap  J_{\tilde{w}_h}   }  \sqrt{|[\tilde{w}_h]|}  \, {\rm d} \mathcal{H}^2 &  \le \int_{S_{1,\rho}^{\delta}(x_j)\cap  J_{\tilde{w}_h}}  \sqrt{|[\tilde{w}_h]|} \  |(\nu^1_{\tilde{w}_h},h^{-1}\nu^2_{\tilde{w}_h},h^{-1}\nu^3_{\tilde{w}_h})| \, {\rm d} \mathcal{H}^2  \\
& = \frac{1}{h^2}\int_{S_{h,\rho}^{\delta}(x_j)   \cap  J_{{w}_h}   }  \sqrt{|[{w}_h]|}   \, {\rm d} \mathcal{H}^2  \to 0 \,.
\end{align*}
By \eqref{almost_the same_and_control_of_energy}(ii)--(iv), \eqref{from_v_to_y},  \eqref{compactness_properties}(ii), and\eqref{eq: rescaliii} we get $\tilde{w}_h \to \bar{y}$ in $L^1(\NNN\Omega_{1,\rho}\EEE;\R^3)$ and
\begin{align*}
\sup_{h>0}\Big(\int_{\Omega_{1,\rho}} |\nabla \tilde{w}_h|^2  {\rm d}x + \mathcal{H}^2( J_{\tilde{w}_h})\Big) \le C\,, 
\end{align*}
for a constant $C>0$ independent of $h>0$.  By Ambrosio's lower semicontinuity  theorem in $SBV$ (cf.\ \cite[Theorem 4.7]{Ambrosio-Fusco-Pallara:2000})   and the fact that $\tilde{w}_h \to \bar y$ in $L^1(\Omega;\R^3)$ we get 
$$ \int_{S_{1,\rho}^{\delta}(x_j)   \cap  J_{\bar y}   }  \sqrt{|[\bar y]|} \, {\rm d} \mathcal{H}^2 \le \liminf_{h \to 0} \int_{S_{1,\rho}^{\delta}(x_j) \cap  J_{\tilde{w}_h}   }  \sqrt{|[\tilde{w}_h]|} \, {\rm d} \mathcal{H}^2 = 0\,.  $$
This shows that $J_{\bar y}$ does not jump on $(\lbrace x_j \rbrace \times \R^2) \cap \Omega_{1,\rho}$ which contradicts the fact that $x_j \in J_y$.  

For (iii) we proceed in a similar fashion and first get that  \eqref{eq: small jump/vol} is satisfied. We let $(R_h)_{h>0}$ be the sequence in Proposition \ref{prop: 2nd main}  and introduce the rescaled sequence $(\tilde R_h)_{h>0}\subset SBV^2(\Omega_{1,\rho};\R^{ 3\times 3\EEE})$ by
\begin{equation*}
\tilde R_h(x):=\tilde R_h(x_1,x_2,x_3):=R_h(x_1,hx_2,hx_3)\,.
\end{equation*}
By a  change of variables,  the properties  \eqref{almost_the same_and_control_of_energy}(ii)--(iv), as well as  \eqref{from_v_to_y} and \eqref{rescaled_deformation_gradient} we get
$${\sup_{h>0}\Big(\int_{\Omega_{1,\rho}} |\nabla \tilde{R}_h|^2  {\rm d}x + \mathcal{H}^2( J_{\tilde{R}_h})\Big) \le C,  \quad \quad    |\nabla_h y_h - \tilde{R}_h |  \to 0 \text{ in measure on $\Omega_{1,\rho}$}\,,} $$
for a constant $C>0$ again independent of $h>0$. Then, again by Ambrosio's lower semicontinuity  theorem,  \eqref{eq: small jump},  \eqref{compactness_properties}(iii), and the fact that $S_{1,\rho}^{\delta}(x_j)  \cap V_I = \emptyset$, we derive  
$$ \int_{ S_{1,\rho}^{\delta}(x_j)   \cap  J_{\bar R}   }  \sqrt{|[\bar R]|} \, {\rm d} \mathcal{H}^2 \le \liminf_{h \to 0} \int_{ S_{1,\rho}^{\delta}(x_j) \cap  J_{\tilde{R}_h}   }  \sqrt{|[\tilde{R}_h]|} \, {\rm d} \mathcal{H}^2 = 0\,.  $$
As in (ii), this yields a contradiction, and the proof of (iii) is concluded.
\end{proof}

\section{The $\Gamma$-limsup inequality}\label{gamma_limsup}

In this last section, we construct recovery sequences for admissible limits $((y\vert \, d_2\vert \, d_3), I)\in\mathcal{A}$. We start by recalling the relevant result for elastic rods, using again the convention in \eqref{eq: convention1}--\eqref{eq: convention2}. \EEE 

\begin{lemma}[Recovery sequences in the Sobolev setting]\label{lemma: limsup Mora}
Let $\Omega_{\ell} := (0,\ell) \times (-\frac{1}{2},\frac{1}{2})^2$ for $\ell >0$. Let  $((y\vert \, d_2\vert \, d_3),I)\in \mathcal{A}$ be such that $\bar y|_{\Omega_{\ell}} \in W^{2,2}(\Omega_{\ell};\R^3)$ and $\bar d|_{\Omega_{\ell}}, \bar  d_3|_{\Omega_{\ell}} \in W^{1,2}(\Omega_{\ell};\R^3)$. Then, there exists a sequence $(y_{ h\EEE})_{h >0} \subset W^{1,2}(\Omega_{\ell};\R^3)$ such that
\begin{align}\label{eq: recov1}
y_h \to \bar y\ \ \text{strongly in }   W^{1,2}(\Omega_{\ell};\R^3), \quad \quad   \nabla_h y_{h}  \to (\bar y_{,1}\vert \, \bar d_2\vert \, \bar d_3\EEE) \ \text{strongly  in }  L^2(\Omega_{\ell};\R^{3 \times 3})\,,
\end{align}
and we have
\begin{align}\label{eq: recov2}
\lim_{h \to 0} \frac{1}{h^2} \int_{\Omega_{\ell} } W(\nabla_h y_h) \, {\rm d}x =  \frac{1}{2}\int_{(0,\ell)}\Q_2({R}^T{R}_{,1})\,\mathrm{d}x_1 \,, \end{align}
where $R := (y_{,1}\vert \, d_2\vert \, d_3)$. Moreover, if $ \bar y \EEE \in L^\infty(\Omega_{\ell};\R^3)$, it holds that $\limsup_{h \to \infty} \Vert y_h \Vert_\infty \le \Vert  \bar y \EEE \Vert_\infty$.
\end{lemma}
  
For the  proof we refer to \cite[Theorem 3.1(ii)]{Mora}. We now proceed with the construction of recovery sequences.

\begin{proof}[Proof of Theorem \ref{main_gamma_convergence_thm}$\rm(ii)$]
Consider an admissible limit $((y\vert \, d_2\vert \, d_3), I)\in \mathcal{A}$, see \eqref{limiting_admissible_pairs}--\eqref{SBV_2_isom}.   We will combine ideas from  \cite[Section 3]{Mora} and  \cite[Subsection 5.4]{schmidt2017griffith}. We first treat the case $\Vert y \Vert_\infty < M$ and address the changes for $\Vert y \Vert_\infty = M$ at the end of the proof.  By choosing a suitable representative, we can assume that $I$ is the union of   finitely many open subintervals of $(0,L)$. We also denote
\begin{equation}\label{points_of_jump}
(J_{ y}\cup J_{ R})\setminus \partial I:=\{t_1,\dots,t_m\} \subset (0,L) \,,
\end{equation}
where $0<t_1<\dots<t_m<L$. Denote by $(J_i)_i^n$  the connected components of $(0,L)\setminus(I\cup \{t_1,\dots,t_m\})$.

We apply  Lemma \ref{lemma: limsup Mora} on each connected component $J_i$ to find recovery sequences $y_h^i \in W^{1,2}( \tilde{J}_i ; \R^3   )$, where $\tilde{J}_i:= J_i \times (-\frac{1}{2},\frac{1}{2})^2$ such that \eqref{eq: recov1}--\eqref{eq: recov2} are satisfied for the respective functions on the respective sets. For $h>0$ sufficiently small, consider the sets $(V_h)_{h>0}\subset \mathcal{A}_{\rm{reg}}(\Omega)$ defined by
\begin{equation}\label{recovery_sequence_for_voids}
V_h:=\Big(I\cup\bigcup_{i=1}^m(t_i-h,t_i+h)\Big)\times (-\tfrac{1}{2},\tfrac{1}{2})^2\,.
\end{equation}
Recalling \eqref{anisotropic_dilation}, we introduce the deformations $(y_h)_{h>0} \subset SBV^2(\Omega;\R^3)$ defined  by
\begin{equation}\label{recovery_sequence_for_deformations}
y_h(x):= \begin{cases}    y_h^i(x) & \text{ if } x \in \tilde{J}_i \setminus V_h, \\
  T_h(\mathrm{id})     & \text{ if } x \in  V_h\,. 
 \end{cases}
\end{equation}
Since $\Vert y \Vert_\infty < M$, Lemma \ref{lemma: limsup Mora} also implies that $\Vert y_h \Vert_\infty \le M$ for $h>0$ sufficiently small.  This shows that   $(y_h,V_h)_{h>0}\subset\hat{\mathcal{A}}_h$, cf.\ \eqref{admissible_configurations_h_level}. Clearly, in view of \eqref{recovery_sequence_for_voids},  we have $\chi_{V_h} \to \chi_{V_I}$ in $L^1(\Omega)$. Moreover, \eqref{eq: recov1} and  \eqref{recovery_sequence_for_deformations} show that $(y_h)_{h>0}$ also satisfies \eqref{compactness_properties}(ii),(iii).   Thus, by the definition of $\tau$-convergence in Definition \ref{type_of_convergence} we have $(y_{h},V_{h})\overset{\tau}{\longrightarrow}((y\vert \, d_2\vert \, d_3),I)$  as $h\to 0$.

Regarding the elastic part of the energy, from \eqref{eq: recov2} and   \eqref{recovery_sequence_for_deformations} we directly infer
\begin{align}\label{lim_sup_surf_for_smooth-before}
\limsup_{h \to 0} \frac{1}{h^2} \int_{\Omega \setminus \overline{V_h}} W(\nabla_h y_h) \, {\rm d}x  &\le  \lim_{h \to 0} \frac{1}{h^2} \sum_{i=1}^{n} \int_{\tilde{J}_i} W(\nabla_h y_h) \, {\rm d}x  \notag\\ 
& =   \frac{1}{2}  \sum_{i=1}^{n} \int_{J_i}\Q_2({R}^T{R}_{,1})\,\mathrm{d}x_1  =  \frac{1}{2}   \int_{(0,L) \setminus I}\Q_2({R}^T{R}_{,1})\,\mathrm{d}x_1 \,. 
\end{align}

We now address the surface part of the energy introduced in \eqref{surface_part}. First, we  set   $E_h:=T_h(V_h)$, where $(V_h)_{h>0}$ are defined in \eqref{recovery_sequence_for_voids}. By  \eqref{surface_part}, \eqref{F_surf_energy}, and the fact that $\partial E_h\cap \Omega_h$ consists of planar interfaces with unit normal  $\pm e_1$,  we have that
\begin{align}\label{lim_sup_surf_for_smooth}
\begin{split}
\lim_{h\to 0}\E_{\mathrm{surf}}^h(V_h)&=\lim_{h\to 0} h^{-2}\Big(\H^2(\partial E_h\cap \Omega_h)+\kappa_h\int_{\partial E_h\cap \Omega_h}|\bm A_h|^{2}\,\mathrm{d}\mathcal{H}^2\Big)\\
&=\lim_{h\to 0}h^{-2}\H^2\Big(\partial\Big(\Big(I\cup\bigcup_{i=1}^m(t_i-h,t_i+h)\Big)\times (-\tfrac{h}{2},\tfrac{h}{2})^2\Big)\cap \Omega_h \Big)\\&=\H^0(\partial I\cap (0,L))+2m= \H^0(\partial I\cap (0,L))+2\H^0((J_{y}\cup J_{R})\setminus \partial I)\,,
\end{split}
\end{align}
where the last step follows from \eqref{points_of_jump}. Now, \eqref{lim_sup_surf_for_smooth-before} and \eqref{lim_sup_surf_for_smooth} show \eqref{gamma_limsup_inequality} in the case $\|y\|_{L^\infty}<M$.

We conclude the proof by addressing the case $\Vert y \Vert_\infty = M$. In this case, we extend  $y, d_2, d_3$ on $(L,L+1)$ such that $y_{,1}, d_2, d_3$ are constant on  $[L,L+1)$ and  $({y}\vert \,  {d_2}\vert \, {d_3}) \in SBV^2_{\mathrm{isom}}(0,L+1)$. For $0<\sigma<1$, we  consider the functions $y^\sigma(x_1) := \sigma y(x_1/\sigma)$, $d^\sigma_2(x_1):= d_2(x_1 /\sigma)$, and $d^\sigma_3(x_1):= d_3(x_1 /\sigma)$  on $(0,\sigma (L+1))$. Now, $\Vert y^\sigma \Vert_\infty < M$ and we can construct a recovery sequence as above. Moreover, one can check that $\lim_{\sigma \to 0}\E^0((y^\sigma\vert \, d^\sigma_2\vert \, d^\sigma_3), \sigma I) = \E^0((y\vert \, d_2\vert \, d_3), I)$. Thus, the conclusion follows by  a standard diagonal sequence argument in the theory of $\Gamma$-convergence.  
\end{proof}

\section*{Acknowledgements} 
This work was supported by the DFG project FR 4083/3-1 and by the Deutsche Forschungsgemeinschaft (DFG, German Research Foundation) under Germany's Excellence Strategy EXC 2044 -390685587, Mathematics M\"unster: Dynamics--Geometry--Structure. \NNN The research of LK was supported by the DFG through the Emmy Noether Programme (project number 509436910). The authors would like to thank the anonymous referees for their valuable comments and suggestions for the final version of this manuscript. \EEE

\typeout{References}

 \end{document}